\documentclass[reqno]{amsart}

\usepackage{amssymb, amsmath,mathtools}
\usepackage{bbm}
\usepackage{mathrsfs}
\usepackage{amscd}
\usepackage{verbatim}
\usepackage{stmaryrd}
\usepackage[dvipsnames]{xcolor}
\usepackage{mathtools}
\mathtoolsset{showonlyrefs}

\usepackage{a4wide}

\usepackage{enumerate}

\usepackage[colorlinks,linkcolor={blue},citecolor={blue},urlcolor={purple},]{hyperref}

\usepackage[colorinlistoftodos,prependcaption,textsize=tiny]{todonotes}

\usepackage{float}

\usepackage{tabularx}
\newcolumntype{L}{>{\arraybackslash}X}
\usepackage{multirow}

\theoremstyle{plain}
\newtheorem{theorem}{Theorem}[section]
\theoremstyle{remark}
\newtheorem{remark}[theorem]{Remark}
\newtheorem{example}[theorem]{Example}
\theoremstyle{plain}
\newtheorem{corollary}[theorem]{Corollary}
\newtheorem{lemma}[theorem]{Lemma}
\newtheorem{proposition}[theorem]{Proposition}
\newtheorem{definition}[theorem]{Definition}

\newtheorem{assumption}[theorem]{Assumption}

\numberwithin{equation}{section}

\DeclareMathOperator{\loc}{loc}

\newcommand\bN{\mathbb{N}}

\newcommand\bR{\mathbb{R}}

\renewcommand\P{\mathbb{P}}

\newcommand\cF{\mathscr{F}}

\newcommand\calL{\mathcal{L}}

\newcommand\cO{\mathcal{O}}

\newcommand\cS{\mathcal{S}}
\newcommand\cU{\mathcal{U}}

\newcommand{\E}{\mathbb{E}}
\newcommand{\1}{\mathbbm{1}}
\newcommand{\one}{\mathbbm{1}}

\newcommand{\vp}{\varphi}

\newcommand{\wt}{\widetilde}

\DeclareMathOperator{\dv}{div}

\newcommand{\lb}{\langle}

\newcommand{\rb}{\rangle}

\newcommand{\vertiii}[1]{{\left\vert\kern-0.25ex\left\vert\kern-0.25ex\left\vert #1
\right\vert\kern-0.25ex\right\vert\kern-0.25ex\right\vert}}

\newcommand{\diff}{\,\mathrm{d}}

\newcommand{\Dom}{\mathcal{O}}
\newcommand{\R}{\mathbb{R}}

\def\N{{\mathbb N}}

\def\R{{\mathbb R}}

\newcommand{\g}{\gamma}

\renewcommand{\O}{\Omega}

\renewcommand{\a}{\kappa}

\newcommand{\D}{\mathscr{D}}

\newcommand{\HD}{\prescript{}{D}{H}}
\renewcommand{\CD}{\prescript{}{D}{C}}

\newcommand{\BD}{\prescript{}{D}{B}}

\renewcommand{\emptyset}{\varnothing}

\newcommand{\Progress}{\mathscr{P}}

\newcommand{\Borel}{\mathscr{B}}

\def\XXint#1#2#3{{\setbox0=\hbox{$#1{#2#3}{\int}$ }
\vcenter{\hbox{$#2#3$ }}\kern-.6\wd0}}

\newcommand{\ellip}{\nu}

\newcommand{\dd}{\mathrm{d}}
\newcommand{\Sf}{\mathcal{S}_\varphi}

\allowdisplaybreaks

\makeatletter
\renewcommand{\l@section}{\@tocline{1}{0pt}{1.5em}{}{}}
\renewcommand{\l@subsection}{\@tocline{2}{0em}{2.3em}{}{}}
\makeatother

\begin{document}

\author{Antonio Agresti}
\address{Department of Mathematics Guido Castelnuovo\\ Sapienza University of Rome\\ P.le
A. Moro 5\\ 00100 Roma\\ Italy.} \email{antonio.agresti@uniroma1.it}

\author{Fabian Germ}
\address{Delft Institute of Applied Mathematics\\
Delft University of Technology \\ P.O. Box 5031\\ 2600 GA Delft\\The
Netherlands.}
\email{f.germ@tudelft.nl}

\author{Mark Veraar}
\address{Delft Institute of Applied Mathematics\\
Delft University of Technology \\ P.O. Box 5031\\ 2600 GA Delft\\The
Netherlands.} \email{M.C.Veraar@tudelft.nl}

\thanks{The first author is a member of GNAMPA (INdAM). The second and third authors have received funding from the VICI subsidy VI.C.212.027 of the Netherlands Organisation for Scientific Research (NWO)}

\date\today

\title[reaction-diffusion equations driven by non-trace-class noise]{An optimal local theory for reaction-diffusion equations driven by non-trace-class noise}

\keywords{stochastic reaction-diffusion equations, non-trace-class noise, multiplicative noise, critical spaces, rough initial data, maximal $L^p$-regularity, regularization, blow-up criteria, positivity}

\subjclass[2020]{Primary 35K57; Secondary 35A01, 35B44, 35B65, 35K90, 35R60, 47D06, 60H15}

\begin{abstract}
We study local well-posedness for a class of stochastic reaction-diffusion equations driven by multiplicative, possibly colored, noise. The interaction between rough stochastic forcing and polynomial nonlinearities naturally leads to solutions with low spatial regularity, making the treatment of the nonlinear terms delicate. Our main contribution is a general local existence and uniqueness theory for SPDEs with rough noise and highly irregular initial data.
The framework also yields new results in standard noise regimes, including trace-class noise and space-time white noise. We identify the critical initial-data spaces for a wide range of nonlinearities, and we establish instantaneous parabolic regularization, general blow-up criteria, and sufficient conditions for positivity preservation. We apply the abstract theory to several prototypical models, including the stochastic Allen-Cahn, Burgers, Fisher-KPP, and coupled Gray-Scott equations. Finally, in the one-dimensional space-time white-noise setting, we combine our local theory with existing global a priori results in a highly singular regime.
\end{abstract}

\maketitle

\vspace{-1cm}
\tableofcontents

\section{Introduction}
\label{s:intro}

Reaction-diffusion equations play a central role in the mathematical modeling of chemical reactions, population dynamics, phase separation, ecology, and pattern formation. 
Classical examples include the Allen--Cahn equation, the Fisher--KPP equation, the Lotka--Volterra system, and the Gray--Scott and Brusselator models. 
In the deterministic setting, the theory of reaction-diffusion equations is highly developed; see, for instance, \cite{CGV19,FMT20,K90,P10_survey,PSY19,R84_global} and the references therein. 
Depending on the structure of the nonlinearities, one has a rich body of results on local and global well-posedness, regularity, positivity, and asymptotic behavior. 
At the same time, even in the deterministic case, obtaining global well-posedness is delicate in general, and further structural assumptions are often indispensable.

Stochastic perturbations arise naturally in reaction-diffusion models when one wishes to account for uncertainty in parameters, unresolved microscopic interactions, or random forcing acting on the system. 
Stochastic reaction-diffusion equations have therefore received substantial attention in recent years; see, for instance, \cite{C03,Cer05,CR05,cerrai2025nonlinear, DKZ19,KN19,KvN12,S21,S21_dissipative,S21_superlinear_no_sign} and the references therein. 
However, compared to the deterministic literature, the stochastic theory is still much less complete, especially for systems and for equations with rough multiplicative noise. 
A large part of the available theory relies on strong coercivity or dissipativity assumptions, and many results concern equations or settings in which the noise or initial data are comparatively regular. 
In particular, for weakly dissipative systems of reaction-diffusion type, many natural questions remain open. In the recent work \cite{milesis2026global}, progress was made for weakly coercive
triangular systems driven by non-trace-class noise. 

Previous work on local and global well-posedness \cite{agresti2023reaction, agresti2024reaction} by the first and third authors treated several weakly dissipative stochastic reaction-diffusion systems with {\em trace-class noise} and {\em transport noise}, including Lotka--Volterra, Brusselator, and Gray--Scott type models. 
This showed that one can go substantially beyond the fully coercive setting in concrete examples. 
Nevertheless, many important problems remain open. 
A major obstacle is the lack of a sufficiently flexible local theory capable of handling rough multiplicative noise and highly irregular initial data.
Thus, such a local theory is not only of independent interest, but also forms the natural starting point for future developments on global well-posedness and, subsequently, on more refined dynamical questions. 

The purpose of the present paper is to develop such a local theory in the case of possibly non-trace-class noise. 
More precisely, we study stochastic reaction-diffusion equations of the form
\begin{equation}
\label{eq:intro_equation}
\left\{
\begin{aligned}
\dd u_i - \dv(a_i\cdot \nabla u_i)\,\dd t
    &= \Big[\dv(F_i(\cdot,u))+f_i(\cdot,u)\Big]\,\dd t + \sum_{j=1}^\ell g_{i}^j(\cdot,u)\,\dd W^j(t),\\
u_i(0)&=u_{0,i},
\end{aligned}
\right.
\end{equation}
on a bounded domain $\Dom\subset \R^d$, for $i\in\{1,\dots,\ell\}$, together with the Dirichlet boundary condition
\begin{align*}
u|_{\partial \cO}=0
\end{align*}
Our results also have versions for Neumann boundary conditions, but we will focus on the Dirichlet case for clarity. 

In the above, we assume without loss of generality that the dimension of the system is equal to the number of noise terms, which can be achieved by adding trivial rows or columns. Here $u=(u_i)_{i=1}^\ell$ is the unknown process, the nonlinearities may have polynomial growth, and the independent driving Brownian noises $(W^j)_{j=1}^\ell$ are allowed to be multiplicative and non-trace-class.

A central difficulty is that the combination of rough stochastic forcing and superlinear nonlinearities naturally leads to solution spaces of low spatial regularity.
In this context, the term \emph{critical} refers to the borderline regularity suggested by the natural scaling of the equation.
At this level, the linear part, the nonlinear drift terms, and the noise have the same strength under parabolic rescaling; below this level one should not expect a general well-posedness theory based on the same methods.
Thus, the critical spaces describe the largest natural classes of initial data for which the equation is still expected to be locally well posed.
As a consequence, the nonlinear terms must be treated at the level of distributions, and the identification of the correct critical initial-data spaces becomes a delicate issue.

The main novelty of the present work lies in the combination of two recent ingredients. The improved abstract theory from \cite{AV-ob} provides local well-posedness in critical spaces once the deterministic and stochastic nonlinearities satisfy suitable mapping estimates. 

The estimates of \cite{AGV} provide sharp
$\gamma$-radonifying bounds for multiplication operators composed with the noise operators $R^j$. Combined with stochastic maximal regularity, these yield sharp estimates for stochastic convolutions driven by non-trace-class noise.
Together, these tools allow us to treat rougher stochastic perturbations than in \cite{agresti2023reaction}, and to remove several technical restrictions on the noise which commonly appear in the literature. 
In particular, even in situations involving trace-class noise as considered in \cite{agresti2023reaction}, the improved abstract framework now enables us to treat more general classes of equations, and, in particular, identify larger critical classes of initial data.

Our main result is a local existence and uniqueness theorem for
stochastic reaction-diffusion equations with non-trace-class multiplicative
noise and highly irregular initial data. The theorem is formulated in a scale
of Sobolev--Besov spaces and identifies the scaling-critical initial-data
spaces for broad classes of polynomial nonlinearities. 

In addition to local well-posedness, we establish instantaneous
parabolic regularization in both the subcritical and critical settings. For positive times, solutions become continuous in space and time. Moreover, solutions constructed with different admissible integrability parameters coincide whenever their initial data coincide. These properties allow us to restart the equation at positive times in a different parameter setting and to derive blow-up criteria in low-regularity norms. Combining regularization with approximation and comparison arguments, we also prove positivity preservation under natural
structural assumptions.

The coefficients of the leading elliptic operator are allowed to be merely measurable in time and in \(\omega\). In a one-dimensional space-time white-noise setting, we show how the local theory can be combined with recent results from \cite{DKZ19,FKN25} to obtain global well-posedness in a highly singular regime. This method is based on our instantaneous regularization results, blow-up criteria, and a suitable restarting procedure.  
The same method can be used to extend global well-posedness results to the case where the initial data have very low (negative) smoothness. In particular, this applies to the results in \cite{C03,KvN12,milesis2026global, S21_dissipative,S21_superlinear_no_sign} as long as one assumes that the local Lipschitz constants have polynomial growth, which holds in many physically relevant scenarios.

\subsection{Scaling and criticality for SPDEs with non-trace-class noise}
\label{ss:scaling_intro}
It is well known in the context of PDEs that, by using scaling arguments, one can identify natural, scaling-critical regularity thresholds for well-posedness, which are often \emph{sharp}. This is also possible for reaction-diffusion equations with non-trace-class noise such as \eqref{eq:intro_equation}. This reveals the maximal order of growth of the nonlinearities and the intrinsic scaling of the SPDE under consideration. Building on this, one can determine the lowest Sobolev regularity/index for which the system of SPDEs \eqref{eq:intro_equation} is well-posed. 
\smallskip

\emph{Scaling invariance.}
We now explain, at a heuristic level, how the critical exponents arise from the different nonlinear terms in \eqref{eq:intro_equation}.
The purpose of the following scalar equation is not to introduce an additional model, but to isolate the three mechanisms present in the general system: a reaction term $f(u)$, a divergence-form term $\operatorname{div}F(u)$, and a multiplicative noise term $g(u) \partial_t W$.
To make the scaling transparent, we replace these nonlinearities by pure powers with independent growth parameters. Thus, we consider
\begin{equation}
\label{eq:intro_scalar_equation_scaling}
\partial_t u - \Delta u  = |u|^{1+\rho_1} + \dv(e |u|^{1+\rho_2})+ |u|^{1+\rho_3}\,\partial_t \mathcal{W}^{(\zeta)}
\end{equation}
where $e\in\mathbb R^d$ is fixed.
Here the first term represents the scaling of $f$, the second that of
$\operatorname{div}F$, and the third that of the multiplicative noise.
The use of monomials is only for the scaling calculation; the results below
apply to general locally Lipschitz nonlinearities with the corresponding
polynomial growth. We assume that the driving noise satisfies
\begin{equation}
\label{eq:scaling_noise}
\partial_t \mathcal{W}^{(\zeta)}(\lambda^2 t,\lambda x) 
\stackrel{{\rm law}}{=}\lambda^{-1-d/2+d/\zeta}\partial_t \mathcal{W}^{(\zeta)}( t, x),
\quad \text{for a $\zeta\in [2,\infty]$.}
\end{equation}
The previous condition suggests that $\mathcal{W}^{(\zeta)}$ lies between a space-time white noise ($\zeta=\infty$), and the white-in-time noise ($\zeta=2$), depending on the value of the coloring $\zeta$. Details are given in Subsection \ref{sss:noise}. A noise satisfying the assumptions of our results and the scaling \eqref{eq:scaling_noise} up to a logarithmic correction is provided in \cite[Subsection 1.1]{AGV}. 

It is well known in the context of PDEs that scaling arguments often reveal
the borderline regularity for well-posedness. For stochastic reaction-diffusion equations of the form \eqref{eq:intro_equation} the same idea can be used to identify which growth rates of the drift and noise
are critical, and which initial-data spaces are naturally associated with them. The calculation below should therefore be understood as a guide to the structure of the assumptions and results that follow.

\smallskip

Suppose that $u$ is a solution to \eqref{eq:intro_scalar_equation_scaling} and define, for $\lambda>0$, the dilated process $u_\lambda$ under the  parabolic scaling:
\begin{equation}
\label{eq:ulambda_formula}
u_{\lambda}(t,x)=\lambda^{\sigma} u(\lambda^2 t, \lambda x).
\end{equation}
Clearly,
\begin{equation}
\label{eq:SPDE_rescaled}
\partial_t u_{\lambda}(t,x) - \Delta u_{\lambda}(t,x) 
= \lambda^{2+\sigma} \big[
\partial_t u( \lambda^{2}t,\lambda x) - \Delta u ( \lambda^{2}t,\lambda x)].
\end{equation}
Next, we analyze the scaling of the nonlinearities in \eqref{eq:intro_scalar_equation_scaling}. First, note that 
\begin{equation*}
 |u_\lambda|^{1+\rho_1} =
\lambda^{\sigma(\rho_1+1)} |u ( \lambda^{2}t,\lambda x)|^{1+\rho_1} 
\end{equation*}
where we see that the above has the same scaling as the linear operator \eqref{eq:SPDE_rescaled} if 
\begin{equation}
\label{eq:scaling_parameter1}
2+\sigma = \sigma(\rho_1 + 1)\quad \text{ or equivalently }\quad
\sigma=2/\rho_1.
\end{equation}
Second, for the divergence-type nonlinearity, we have 
$$
 \dv (e |u_\lambda|^{1+\rho_2} )(t,x)=
\lambda^{1+\sigma(1+\rho_2)} 
\dv (e |u|^{1+\rho_2} )(\lambda^2 t,\lambda x),
$$
the scaling of which corresponds to the linear part \eqref{eq:SPDE_rescaled} if 
\begin{equation}
\label{eq:scaling_parameter2}
2+\sigma = 1+\sigma(1+\rho_2)\quad \text{ or equivalently }\quad
\sigma=1/\rho_2.
\end{equation}
Finally, for the noise term in \eqref{eq:SPDE_rescaled}, from the scaling of the noise \eqref{eq:scaling_noise}, it follows that 
$$
|u_\lambda(t,x)|^{1+\rho_3}\,\partial_t \mathcal{W}^{(\zeta)}(t,x)
\stackrel{{\rm law}}{=} 
\lambda^{1+d/2-d/\zeta+(1+\rho_3)\sigma} |u(\lambda^2 t,\lambda x)|^{1+\rho_3}\,\partial_t \mathcal{W}^{(\zeta)}(\lambda^2 t,\lambda x),
$$
which coincides with the scaling in \eqref{eq:SPDE_rescaled} if
\begin{equation}
\label{eq:scaling_parameter3}
\sigma= \frac{1}{\rho_3}\Big( 1-\frac{d}{2}+\frac{d}{\zeta} \Big).
\end{equation}
Throughout this manuscript, we impose that $\zeta\in [2,\infty]$ is such that the right-hand side above is positive (see Assumption \ref{ass: zeta f}). 

Summarizing the previous scaling analysis in \eqref{eq:scaling_parameter1}-\eqref{eq:scaling_parameter3}, the various nonlinearities in \eqref{eq:SPDE_rescaled} have the same scaling if and only if we have
\begin{equation}
\label{eq:scaling_parameter4}
\text{fully critical relations:}\quad 
\rho_1= 2\rho_2, \quad \text{ and }\quad 
\rho_3 = \frac{\rho_1}{2} \Big( 1-\frac{d}{2}+\frac{d}{\zeta}  \Big). 
\end{equation}

\emph{Critical spaces.}
The scaling invariance of local solutions $u\mapsto u_\lambda$ as in \eqref{eq:ulambda_formula} reveals the optimal regularity threshold for well-posedness of \eqref{eq:intro_equation}.
To compute critical spaces for the latter, one considers the map induced by \eqref{eq:ulambda_formula} on the initial data $u_0$:
\begin{equation}
\label{eq:map_induced_on_initial_data}
u_{0}\mapsto u_{0,\lambda}:=\lambda^\sigma u_0(\lambda \cdot),
\end{equation}
where $\sigma=2/\rho_1$ and \eqref{eq:scaling_parameter4} holds. In particular, the relations \eqref{eq:scaling_parameter1}, \eqref{eq:scaling_parameter2} and \eqref{eq:scaling_parameter3} are also satisfied.
A function space $X$ on $\cO$ for the initial data is \emph{critical} if it is locally invariant under the mapping \eqref{eq:map_induced_on_initial_data}. 
For the homogeneous version of $X$ this translates to a global invariance under the mapping $u_0\mapsto u_{0,\lambda}$. 
In particular, critical spaces are given by the following Lebesgue and Besov space
\begin{equation}
\label{eq:critical_spaces_examples}
L^{\frac{d\rho_1}{2}}(\cO) \qquad \text{ and }\qquad B^{\frac{d}{q}-\frac{2}{\rho_1}}_{q,p}(\cO) \ \text{ for } \ q,p\in (1,\infty).
\end{equation}
Indeed, the corresponding homogeneous versions satisfy $\|u_{0,\lambda}\|_{L^{\frac{d\rho_1}{2}}(\mathbb R^d)}= \|u_{0}\|_{L^{\frac{d\rho_1}{2}}(\mathbb R^d)}$, and 
$$
\|u_{0,\lambda}\|_{\dot B^{\frac{d}{q}-\frac{2}{\rho_1}}_{q,p}(\mathbb R^d)}
=
\lambda^{\sigma}\|u_0(\lambda \cdot)\|_{\dot B^{\frac{d}{q}-\frac{2}{\rho_1}}_{q,p}(\mathbb R^d)}
\eqsim 
\|u_{0}\|_{\dot B^{\frac{d}{q}-\frac{2}{\rho_1}}_{q,p}(\mathbb R^d)},
$$
where the implicit constant is independent of $\lambda>0$.
One can check that $\frac{d}{q}-\frac{2}{\rho_1}$ is the only value of the smoothness $s$ for which the homogeneous Besov $\dot{B}^{s}_{q,p}(\R^d)$ is scaling invariant under the mapping \eqref{eq:map_induced_on_initial_data}, and similarly for the integrability $\frac{d\rho_1}{2}$ among Lebesgue spaces.
Independently of the chosen structure of the function spaces, the Sobolev index of a critical space always amounts to 
\begin{equation}
\label{eq:critical_sobolev_index}
-2/\rho_1,\quad  \text{ (Critical Sobolev index)}
\end{equation}
which is shown to be uniquely identified from the above scaling argument.

\smallskip

If one of the scaling identities \eqref{eq:scaling_parameter1}, \eqref{eq:scaling_parameter2} and \eqref{eq:scaling_parameter3} is not satisfied, then the corresponding nonlinearity is \emph{not} critical. If that is the case and a strictly smaller $\rho_i$ is used, then the corresponding term vanishes in the limit
$\lambda\to 0$. In this sense, the scaling procedure identifies
the nonlinearities of highest order, matching
the scaling of the principal part.

\subsection{Illustration: optimal local well-posedness for scalar SPDEs}
\label{ss:illustration}
The preceding scaling calculation predicts the critical growth of the
nonlinearities and the corresponding initial-data spaces. The following
simplified statement (which follows from Corollary \ref{cor: critical localrho1} and Theorem \ref{thm:mainlocal zeta infty}) shows how this prediction is realized by our main theorem in a more concrete scalar situation:
\begin{equation}
\label{eq:intro_scalar_equation}
\dd u - \Delta u\,\dd t =\big[ f(u)+\dv ( F(u))\big]\,\dd t + g(u)\,\dd W(t)
\end{equation}
on a bounded smooth domain $\Dom\subset \R^d$ with Dirichlet boundary condition. Here the intuitive meaning of the noise term is
\[W(t) = \sum_{n= 1}^\infty \mu_n \varphi_n w_n(t),\]
where $(w_n)_{n\geq 1}$ are standard i.i.d.\ Brownian motions. 
We assume $\cS_\vp = (\varphi_n)_{n\geq 1}$ is an orthonormal system in $L^2(\cO)$ with $\vp_n\in L^\infty(\cO)$, $n\in\bN$, and the coefficients satisfy $\mu = (\mu_n)_{n\geq 1}\in \ell^\zeta(\Sf)$, where
the space
$\ell^\zeta(\Sf)$ denotes the space of sequences $(\mu_n)_{n\geq 1}$ for which $\|\mu\|_{\ell^\zeta(\Sf)}<\infty$, with the norm
\[\|\mu\|_{\ell^\zeta(\Sf)}^\zeta:=\sum_{n=1}^\infty |\mu_n|^\zeta \|\varphi_n\|_{L^\infty(\cO)}^2,\]
and $\|\mu\|_{\ell^\infty(\Sf)} = \|\mu\|_{\ell^\infty}$. 
The parameter $\zeta\geq 2$ quantifies the spatial roughness of the covariance of the noise ($\zeta=2$ being trace-class noise, and $\zeta=\infty$ being space-time white noise). In the literature the orthonormal system $\varphi$ is often taken as the eigensystem of $-\Delta$. For general smooth domains one typically has $\sup_{n\geq 1}\|\varphi_n\|_{L^\infty(\cO)}=\infty$. For special domains such as an interval or a cube the eigenfunctions satisfy $\sup_{n\geq 1}\|\varphi_n\|_{L^\infty(\cO)}<\infty$. Note that in our setting one need not choose the eigensystem of $-\Delta$. 
\medskip

\begin{theorem}
\label{t:intro}
Suppose that for the noise coefficient $\mu\in \ell^\zeta(\Sf)$, the intensity $\zeta\geq 2$ satisfies the condition
\[
\zeta<\frac{2d}{d-2}\quad\text{ whenever } \quad d\ge 2.
\]
Suppose that the nonlinearities $f$, $F$ and $g$ satisfy the locally Lipschitz bounds 
\begin{align*}
|f(x)-f(y)|&\leq C (1+|x|^{\rho}+|y|^{\rho})|x-y|,\\
|F(x)-F(y)|&\leq C (1+|x|^{\rho/2}+|y|^{\rho/2})|x-y|, 
\\ |g(x)-g(y)|&\leq C (1+|x|^{\rho_3}+|y|^{\rho_3})|x-y|, 
\end{align*}
where  $\rho>\max\Bigl\{\frac{2}{d},\frac{4}{d+2}\Bigr\}$ and $\rho_3 \leq \frac{\rho}{2}(1-\frac{d}{2}+\frac{d}{\zeta})$. Suppose that $q\in (2, \infty)$ satisfies
\begin{equation}\label{eq:rhointro}
\max\Big\{\frac{2}{\rho(\rho+1)},\frac{2}{\rho}-\frac{1}{\rho_3+1}\Big\}
<
\frac{d}{q}
<
1-\frac{d}{2}+\frac{d}{\zeta} + \min\Bigl\{\frac{d}{\rho+1},\frac{2}{\rho}
\Bigr\}.
\end{equation}
Then one can choose suitable temporal parameters $(p,\kappa)$ so that \eqref{eq:intro_scalar_equation} admits a unique maximal local solution $(u,\sigma)$ in a suitable space for $\cF_0$-measurable initial data $u_0$ in a critical Besov space 
$$
\BD^{\,\frac{d}{q}-\frac{2}{\rho}}_{q,p}(\Dom).
$$

Moreover, $u$ a.s.\ regularizes to 
\[u\in C^{\theta_1,\theta_2}_{\loc}((0,\sigma)\times \overline{\cO}),
        \quad\text{for any $\theta_1\in [0,\lambda/2)$ and $\theta_2\in [0,\lambda)$ with $\lambda = 1-\tfrac{d}{2}+\tfrac{d}{\zeta}$.}\]
\end{theorem}

In the above (and in all other results in the manuscript), the conditions simplify if the noise coefficient is globally Lipschitz. Indeed, then the conditions hold for every $\rho_3>0$, and therefore it can be taken arbitrarily small, and this simplifies the lower bound on $d/q$ considerably. In the above theorem $q=2$ is allowed if $\zeta<\infty$.

The term \emph{optimal} refers to the following aspects of Theorem~1.1. \begin{itemize}
    \item \textbf{Critical initial regularity.}
    The initial-data space
    ${}_D B^{d/q-2/\rho}_{q,p}(\mathcal O)$ has Sobolev index
    $-2/\rho$, exactly as prescribed by the scaling analysis
    in Subsection~1.1. Thus, local well-posedness is obtained
    at the critical regularity threshold.

    \item \textbf{Growth exponents up to the critical threshold.}
    The permitted polynomial growth orders for $f$, $F$ and $g$ are
    \[
        1+\rho,\qquad
        1+\frac{\rho}{2},\qquad
        1+\frac{\rho\lambda}{2}, \qquad \text{where} \ \lambda=1-\frac d2+\frac d\zeta.
    \]
    These are the growth orders for which the three nonlinear terms have the same scaling as the principal part.

    \item \textbf{Optimal space-time H\"older smoothness.}
    The solution instantaneously belongs to $C^{\theta_1, \theta_2}_{\loc}((0,\sigma)\times \cO)$ with $\theta_1<\frac{\lambda}{2}$ and $\theta_2<\lambda$.
    These exponents agree with the regularity predicted by the  corresponding linear stochastic heat equation (cf.\ \cite[Theorem 1.2]{hu2025spatio} for space-time white noise).

    \item \textbf{Sharp estimates for the stochastic term.}
     The $\gamma$-radonifying estimates established in \cite{AGV} are sharp
    in the parameter classes considered there. They provide the linear  estimates needed to treat the non-trace-class noise in the present  nonlinear problem.
\end{itemize}
We do not claim that every auxiliary restriction on $q$ in
\eqref{eq:rhointro} is optimal. 
For a continuous, nonnegative, nondecreasing function
$f:[0,\infty)\to[0,\infty)$, \cite[Theorem~3.4]{LRSV} characterizes
local existence of $L^q$-bounded solutions for every nonnegative
$L^q(\cO)$ initial datum by
\[
 \limsup_{r\to\infty}\frac{f(r)}{r^{1+2q/d}}<\infty,
 \qquad q>1.
\]
Thus if $\frac{d}{q} = \frac{2}{\rho}$, the power $1+\rho$ is already sharp in that deterministic class.

The expression 
$$
\lambda=1-\frac{d}{2}+\frac{d}{\zeta}
$$ 
in the above result is related to the regularity of the noise. Due to the condition on $\zeta$ one can check that $\lambda\in (0,1]$. Note that the space regularity of the solution $u$ is $C^{\lambda-}(\cO)$, and the regularity of the noise is $C^{\lambda-1-}(\cO)$.

\smallskip

For concreteness of the conditions on the parameters, consider $d\geq 3$, $\rho\geq 2$, and a globally Lipschitz noise coefficient $g$. Then the condition becomes  
\[\frac{2}{\rho(\rho+1)} < \frac{d}{q}<1+\frac{2}{\rho}-\frac{d}{2}+\frac d\zeta.
\]
Thus, we prove well-posedness with initial value in critical space $\BD^{d/q-2/\rho}_{q,p}(\cO)$ with smoothness 
$$
\frac{d}{q}-\frac{2}{\rho}\approx 
-\frac{2}{\rho+1}+\varepsilon \quad \text{ for }\quad \varepsilon>0.
$$
Finally, we point out that if $f=0$ we can allow for more irregular initial data, see Corollary \ref{cor: critical local rho2}.

\begin{example}[Cubic nonlinearity $f$]
For $\rho = 2$  and $d=3$, this simplifies to $\zeta\in [2, 6)$ with the following restriction on $q$: 
\begin{equation}
\label{eq:bound_q_allen_cahn_intro}
\frac{1}{3}<\frac{3}{q}<\frac{1}{2}+\frac{3}{\zeta},
\end{equation}
for a Lipschitz noise coefficient $g$.
If instead of global Lipschitzness we consider a general growth $\rho_3+1$ of the noise coefficient, then we also need $\rho_3 \leq \min\Big\{\frac12,-\frac{1}{2}+\frac{3}{\zeta}\Big\}$, where the $1/2$ comes from  \eqref{eq:rhointro} if we want to be able to let $q\uparrow 9$.

Even in this special case, the admissible range of integrability exponents $q$ and noise intensities $\zeta$ appears to be completely new. The same holds for most other choices. Letting $q\uparrow 9$, one obtains the largest possible class of initial values. This can be summarized as 
$$
u_0\in \BD^{-\frac{2}{3}+}_{9-,\infty-}(\Dom).
$$
In the absence of noise, it is possible to prove well-posedness in $\BD^{3/q-1}_{q,p}(\cO)$ for all $q,p$ sufficiently large for parabolic PDEs with cubic type nonlinearity, see e.g.\ \cite[Theorem 3.2]{CriticalQuasilinear}. The upper bound on $q$ is due to the roughness of the noise in time and space. We do not know if the bound in \eqref{eq:bound_q_allen_cahn_intro} is generically optimal. 
A similar comment applies to the restrictions in Theorem \ref{t:intro}.
\end{example}

\medskip

The above theorem is representative for the general picture developed below. 
The identification of such critical spaces is important for at least two reasons. 
First, these spaces are natural from the point of view of scaling and therefore provide a conceptually satisfactory local well-posedness theory. 
Second, they form the right starting point for future global arguments. 
Indeed, once local solutions are available in critical spaces, the regularization results proved in this paper show that solutions immediately enter better spaces, and our blow-up criteria reduce global well-posedness questions to the derivation of suitable a priori bounds in norms that are often accessible in concrete models.

We apply the abstract theory to several non-trace-class stochastic perturbations of the prototypical examples:
\begin{itemize}
    \item Allen--Cahn equation and Gray--Scott system -- Subsection \ref{ss: AllenCahn}.
        \item Fisher--KPP equation and Lotka--Volterra system -- Subsection \ref{ss: FisherKPP}.
    \item  Burgers equation -- Subsection \ref{ss: Burgers}.
    \item 1D Reaction-diffusion equations with space-time white noise -- Subsection \ref{ss:applwhite}.
\end{itemize}
The novelties of our results include:
\begin{itemize}
    \item Critical initial data in the PDE sense -- Subsection \ref{ss:scaling_intro}, and Corollaries \ref{cor: critical local rho2} and \ref{cor: critical localrho1}.
    \item Instantaneous regularization of solutions -- Theorem \ref{thm: regularization}.
    \item Sharp blow-up criteria in spaces with low smoothness -- Theorem \ref{thm: blow up} and Corollary \ref{cor: blow up}. 
    \item Positivity for rough critical initial data -- Theorem \ref{thm:positity_nonlinear_eq}.
\end{itemize}
We expect that the framework developed here will be useful in a number of further directions, e.g.\ global well-posedness, and long-time behavior for stochastic reaction-diffusion systems with rough noise and rough initial data.

The paper is organized as follows. 
In Section~\ref{sec:Main} we state our main results on local well-posedness, regularization, blow-up criteria, and positivity for stochastic reaction-diffusion equations. Section~\ref{sec:appl} contains applications to several concrete models. 
Finally, the proofs of the main results are given in Section~\ref{sec:proofs}.

\subsubsection*{Acknowledgements}
The authors thank Esmée Theewis for helpful comments. Moreover, we thank the reviewers for their careful reading and their valuable feedback.

\section{Main results}\label{sec:Main}

In this section we collect the main results of the paper. Applications can be found in Section \ref{sec:appl}. The proofs of the main results will be given in Section \ref{sec:proofs}.

\subsection{Setting}

Consider the following system of SPDEs on $\Dom$:
\begin{equation}
\label{eq:reaction_diffusion_system}
\left\{
\begin{aligned}
\dd u_i -\dv(a_i\cdot\nabla u_i) \,\dd t& = \Big[\dv(F_i(\cdot, u)) +f_i(\cdot, u)\Big]\,\dd t + \sum_{j=1}^{\ell} g_{i}^j (\cdot,u) \,\dd W^j(t),\\
u_i(0)&=u_{0,i},
\end{aligned}\right.
\end{equation}
where $i\in \{1,\dots,\ell\}$ and $\ell\geq 1$ is an integer.
Here $u=(u_i)_{i=1}^{\ell}:[0,\infty)\times \O\times \Dom\to \R^\ell$ is the unknown process. We write
$$
\dv (a_i\cdot\nabla u_i):=\sum_{j,k=1}^d \partial_j(a^{j,k}_i \partial_k u_i)
$$
in distributional sense. Coupling only takes place through the nonlinear mappings $F$, $f$ and $g$.

\subsubsection{Function spaces with Dirichlet boundary conditions}
\label{ss:function_spaces_Dir}

In this subsection we fix the notation for the function spaces used below (see \cite{AV19_QSEE_1, AVsurvey} for details). Let
$\cO\subset \bR^d$ be a bounded $C^2$-domain and let $q\in (1,\infty)$. We
write $q'$ for the conjugate exponent of $q$.

Let $\Delta_D$ denote the strong Dirichlet Laplacian on $L^q(\cO)$, i.e.
\[
    \mathsf{D}(\Delta_D)=H^{2,q}(\cO)\cap H^{1,q}_0(\cO),
    \qquad
    \Delta_D u=\Delta u.
\]
Equivalently, we set
\[
    \HD^{2,q}(\cO):=H^{2,q}(\cO)\cap H^{1,q}_0(\cO),
    \qquad
    \HD^{0,q}(\cO):=L^q(\cO),
\]
and
\[
    \HD^{1,q}(\cO):=H^{1,q}_0(\cO)
    =\{u\in H^{1,q}(\cO): u|_{\partial\cO}=0\}.
\]
The negative part of the scale is defined by duality. More precisely,
\begin{equation}
\label{eq:duality_HD_function_spaces}
    \HD^{-2,q}(\cO):=(\HD^{2,q'}(\cO))^*,
\end{equation}
where the duality is understood with respect to the extension of the
$L^2$-pairing. The intermediate spaces are obtained by complex interpolation:
for $-2\le s_0<s_1\le 2$ and $\theta\in(0,1)$ we put
\[
    \HD^{s,q}(\cO)
    :=
    [\HD^{s_0,q}(\cO),\HD^{s_1,q}(\cO)]_\theta,
    \qquad
    s=(1-\theta)s_0+\theta s_1.
\]
In particular, this leads to $\HD^{0,q}(\cO)=L^q(\cO)$. The spaces $H^{s,q}(\cO)$ can be introduced in the same way. 

We shall also use the corresponding real interpolation spaces. For
$p\in[1,\infty]$, $-2\le s_0<s_1\le 2$, and $\theta\in(0,1)$, we define
\begin{equation}
    \label{eq:def_BD_function_spaces}
    \BD^{s}_{q,p}(\cO)
    :=
    (\HD^{s_0,q}(\cO),\HD^{s_1,q}(\cO))_{\theta,p},
    \qquad
    s=(1-\theta)s_0+\theta s_1.
\end{equation}
By the reiteration theorem this definition is independent of the particular
choice of $s_0,s_1$ and $\theta$.
The space $\BD^{s}_{q,p}(\cO)$ is continuously embedded into the usual
Besov space $B^s_{q,p}(\cO)$ whenever $s\in(0,2)$ and $s\ne 1/q$.

Let us recall the standard identifications with the usual Bessel potential and
Besov spaces on $\cO$. For $s\in(0,2)$ with $s\ne 1/q$ one has
\[
    \HD^{s,q}(\cO)
    =
    \begin{cases}
        H^{s,q}(\cO), & s\in(0,1/q),\\[2mm]
        \{u\in H^{s,q}(\cO): u|_{\partial\cO}=0\}, & s\in(1/q,2).
    \end{cases}
\]
Similarly, for $s\in(0,2)$ with $s\ne 1/q$,
\[
    \BD^{s}_{q,p}(\cO)
    =
    \begin{cases}
        B^{s}_{q,p}(\cO), & s\in(0,1/q),\\[2mm]
        \{u\in B^{s}_{q,p}(\cO): u|_{\partial\cO}=0\}, & s\in(1/q,2).
    \end{cases}
\]

We shall frequently use the following Sobolev embeddings for the Dirichlet
scale. If $-2\le s_1\le s_0\le 2$, $1<q_0\le q_1<\infty$, then
\[
    s_0-\frac{d}{q_0}\ge s_1-\frac{d}{q_1} \ \text{implies} \ 
    \HD^{s_0,q_0}(\cO)\hookrightarrow \HD^{s_1,q_1}(\cO).
\]
Moreover, for $s\ge0$ and $s-d/q>\alpha\ge0$ with $\alpha\notin\N$, one has
\[
    \HD^{s,q}(\cO)\hookrightarrow H^{s,q}(\cO)
    \hookrightarrow C^\alpha(\overline{\cO}) \ \text{and} \ 
    \BD^{s}_{q,p}(\cO)
    \hookrightarrow B^{s}_{q,p}(\cO)
    \hookrightarrow C^\alpha(\overline{\cO}).
\]

\subsubsection{Nonlinearities and coefficients}
The following is the main assumption on the coefficients and nonlinearities.
\begin{assumption}
\label{ass:reaction_diffusion_global} 
Let $d,\ell\geq 1$ be integers. We say that Assumption \ref{ass:reaction_diffusion_global}$(s,q,p,\rho_1,\rho_2,\rho_3)$ holds if $q\in [2,\infty)$, $p\in (2,\infty)$, $\rho_1,\rho_2,\rho_3>0$, $s\in [0, 1)$ and for all $i\in \{1,\dots,\ell \}$ the following hold:
\begin{enumerate}[{\rm(1)}]
\item\label{it:reaction_diffusion_global0} {\rm (Regularity of the coefficients)} $a_i: \R_+\times \O\times \cO\to \R^{d\times d}$ is $\Progress\otimes \Borel(\cO)$-measurable, and there exists an $M$ such that a.s.\ for all $t\in [0,\infty)$, 
\[\|a^{j,k}_i(t,\cdot)\|_{C^1(\overline{\mathcal{O}})}  \leq M,\]
\item\label{it:reaction_diffusion_global1} {\rm (Parabolicity)}
There exists $\ellip>0$ such that a.s.\ for all $t\in [0,\infty)$, $x\in \overline{\cO}$ and $\xi\in \R^d$, 
$$
\sum_{j,k=1}^d a_i^{j,k}(t,x)
 \xi_j \xi_k
\geq  \ellip |\xi|^2.
$$
\item\label{it:growth_nonlinearities} {\rm (Local Lipschitzness)} For all $i,j\in \{1, \ldots, \ell\}$, $k\in \{1, \ldots, d\}$, the maps
\begin{align*}
F_i^k, \ f_i:\R_+\times \O\times \Dom\times \R^{\ell}\to \R,\qquad
g_{i}^j:\R_+\times \O\times \Dom\times \R^{\ell}\to \bR,
\end{align*}
are $\Progress\otimes \Borel(\Dom)\otimes \Borel(\R^{\ell})$-measurable, and satisfy
\begin{equation*}
F_i^k(\cdot,0), \ f_i(\cdot,0)\in L^{\infty}(\R_+\times \O\times \Dom),\qquad
g_i^j(\cdot,0)\in L^{\infty}(\R_+\times \O\times \Dom),
\end{equation*}
and a.s.\ for all $t\in \R_+$, $x\in \Dom$ and $y,y'\in\R^{\ell}$, for some $\rho_1,\rho_2,\rho_3>0$,
\begin{align*}
|f_i(t,x,y)-f_i(t,x,y')|
&\lesssim (1+|y|^{\rho_1}+|y'|^{\rho_1})|y-y'|,\\
|F_i^k(t,x,y)-F_i^k(t,x,y')|
&\lesssim (1+|y|^{\rho_2}+|y'|^{\rho_2})|y-y'|.
\\ |g_i^j(t,x,y)-g_i^j(t,x,y')|
&\lesssim (1+|y|^{\rho_3}+|y'|^{\rho_3})|y-y'|.
\end{align*}
\end{enumerate}
\end{assumption}

The above assumption is in its most general form. It may help the reader to consider a special choice, which in most cases does not lead to any restrictions: 
\begin{remark}\label{rem:simplifyrho}
Let $\zeta\in [2, \infty]$ be the parameter related to the noise intensity which will be introduced in Assumption \ref{ass: zeta f} below. Let $\rho>0$. Recall from Subsection \ref{ss:scaling_intro} that the following choice for $\rho_1, \rho_2, \rho_3$ is natural as it leads to the same scaling behavior in the SPDE:
\begin{align*}
\rho_1 = \rho, \qquad \rho_2 = \frac{\rho}{2}, \qquad \rho_3 \leq \frac{\rho}{2}\Big(1-\frac{d}{2} + \frac{d}{\zeta}\Big). 
\end{align*}
The conditions on $\zeta$ ensure that $\rho_3\in (0,\rho/2]$. Here the end-point $\rho/2$ only  occurs if $\zeta=2$ (i.e.\ trace-class noise).
\end{remark}

\subsubsection{The noise term}\label{sss:noise}

Let us formalize the structure of the stochastic forcing driving the system. 
Let $W^j = R^j B^j$, where $B^1,\ldots, B^\ell$ are independent $L^2(\cO)$-cylindrical Brownian motions. Moreover, $R^j\in \calL(L^2(\cO))$ is given by
\begin{align}
\label{eq: def R}
    R^j = \sum_{n=1}^\infty \mu_{n}^j e_n^j\otimes \varphi_n^j,
\end{align}
for suitable intensities $(\mu_n^j)_{n\geq 1}$, an orthonormal basis $(e_{n}^j)_{n\geq 1}$ and an orthonormal system $(\varphi_n^j)_{n\geq 1}$ in $L^2(\cO)$. In other words, the covariance $R^j(R^j)^*$ of $W^j(1)$ is given by 
\[R^j(R^j)^* = \sum_{n=1}^\infty |\mu_{n}^j|^2 \varphi_n^j\otimes \varphi_n^j.\]

Formally, the above corresponds to having
\[W^j(t) = \sum_{n=1}^\infty \mu_n^j \varphi_n^j w^j_n(t),\]
where $(w^j_n)_{j,n}$ is a family of standard independent Brownian motions.

The stochastic integral with respect to $W^j$ is then defined as
\begin{align}\label{eq:stochintWB}
    \int_0^t g_i^j(\cdot, u) \,\dd W^j(r)  = \int_0^t g_i^j(\cdot, u) R^j \,\dd B^j(r).
\end{align}
Thus, we need that $g_i^j(\cdot, u) R^j\in L^2(0,t;\gamma(L^2(\cO),\HD^{-s,q}(\cO)))$ a.s.\ to  ensure the stochastic integrability. Here $\gamma(L^2(\cO),Z)$ denotes the $\gamma$-radonifying operators from $L^2(\cO)$ into a Banach space $Z$. For details on $\gamma$-radonifying operators we refer the reader to \cite[Chapter 9]{Analysis2}. For details on stochastic integration in spaces such as $H^{-s,q}(\cO)$, the reader is referred to \cite[Section 4]{NVW13}.

To formulate the following assumptions concerning the structure of the noise, we introduce, for an orthonormal system
$\Sf^j = (\varphi_n^j)_{n\geq 1}$ of $L^2(\cO)$, such that $\varphi_n^j\in L^\infty(\cO)$, the space
\begin{align}
\label{def:ellzeta}
    \ell^\zeta(\Sf^j) =
    \begin{cases}
    \{(\mu_n)_{n\geq 1}\in \bR^\bN\,:\, \|\mu\|_{\ell^\zeta(\Sf^j)}^\zeta:=\sum_{n=1}^\infty |\mu_n|^\zeta \|\varphi_n^j\|_{L^\infty(\cO)}^2<\infty\},
   & \zeta\in [2,\infty),\\
    \ell^\infty, & \zeta = \infty.
    \end{cases}
\end{align}

Henceforth we fix orthonormal systems $\Sf^j$ in \eqref{eq: def R}
and impose the following relation on $\zeta$ and $d$ whenever $\mu^j\in\ell^{\zeta}(\Sf^j)$ for $\zeta<\infty$.

\begin{assumption}
\label{ass: zeta f}
For each $j\in \{1, \ldots, \ell\}$, $(\varphi^j_{n})_{n\geq 1}$ form an orthonormal system in $L^2(\cO)$ such that $\varphi^j_{n}\in L^\infty(\cO)$ for each $n\geq 1$. The sequence $\mu^j = (\mu_n^j)_{n\geq 1}\in \ell^\zeta(\Sf^j)$, where
$\zeta\in [2,\infty]$ is arbitrary if $d=1$ and is such that
\begin{align}
\label{eq: condition zeta}
    \zeta<\frac{2d}{d-2} 
    \quad\text{whenever} \quad d\geq 2.
\end{align}

\end{assumption}

    In many places throughout the paper we need that, for $\zeta\geq 2$ the number
    \begin{align}
        1-\frac{d}{2}+\frac{d}{\zeta}\in(0,1],
    \end{align}
    which is always satisfied in $d=1$ and in $d\geq 2$ leads precisely to the condition \eqref{eq: condition zeta}.

A key theorem for handling the noise term is our following recent result \cite[Theorem 4.1]{AGV}  on sharp estimates associated with \eqref{eq:reaction_diffusion_system} provided Assumption \ref{ass: zeta f} is satisfied. We state the main result of that latter paper below, as it provides an important ingredient for several of the proofs of the results below. It is a key ingredient in estimating the $\gamma$-norm of $g_i^j(\cdot, u) R^j$ which  appears when bounding \eqref{eq:stochintWB}. Below we let $M_g h = gh$ denote multiplication by $g$ for $h\in L^2(\cO)$. 

\begin{theorem}[Sharp bounds for coloured noise]
\label{thm:MgTmu delta}
Let $R^j$ be as in \eqref{eq: def R} with associated orthonormal system $(\varphi_n^j)_{n\geq 1}$ in $L^2(\cO)$ such that $\varphi^j_{n}\in L^\infty(\cO)$ for each $n\geq 1$. 
Let $q\in (1, \infty)$ and $s\in [0,d)\cap [0,2]$. Suppose that either
\begin{enumerate}[\rm (1)]
\item\label{it1:MgTmu delta} $\zeta\in (2,\infty)$, $\eta\in (1, q)$, $
\frac{s}{d} + \frac{1}{q}  = \frac{1}{\eta} + \frac{1}{2} - \frac{1}{\zeta}$ and $\frac{1}{\eta} - \frac{1}{\zeta}<\frac{1}{2}$,
\item\label{it2:MgTmu delta} $\zeta = 2$, $\eta\in (1, q]$ and 
$\frac{s}{d} + \frac{1}{q} = \frac{1}{\eta}$;
\item\label{it3:MgTmu delta} $\zeta = \infty$, $2<\eta<q<\infty$ and 
$\frac{s}{d} + \frac{1}{q} = \frac{1}{\eta}+\frac{1}{2}$. 
\end{enumerate}
Then for each $g\in L^{\eta}(\mathcal{O})$ and $\mu^j\in \ell^\zeta(\Sf^j)$, the operator $M_g R^j:L^2(\mathcal{O})\to \HD^{-s,q}(\mathcal{O})$ is $\gamma$-radonifying and satisfies the estimate
\begin{align}
\label{eq:maingammabound}
\|M_g R^j\|_{\gamma(L^2(\mathcal{O}),\HD^{-s,q}(\mathcal{O}))}\lesssim \|\mu^j\|_{\ell^{\zeta}(\Sf^j)} \|g\|_{L^{\eta}(\mathcal{O})},
\end{align}
where the constant depends on the parameters $(d,s,q,\eta,\zeta)$. 
\end{theorem}
Similar results hold for other boundary conditions and are immediate from \cite[Theorem 4.1]{AGV}. 
Through Theorem \ref{thm:MgTmu delta}, one can often check that $M_{g_i^j(\cdot, u)} R^j$ is in $L^2(0,t;\gamma(L^2(\cO), \HD^{-s,q}(\cO)))$ a.s.\ under suitable conditions on $u$ and $g_i$. This will be used in Section \ref{ss:locreg} below to obtain the well-definedness of the stochastic integral in \eqref{eq:stochintWB}.

In the existing literature on non-trace-class noise, it is common practice to assume that the  functions $(\varphi_n^j)_{n\geq 1}$ are the $L^2$-normalized eigenfunctions of a leading differential operator. A key advantage of our framework is that the sharp $\gamma$-radonifying bounds in Theorem \ref{thm:MgTmu delta} above rely exclusively on the $\ell^\zeta(\mathcal{S}_\varphi^j)$-norm and suitable norm of $g$. By decoupling the noise basis from the differential operator, our approach provides a highly flexible local theory that easily accommodates general orthonormal systems and domains where explicit eigenfunction bounds are unavailable.

Let us comment on a common choice in the literature to see what $\mu\in \ell^\zeta(\Sf^j)$ means. 
\begin{example}[Eigenfunctions of the Dirichlet Laplacian]
\label{ex:laplacian_eigenfunctions}
A canonical choice in the literature is to take the orthonormal system $\mathcal{S}_\varphi = (\varphi_n)_{n \ge 1}$ as the $L^2(\mathcal{O})$-normalized eigenfunctions of the negative Dirichlet Laplacian $-\Delta$ on a bounded smooth domain $\mathcal{O} \subset \mathbb{R}^d$, associated with eigenvalues $0 < \lambda_1 \le \lambda_2 \le \dots \to \infty$. 

According to standard spectral estimates (see, e.g., \cite{grieser2002uniform}) and Weyl's law ($\lambda_n \sim n^{2/d}$), the eigenfunctions satisfy the asymptotic bound:
\begin{equation}
    \|\varphi_n\|_{L^\infty(\mathcal{O})} \lesssim \lambda_n^{(d-1)/4} \sim n^{\frac{d-1}{2d}}.
\end{equation}
In this setting, the spatial coloring condition $\mu \in \ell^\zeta(\mathcal{S}_\varphi)$ is ensured if the sequence of scalars $(\mu_n)_{n \ge 1}$ satisfies the summability constraint:
\begin{equation}
    \sum_{n \ge 1} |\mu_n|^\zeta n^{\frac{d-1}{d}} < \infty.
\end{equation}
For further comparison with the literature the reader is referred to \cite[Section 6.3]{AGV}.
\end{example}

\subsection{Local existence, uniqueness, and regularity}\label{ss:locreg}

Henceforth we view equation \eqref{eq:reaction_diffusion_system} as an abstract evolution equation of the form \eqref{eq:SEE}, with the following choice of spaces:
\begin{align}
\label{eq: choice spaces}
        X_0 = \HD^{-1-s,q}(\cO),\quad X_1 = \HD^{1-s,q}(\cO),
\end{align}
in which case the interpolation spaces take the form
\begin{align}
    X_{1-\tfrac{1+\kappa}{p},p} =  \BD^{1-s-2\tfrac{1+\kappa}{p}}_{q,p}(\cO),\quad
    X_\theta = \HD^{-1-s+2\theta,q}(\cO),\quad\text{for $\theta\in (0,1)$.}
\end{align}
For simplicity in the notation we write $\HD^{s,q}(\cO)$ instead of $\HD^{s,q}(\cO;\R^\ell)$. The latter is needed since we use systems.

\begin{definition}[Strong $(s,q,p,\kappa)$-solution]
\label{def:strong_sqpkappa_solution}
Let $s \in [0, 1)$, $q \in [2, \infty)$, $p \in (2, \infty)$, and $\kappa \in [0, p/2 - 1)$. Let $\tau$ be a finite stopping time, and let $u_0 \in L^0_{\cF_0}(\Omega; \BD^{1-s-2\frac{1+\kappa}{p}}_{q,p}(\cO))$ be a given initial condition. 

A strongly progressively measurable process $u = (u_i)_{i=1}^\ell : [0, \tau] \times \Omega \to \HD^{-1-s,q}(\cO)$ is called a \emph{strong $(s,q,p,\kappa)$-solution} to \eqref{eq:reaction_diffusion_system} on $[0, \tau]$ if, almost surely,
\begin{align*}
    u & \in L^p(0, \tau, w_\kappa; \HD^{1-s,q}(\cO)) \cap C([0, \tau]; \BD^{1-s-2\frac{1+\kappa}{p}}_{q,p}(\cO)),\\ 
    \dv(F_i(\cdot, u)), & \, f_i(\cdot, u)\in L^{1}(0,\tau; \HD^{-1-s,q}(\cO)),\\ 
    g_i^j(\cdot, u)R^j&\in  L^{2}(0,\tau;\gamma(L^2(\cO),\HD^{-s,q}(\cO))),
\end{align*}
and for all $t \in [0, \tau]$ and $i \in \{1, \dots, \ell\}$, the following holds as an equality in $\HD^{-1-s,q}(\cO)$ almost surely:
\begin{align*}
    u_i(t) - u_{0,i} - \int_0^t \dv(a_i \cdot \nabla u_i(r)) \diff r &= \int_0^t \Big[\dv(F_i(\cdot, u(r))) + f_i(\cdot, u(r))\Big] \diff r \\
    &\quad + \sum_{j=1}^{\ell}\int_0^t \one_{[0,\tau]} g_i^j(\cdot, u(r)) \diff W^j(r).
\end{align*}
\end{definition}

Since $\HD^{-1-s,q}(\cO)=(\HD^{1+s,q'}(\cO))^*$ and $s\in [0,1)$, the conditions on $\dv F_i(\cdot,u)$ are understood via the duality form
$$
\langle \dv F_i(\cdot,u), \phi\rangle =-\int_{\cO} F_i(\cdot,u)\cdot \nabla \phi\,\dd x, \qquad \phi\in \HD^{1+s,q'}(\cO).
$$
Since the elements of $\HD^{1-s,q}(\cO)$ are functions for $s\in [0,1)$, the composition $F_i(\cdot,u)$ is well-defined in a pointwise sense. The same applies to $f_i(\cdot,u)$ and $g^j_i(\cdot,u)$, where for the latter, due to Theorem \ref{thm:MgTmu delta}, it is enough to prove $g_i^j(u)\in L^p(0,\tau,w_{\a};L^{\eta}(\Dom))$ a.s.\ for a suitable $\eta\in (1, q]$. 
From the conditions in Theorem \ref{thm:mainlocal} below, one can check that the integrability assumptions of the above Definition \ref{def:strong_sqpkappa_solution} are fulfilled. This will be clear from the proofs.

Due to the above integrability conditions, and standard measurability properties, the deterministic integrals exist as Bochner integrals in $\HD^{-1-s,q}(\cO)$ (see \cite[Chapter 1]{Analysis1}), and since $q\in[2, \infty)$, the stochastic integral exists as $\HD^{-s,q}(\cO)$-valued It\^o integral (see \cite[Section 4]{NVW13}).

An advantage of working with spaces of distributions is that it allows us to consider strong solutions even though the objects are not defined pointwise. Under the right regularity conditions on $u$, it is equivalent to use analytically weak solutions (see Remark \ref{rem:solconcept}\eqref{it1:solconcept}).  The advantage of the latter is that the integrals all become real-valued. In special cases, one can also consider mild solutions (see Remark \ref{rem:solconcept}\eqref{it2:solconcept}). The advantage of mild solutions is that the integrals do take values in spaces with positive spatial smoothness at least for almost all $t$.

\begin{remark}[Other solution concepts]\label{rem:solconcept}
The strong equality in $\HD^{-1-s,q}(\cO)$ formulated in Definition \ref{def:strong_sqpkappa_solution} is closely related to other standard solution concepts in the SPDE literature \cite{DPZ,VThesis}.
\begin{enumerate}[{\rm(i)}]
    \item\label{it1:solconcept} \emph{Analytically weak solutions:} Because the SPDE holds as a strong equality in the base space $X_0 = \HD^{-1-s,q}(\cO)$, it automatically satisfies the equation in the classical analytically weak (or distributional) sense. By taking the duality pairing with any test function $\phi \in \HD^{2,q'}(\cO)\stackrel{{\rm d}}{\hookrightarrow} \HD^{1+s,q'}(\cO)$ (where $\frac{1}{q} + \frac{1}{q'} = 1$), one can integrate by parts to transfer all the spatial derivatives onto $\phi$. Almost surely, for all $t \in [0, \tau]$ and $i \in \{1, \dots, \ell\}$,
    \begin{align*}
        \lb u_i(t),\phi\rb  - \lb u_{0,i},\phi \rb
        &= \int_0^t \int_{\cO}  u_i(r) \dv (a_i^T \cdot\nabla \phi)\,\dd x \,\dd r \\
        &\quad + \int_0^t \int_{\cO}\big( f_i(\cdot, u(r)) \phi - F_i(\cdot, u(r))\cdot \nabla \phi  \big) \,\dd x\, \diff r \\
        &\quad + \sum_{j=1}^\ell\sum_{n=1}^\infty\mu_n^j \int_0^t\one_{[0,\tau]}(r) \int_{\cO}   g_i^j(\cdot, u(r))\varphi_n^j \phi\, \dd x \,\dd w_n^j(r),
    \end{align*}
    where 
    $$
    w^j_n(t)=B^{j}(\one_{[0,t]}\otimes e^j_n),
    $$
    where $(B^j)_{j=1}^\ell$ are the $L^2(\cO)$-cylindrical Brownian motions as in Subsection \ref{sss:noise}. 
    Thus, a solution obtained through our $L^p(L^q)$-framework satisfies the SPDE in the standard weak sense commonly used in PDEs without the need to artificially weaken the solution definition itself. If $u, \dv (F_i(\cdot,u)),f_i(\cdot,u)$ and $g_i^j(\cdot,u)$ satisfy the integrability conditions in Definition \ref{def:strong_sqpkappa_solution} and $u$ satisfies the above integral identity, then it is a strong $(s,q,p,\kappa)$-solution in the sense of Definition \ref{def:strong_sqpkappa_solution}. 
    
    \item\label{it2:solconcept} \emph{Mild solutions:} In the special case that $a_i$ is independent of $(t,\omega)$, one has that $-A_i$ generates a strongly continuous analytic semigroup on $L^q(\cO)$. Moreover, using suitable extrapolation theory, one can check that it extends to a semigroup on $\HD^{-1-s,q}(\cO)$ which regularizes to spaces of functions. In that case, assuming enough smoothness of $u$ one can equivalently formulate the equation component-wise as a.s.\ for all $t\in [0,\tau]$ 
    \begin{align*}
    u_i(t) = e^{-tA_i}u_{0,i} & + \int_0^t e^{-(t-r)A_i} \Big[\dv(F_i(\cdot, u(r))) + f_i(\cdot, u(r))\Big] \diff r \\
    & + \sum_{j=1}^\ell \int_0^t e^{-(t-r)A_i} \one_{[0,\tau]}(r) g_i^j(\cdot, u(r)) \diff W^j(r).
    \end{align*}
\end{enumerate}
\end{remark}

Next we continue to define local and maximal solutions. 
\begin{definition}[Maximal local $(s,q,p,\kappa)$-solution]
\label{def:maximal_sqpkappa_solution}
Let the parameters and initial data $u_0$ be as in Definition \ref{def:strong_sqpkappa_solution}.
\begin{enumerate}[{\rm(1)}]
    \item \emph{(Local Solution)} A pair $(u, \sigma)$ is called a \emph{local $(s,q,p,\kappa)$-solution} to \eqref{eq:reaction_diffusion_system} if $\sigma : \Omega \to [0, \infty]$ is a stopping time, $u : [0, \sigma) \to \HD^{-1-s,q}(\cO)$ is strongly progressively measurable, and there exists an increasing localizing sequence of stopping times $(\sigma_n)_{n \ge 1}$ such that $\lim_{n \to \infty} \sigma_n = \sigma$ almost surely, and for each $n \ge 1$, the stopped process $u|_{[0, \sigma_n]}$ is a strong $(s,q,p,\kappa)$-solution on $[0, \sigma_n]$.
    
    \item \emph{(Uniqueness)} A local $(s,q,p,\kappa)$-solution $(u, \sigma)$ is called \emph{unique} if for any other local $(s,q,p,\kappa)$-solution $(v, \tau)$ with the same initial data, $u = v$ almost surely on $[0, \sigma \wedge \tau)$.
    
    \item \emph{(Maximality)} A unique local $(s,q,p,\kappa)$-solution $(u, \sigma)$ is called \emph{maximal} if for any other unique local solution $(v, \tau)$ with the same initial data, it holds that $\tau \le \sigma$ almost surely, and $u = v$ on $[0, \tau)$.
\end{enumerate}
\end{definition}

The following is our main result on existence and uniqueness of \eqref{eq:reaction_diffusion_system}.  To describe the precise time-regularity we need the weighted vector-valued Bessel potential spaces. For details the reader is referred to \cite{AV19_QSEE_1, AV19_QSEE_2, AVsurvey}. We recall the definition here. Define the weight $w_\kappa^a(t):=(t-a)^\kappa$ for real $a$ and $\kappa\geq 0$.
For $-\infty\leq a<b\leq \infty$, $I = (a,b), p\in (1,\infty)$, $\theta\in (0,1)$ and a Banach space $X$, we define the vector-valued Bessel potential spaces by
\begin{align}
    H^{\theta,p}(I,w^a_\kappa;X):=[L^p(I,w_\kappa^a;X),W^{1,p}(I,w^a_\kappa;X)]_\theta,
\end{align}
where $W^{1,p}(I,w^a_\kappa;X)$ is the Sobolev space of regularity $1$ and integrability $p$. If $w^a_\kappa = 1$, then we omit it from the notation.

Next we state a general local well-posedness result for $\zeta<\infty$. The case $\zeta=\infty$ and $d=1$ (i.e.\ multiplicative white noise) will be considered in Theorem \ref{thm:mainlocal zeta infty} below. 
Table \ref{tab:restrictions} provides a schematic overview of the parameters used below and their roles.

\begin{theorem}[Local existence and uniqueness]
    \label{thm:mainlocal}
    Let Assumption \ref{ass:reaction_diffusion_global}$(s,q,p,\rho_1,\rho_2,\rho_3)$
    hold, let $\kappa\in [0,\tfrac{p}{2}-1)$ 
    and consider \eqref{eq:reaction_diffusion_system} for an initial condition $u_0\in L^0_{\cF_0}(\Omega,\BD_{q,p}^{1-s-2\tfrac{1+\kappa}{p}}(\cO))$. Let  Assumption \ref{ass: zeta f} hold for $\zeta\in [2,\infty)$. 
 Suppose that the smoothness parameter $s$ satisfies
 \begin{align}
 \tag{Cs}
 \label{eq: Cs}
 \frac{d}{2} - \frac{d}{\zeta}< s<\min\Big\{1,d-\frac{d}{q}\Big\},
 \end{align}
 where we allow $s=0$ if $\zeta=2$. 
\begin{enumerate}[$\bullet$]
\item If $f\neq 0$, then suppose
    \begin{align}
    \hspace{-1cm}
        \tag{Cf1}
        \label{eq: mainlocal combined ineq f 1}
        1+\frac{1}{\rho_1+1}\frac{d}{q}
        <s + 2\frac{1+\kappa}{p} + \frac{d}{q}&< \min\Big\{d+1,\,\, 1+\frac{d(\rho_1+1)}{q}\Big\},
        \qquad s+\frac{d}{q}-1 < \frac{d}{\rho_1 + 1},\\
        \label{eq: mainlocal combined ineq f 2}
        \tag{Cf2}
        s + 2\frac{1+\kappa}{p}+\frac{d}{q} &\leq  \frac{\rho_1+2}{\rho_1}.
    \end{align}
    \item If $F\neq 0$, then suppose
    \begin{align}
    \label{eq: mainlocal combined ineq F}
    \tag{CF}
    1+\frac{1}{\rho_2+1}\frac{d}{q}<  s+2\frac{1+\kappa}{p}+\frac{d}{q}\leq   \frac{\rho_2+1}{\rho_2},
    \qquad s+\frac{d}{q}-1<\frac{d}{\rho_2+1}.
    \end{align}
    \item For the growth parameter $\rho_3$ of $g$ we suppose that
    \begin{equation}
       \tag{Cg}
        \begin{split}
            \label{eq: thm local s ineq 3}
        s + 2\frac{1+\kappa}{p} + \frac{d}{q} &\leq 1+\frac{1}{\rho_3}\Big(1-\frac{d}{2}+\frac{d}{\zeta} \Big),
\\
s + \frac{2\rho_3}{1+\rho_3}\frac{1+\kappa}{p}&\leq 1.
        \end{split}
        \end{equation}
        \end{enumerate}
    Then there exists a maximal $(s,q,p,\kappa)$-solution $(u,\sigma)$ to \eqref{eq:reaction_diffusion_system} such that $\sigma>0$ a.s.\ and 
    \begin{enumerate}[{\rm(1)}]
        \item\label{it1:localwellposed} {\em (Regularity)} There exists $\theta^*\in (0,1/2]$ such that for each localizing sequence $(\sigma_n)_{n\geq 1}$ for $(u,\sigma)$ one has
for all $n\geq 1$ and all $\theta\in [0,\theta^*)$
                \[u\in H^{\theta,p}(0,\sigma_n,w_{\kappa};\HD^{1-s-2\theta,q}(\cO))\cap C([0,\sigma_n];\BD^{1-s-2\tfrac{1+\kappa}{p}}_{q,p}(\cO)).\] 
        \item\label{it2:localwellposed} {\em (Localization)} Let $v_0\in L^0_{\cF_0}(\Omega;\BD^{1-s-2\tfrac{1+\kappa}{p}}_{q,p}(\cO))$. If $(v,\tau)$ is the maximal $(s,q,p,\kappa)$-solution to \eqref{eq:reaction_diffusion_system} with initial value $v_0$, then a.s.\ on the set $\{u_0=v_0\}$ one has $\sigma=\tau$ and $u = v$ on $[0,\sigma\wedge \tau)$.
    \end{enumerate}
\end{theorem}

The reader is encouraged to check what the conditions become in the special case of Remark \ref{rem:simplifyrho}. This will also be discussed in the critical situations below in Corollary \ref{cor: critical localrho1}.

The explicit value of $\theta^*$ is known; we elaborate on the following. 

\begin{remark}
\label{rem: theta star}
    We remark here that, for a given set of parameters $(s,d,p,q,\kappa,\zeta,\rho_1,\rho_2)$, the value of $\theta^*$ is explicitly known. Indeed, by Theorem \ref{thm:local}, we know that $\theta^*\leq\min\{\tfrac{1}{2},1-\alpha^*_1,1-\alpha^*_2\}$, where $\alpha^*_1$ is the smallest admissible number in \eqref{eq: range alpha case I}, and $\alpha^*_2$ is the smallest admissible number in \eqref{eq: range alpha 2}. Plugging in these lower bounds on $\alpha_1,\alpha_2$ gives that
    \begin{align}
    \label{eq: theta star bound}
        \theta^*\leq \min\Big\{ &\frac{1}{2},\,
        (\rho_1+1)\frac{1+\kappa}{p}+\frac{\rho_1}{2}\Big(s+\frac{d}{q}-1\Big),\,
        \frac{1}{2}\Big(1-s+\frac{\rho_1d}{q}\Big),\,
        \\
        &\frac{1}{2}+\frac{\rho_2}{2}\Big(s+\frac{d}{q}-1\Big)+(\rho_2+1)\frac{1+\kappa}{p},\,1-\frac{s}{2}+\frac{\rho_2 d}{2q}
        \Big\}.
    \end{align}
    Due to the length of the expression, we omitted this information in the statement of Theorem \ref{thm:mainlocal}.
    Following the proof of Theorem \ref{thm:mainlocal} it is also clear that $\theta^*\in (0,1/2]$, due to the range of values of $\alpha_1,\alpha_2$ in \eqref{eq: range alpha case I} and \eqref{eq: range alpha 2}, and that if $f = F = 0$, then $\theta^* = \frac{1}{2}$.
    
    Moreover, the general dependence of $\theta^*$ on $(s,q,p,\rho_1,\rho_2,\rho_3,\kappa)$ and hence, the (positive) degree of smoothness in time of $u$, is not relevant here. It is more useful when dealing with specific applications and hence, more specific parameters $(s,q,p,\rho_1,\rho_2,\rho_3,\kappa)$. Indeed, we will see in Theorem \ref{thm: regularization} that since $\theta^*>0$, the function $u$ will always be in $C^{\theta_1,\theta_2}((0,\sigma)\times\overline{\cO})$ a.s.\ for strictly positive values of $\theta_1$ and $\theta_2$ independent of $\theta^*$.
\end{remark}

\begin{table}[H]
\centering
\renewcommand{\arraystretch}{1.25}
\begin{tabularx}{\linewidth}{|p{2.7cm}|X|}
\hline
\textbf{Restriction}&\textbf{Role in the proof}\\\hline
$q\in[2,\infty)$
& Ensures that the underlying Sobolev spaces are UMD spaces of type~$2$,
as required by the stochastic maximal-regularity framework used here.
\\\hline
$p\in(2,\infty)$, \newline $\kappa\in[0,p/2-1)$
& The temporal integrability and weight range used for weighted
stochastic maximal regularity; see \cite{AV19}.
\\\hline
$\zeta\in[2,\infty)$
& Describes the colouring regimes of the noise. The endpoint $\zeta=\infty$,
which includes space-time white noise, is treated separately in
dimension one in Theorem~\ref{thm:mainlocal zeta infty}.
\\\hline
Condition \eqref{eq: Cs}
& The restriction $s<1$ ensures that $\HD^{1-s,q}(\cO)$ consists of
functions, so the nonlinear compositions are meaningful. The remaining
bounds ensure that the parameters in the $\gamma$-radonifying estimate
of Theorem~\ref{thm:MgTmu delta} are admissible.
\\\hline
Condition \eqref{eq: mainlocal combined ineq f 1}
& Ensures compatibility of the Sobolev embeddings used for $f(u)$ with
the admissible ranges of $\alpha_1$ and $\beta_1$; see
\eqref{eq: range alpha case I}.
\\\hline
Condition \eqref{eq: mainlocal combined ineq f 2}
& Results from combining those Sobolev embeddings with the abstract
criticality inequality \eqref{eq:subcritical}.
\\\hline
Condition \eqref{eq: mainlocal combined ineq F}
& The strict inequalities express compatibility of the Sobolev
embeddings for $\operatorname{div}F(u)$ with the admissible
$\alpha_2,\beta_2$ ranges. The non-strict inequality is the corresponding
criticality condition \eqref{eq:subcritical2}.
\\\hline
Condition \eqref{eq: thm local s ineq 3}
& Arises by combining the $\gamma$-radonifying estimate for the noise
with the Sobolev estimate for $g(u)$ and the stochastic criticality
condition \eqref{eq:subcritical3}.
\\\hline
\end{tabularx}
\caption{Roles of the parameter restrictions in the proof of
Theorem~\ref{thm:mainlocal}.}
\label{tab:restrictions}
\end{table}

Further comments on Theorem \ref{thm:mainlocal} are given below.

\begin{remark}\
\begin{itemize}
\item 
In the setting of Theorem \ref{thm:mainlocal}, we consider highly irregular initial data. This forces $s$ to be large and thus the solutions to live in spaces of distributions (or spaces with very low fractional regularity). If one takes $s\in (0,1)$ smaller, then the proof of the local well-posedness becomes simpler since more regularity can be used. However, in that situation one can often not reach criticality and thus the natural scaling; see Subsections \ref{ss:scaling_intro}-\ref{ss:illustration}. In Proposition \ref{prop:mainlocal_bounded} we will actually present another well-posedness result for small $s$ as it will be required later on. 

\item 
The restrictions in Theorem~\ref{thm:mainlocal} arise from the
stochastic maximal-regularity framework, the mapping properties of the
nonlinearities, and the criticality conditions. Table \ref{tab:restrictions}
summarizes their roles in the proof. These are requirements of the
present argument; their necessity for well-posedness in greater
generality is not asserted. Details are given in
Subsection~\ref{ss:local_wp_proof_critical}.
\end{itemize}
\end{remark}

Due to parabolic regularization, the solution $u$ of Theorem \ref{thm:mainlocal} instantaneously becomes smoother. In order to state the regularity result also in H\"older sense we introduce the space-time H\"older spaces. For $\theta_1,\theta_2\in (0,1)$, the space $C^{\theta_1,\theta_2}([a,b]\times\overline{\cO})$ denotes the space of functions $u:[a,b]\times\overline{\cO}\to\bR^\ell$ such that there is a constant $C$ such that for all $t,s\in [a,b]$ and $x,y\in \overline{\cO}$
\begin{align}
    |u(t,x)-u(s,y)|\leq C |t-s|^{\theta_1} + C|x-y|^{\theta_2},
    \quad\text{for $s,t\in (a,b)$ and $x,y\in\overline{\cO}$.}
\end{align}
Moreover, we say that $u\in C_{\rm loc}^{\theta_1,\theta_2}((a,b)\times\overline{\cO})$ if for all $a<c<d<b$ it holds that $u|_{[c,d]}$ is in $C^{\theta_1,\theta_2}([c,d]\times\overline{\cO})$. 

\begin{theorem}[Instantaneous regularization]
\label{thm: regularization}
    Suppose that the conditions of Theorem \ref{thm:mainlocal} hold. 
    Let $(u,\sigma)$ be the maximal $(s,q,p,\kappa)$-solution to \eqref{eq:reaction_diffusion_system} with initial data $u_0\in L^0_{\cF_0}(\Omega,\BD^{1-s-2\tfrac{1+\kappa}{p}}_{q,p}(\cO))$. 
    Then we have almost surely 
    \begin{align}
    \label{eq: reg u result 1}
        u&\in H^{\theta,\bar p}_{\loc}\big((0,\sigma),\HD^{1-\tfrac{d}{2}+\tfrac{d}{\zeta} -2\theta-\varepsilon,\bar q}(\cO)\big),\quad\text{for $\overline q\geq 2$, $\overline p>2$ and $\theta,\varepsilon\in (0,1/2)$},\\
        \label{eq: reg u result 2}
        u&\in C^{\theta_1,\theta_2}_{\loc}((0,\sigma)\times \overline{\cO}),
        \quad\text{for $\theta_1\in [0,\lambda/2)$ and $\theta_2\in [0,\lambda)$ with $\lambda = 1-\tfrac{d}{2}+\tfrac{d}{\zeta}$,}
    \end{align}
    where the statement \eqref{eq: reg u result 1} also holds for $\theta = 0$ and $\varepsilon = 0$ if $\zeta = 2$.
\end{theorem}

Due to the restrictions on $\zeta$, one can check that $\lambda=1-\tfrac{d}{2}+\tfrac{d}{\zeta}\in (0,1]$, so that the spatial regularity always lies in $(0,1)$. In particular, if $\zeta>2$, then $\nabla u$ might only exist in distributional sense. 

Note that \cite{agresti2023reaction} is basically the case $\zeta=2$, and therefore in Theorem \ref{thm: regularization} we recover the regularity result of \cite{agresti2023reaction}. 

\begin{remark}[On Neumann boundary conditions]
\label{r:neumann}
The results of this section also hold true if  $F_i(\cdot,y)\cdot n|_{\partial\Dom}=0$ for all $y\in \R^\ell$, and the boundary condition $(a_i\nabla u_i )\cdot n=0$ on $\partial\cO$ is independent of $t$. For instance, this happens if $a_i$ are time-independent, or if all the operators in \eqref{eq:reaction_diffusion_system} are constant multiples of the Laplacian. More details can be found in Remark \ref{rem:NeumannMR}.
\end{remark}

\subsection{Identification of the critical spaces of initial data}
\label{ss:critical cases}
Next we present corollaries in which we investigate when  the critical space of initial values can be reached. We do this in two parts.
First we present the case $F \neq 0$ and  $f=0$ (thus no $\rho_1$ term) in Corollary \ref{cor: critical local rho2}, since this case is somewhat simpler and leads to fewer restrictions on $q$.  Then we present the general case with $f\neq 0$ and $F \neq 0$ in Corollary \ref{cor: critical localrho1}.

\begin{corollary}[critical $F$ and $f = 0$]
    \label{cor: critical local rho2}
   Let Assumptions \ref{ass:reaction_diffusion_global}$(s,q,p,\rho_2,\rho_3)$ and  \ref{ass: zeta f} 
   hold with $\zeta\in [2,\infty)$ and
   \begin{align*}
\rho_2 = \frac{\rho}{2} \ \ \text{and} \ \ \rho_3 \leq \frac{\rho}{2}\Big(1-\frac{d}{2} + \frac{d}{\zeta}\Big),
\end{align*}
for $\rho>\max\big\{\frac{4}{2+d},\,\frac{2}{d}\big\}$.  Suppose that $f=0$ and that $q$ satisfies
\begin{align}
\label{eq:rho2condq}
\frac{2}{\rho} - \frac{1}{1+\rho_3}< \frac{d}{q}<1-\frac{d}{2}+\frac{d}{\zeta}  + \min\Big\{\frac{2d}{\rho+2},\,
\frac{2}{\rho}
\Big\}.
\end{align}
Let $p\in (2, \infty)$, and $\kappa\in [0,p/2-1)$ be such that
\begin{align}
\label{eq:rho2condpkappa}
      \max\Big\{
      \frac{\rho+2}{\rho}-d
      ,\,\big(\frac{2}{\rho} - \frac{d}{q}\big)(1+\rho_3),\,
      \frac{2}{\rho}-\frac{2d}{\rho+2}
      \Big\} 
      &< 2\frac{1+\kappa}{p}< 1 
      -\frac{d}{2}+\frac{d}{\zeta} + \frac{2}{\rho} -\frac{d}{q},
\end{align}
where, if $\zeta = 2$, instead of the upper bound in \eqref{eq:rho2condpkappa}, $2\tfrac{1+\kappa}{p}\leq \tfrac{\rho+2}{\rho} -\tfrac{d}{q}$ is required.
Then for the critical choice 
\[s =  \frac{\rho+2}{\rho} - \frac{d}{q} - 2\frac{1+\kappa}{p}\] 
the conditions of Theorems \ref{thm:mainlocal} and \ref{thm: regularization}   hold with critical space of initial values $L^0_{\cF_0}(\Omega,\BD_{q,p}^{\frac{d}{q}-\frac{2}{\rho}}(\cO))$.  
\end{corollary}

\begin{remark}\label{rem:rho2non} \
\begin{enumerate}[\rm (i)]
\item\label{it:rho2non1} Note that $\rho>\frac{4}{2+d}$ is a necessary condition for $s<1$ to be true for some $(p,\kappa,q)$ in the admissible range.  Indeed, this follows from $1>s =  \frac{\rho+2}{\rho} - \frac{d}{q} - 2\frac{1+\kappa}{p}>1+\frac{2}{\rho} - \frac{d}{2} - 1 = \frac{2}{\rho} - \frac{d}{2}$, which implies the lower bound on $\rho$.
The condition $\rho>\frac{2}{d}$ is required, such that an admissible pair $(p,\kappa)$ can be found.
\item We remark also (this is shown in the proof of Corollary \ref{cor: critical local rho2}) that under condition \eqref{eq:rho2condq} there always exist admissible $(p,\kappa)$ such that $\tfrac{1+\kappa}{p}\in (0,1/2)$ satisfies \eqref{eq:rho2condpkappa}.
\end{enumerate}
\end{remark}

Next we consider the case where $f$ is critical, and possibly $F$ is critical as well.
\begin{corollary}[critical $f$ and $F$]
    \label{cor: critical localrho1}
   Let Assumptions \ref{ass:reaction_diffusion_global}$(s,q,p,\rho_1,\rho_2,\rho_3)$ and  \ref{ass: zeta f} 
   hold with $\zeta\in [2,\infty)$ and 
\begin{align*}
\rho_1 = \rho, \qquad \rho_2 = \frac{\rho}{2}, \qquad \rho_3 \leq \frac{\rho}{2}\Big(1-\frac{d}{2} + \frac{d}{\zeta}\Big),
\end{align*}
for $\rho>\max\{\frac{2}{d},\frac{4}{2+d}\}$.
   Suppose that $q$ satisfies
\begin{align}
\label{eq:rho1condq}
\max\Big\{\frac{2}{\rho(\rho+1)},\frac{2}{\rho} - \frac{1}{\rho_3+1}\Big\}
&<\frac{d}{q}<1 -\frac{d}{2}+\frac{d}{\zeta} +\min\Big\{\frac{d}{\rho+1},\,
\frac{2}{\rho}
\Big\},
\end{align}
Let $p\in (2, \infty)$, and $\kappa\in [0,p/2-1)$ be such that
\begin{align}\label{eq:rho1condpkappa}
      \max\Big\{ 
      \frac{\rho+2}{\rho}-d,\,\Big(\frac{2}{\rho} - \frac{d}{q}\Big)(\rho_3+1),\,
      \frac{2}{\rho} - \frac{d}{\rho+1}
      \Big\} 
      &< 2\frac{1+\kappa}{p}< 1 -\frac{d}{2}+\frac{d}{\zeta} + \frac{2}{\rho}-\frac{d}{q},
\end{align}
where, if $\zeta = 2$, we may suppose $2\tfrac{1+\kappa}{p}\leq \tfrac{\rho+2}{\rho}-\tfrac{d}{q}$ in place of the upper bound in \eqref{eq:rho1condpkappa}.
Then, for the critical choice 
\begin{align}
s =  \frac{\rho+2}{\rho} - \frac{d}{q} - 2\frac{1+\kappa}{p},
\end{align} 
the conditions of Theorem \ref{thm:mainlocal} and \ref{thm: regularization} hold with critical space of initial values $L^0_{\cF_0}(\Omega,\BD_{q,p}^{\frac{d}{q}-\frac{2}{\rho}}(\cO))$.  
\end{corollary}

\begin{remark}
    Once again we refer to Remark \ref{rem: theta star} regarding the precise value of $\theta^*$.
\end{remark}

Next we extend the discussion on the necessity of the conditions from Remark \ref{rem:rho2non}\eqref{it:rho2non1}. 
\begin{remark}
\label{rem:rho1}\
\begin{enumerate}[\rm (i)]
\item We remark also that under condition \eqref{eq:rho1condq} there always exist admissible $(p,\kappa)$ such that $\tfrac{1+\kappa}{p}\in (0,1/2)$ satisfies \eqref{eq:rho1condpkappa}.
\item We remark also that for the Allen Cahn equation with $\rho = 2$, we can consider $d\geq 2$, even in the critical case (see Example \ref{ex:AllenCahn2}), whereas until now (for $p>2$), only $d\geq 3$ was possible in the critical case (compare to \cite{AV19_QSEE_1}). Let us also mention that for $p=q=2$, $s=0$ and $d=2$ the critical case was obtained in \cite{BGV}.
\end{enumerate}
\end{remark}

\begin{remark}
The special case of trace-class noise of Theorem \ref{thm:mainlocal} already appeared in \cite{agresti2023reaction}. Moreover, in the latter paper an additional transport noise term was considered. The techniques of our current manuscript can be used to extend the results of \cite{agresti2023reaction}.  This is due to the fact that instead of \cite{AV19_QSEE_1,AV19_QSEE_2} we can apply its more flexible version as presented in Subsection \ref{sec:loc-well-posed} where there is the additional range shift index $\alpha$. The details will be proved in \cite{AV-ob}. 

Many more critical spaces were identified in Corollaries \ref{cor: critical local rho2} and \ref{cor: critical localrho1}. In particular, no lower bound on $\frac{d}{q}$ appears in Corollary \ref{cor: critical local rho2} if $\rho\geq 2$ and the noise coefficient $g$ is globally Lipschitz. Moreover, the Fujita type condition on $\rho$ appearing in both corollaries is less restrictive than the one in \cite[Section 2.2]{agresti2023reaction}, where we note that $h = \rho+1$. Finally, unlike in \cite{agresti2023reaction} we can now even  identify several critical spaces for $d=1$. 
\end{remark}

\subsection{Blow up criteria}
Theorems \ref{thm:mainlocal} and \ref{thm: regularization} and their consequences discussed in Subsection \ref{ss:critical cases} yield a general framework to obtain maximal $(s,q,p,\a)$-solution $(u,\sigma)$ to the system of SPDEs \eqref{eq:reaction_diffusion_system}, with a.s.\ positive lifetime $\sigma$. In many situations, and especially in applications to physical problems, one is interested in understanding whether $\sigma=\infty$ a.s., that is, whether $u$ is \emph{global} in time. The latter question is very delicate and depends on the fine structure of the SPDE under consideration see e.g.\ \cite{S21_dissipative,S21_superlinear_no_sign,C03,agresti2024reaction,AVsurvey} and the references therein. Moreover, there are many open problems on global well-posedness of reaction-diffusion equations even in the deterministic case \cite{P10_survey, FMT20}.
However, it is possible to give sufficient conditions that, together with suitable a-priori (energy) bounds on the solutions, provide $\sigma=\infty$ a.s.\ These are typically called \emph{blow-up criteria}. For the maximal $(s,q,p,\a)$-solution provided by Theorem \ref{thm:mainlocal}, we can prove the following blow-up criteria. 

\begin{theorem}[Blow-up criteria]
\label{thm: blow up}
Let the assumptions of Theorem \ref{thm:mainlocal} be satisfied with parameters $(s,q,p,\kappa,\zeta,\rho_1,\rho_2,\rho_3)$, and let $(u,\sigma)$ be the maximal $(s,q,p,\kappa)$-solution to \eqref{eq:reaction_diffusion_system} provided by Theorem \ref{thm:mainlocal}.
Let  
$(s_0,q_0,p_0,\kappa_0,\zeta,\rho_{0},\rho_{0}/2,\rho_{0,3})$ be another set of parameters satisfying the conditions of Corollary \ref{cor: critical local rho2} or Corollary \ref{cor: critical localrho1} where   $\rho_{0,3} \leq \frac{\rho_{0}}{2}(1-\frac{d}{2}+\frac{d}{\zeta})$.
Then the following assertion holds for all $r>0$:
\begin{enumerate}[{\rm(1)}]
    \item\label{it:eq: blow up 1} For $\beta_0:=\tfrac{d}{q_0}-\tfrac{2}{\rho_0}$ and any $q_1>q_0$,
\begin{align*}
    \P\Big(r<\sigma<\infty,\, \sup_{t\in [r,\sigma)}\|u(t)\|_{B^{\beta_0}_{q_1,\infty}(\cO)}<\infty \Big) = 0.
\end{align*}
\item\label{it:eq: blow up 2} For $\nu_0:=\tfrac{d}{q_0}-\tfrac{2}{\rho_0}+\tfrac{2}{p_0}$ such that $\nu_0\neq \frac{1}{q_0}$ and $\beta_0=\tfrac{d}{q_0}-\tfrac{2}{\rho_0}\neq \frac{1}{q_0}$, 
\begin{align*}
    \P\Big(r<\sigma<\infty,\, \sup_{t\in [r,\sigma)}\|u(t)\|_{B^{\beta_0}_{q_0,p_0}(\cO)} + \|u\|_{L^{p_0}(r,\sigma;H^{\nu_0,q_0}(\cO))}<\infty \Big) = 0.
\end{align*}
\end{enumerate}
\end{theorem}

In light of the scaling argument in Subsection \ref{ss:scaling_intro}, the blow-up criterion in \eqref{it:eq: blow up 2} is sharp since the space-time Sobolev index of the space $L^{p_0}(r,\sigma;H^{\nu_0,q_0}(\cO))$ is given by 
$$
-\frac{2}{p_0}+\nu_0-\frac{d}{q_0}=-\frac{2}{\rho_0},
$$
which coincides with the \emph{critical} Sobolev index, see \eqref{eq:critical_sobolev_index}. 
With a similar argument one can check that \eqref{it:eq: blow up 1} is subcritical.
As is clear from the proof of
\eqref{it:eq: blow up 2}, the restrictions
$\beta_0\neq 1/q_0$ and $\nu_0\neq 1/q_0$
can be removed if the spaces
$B^{\beta_0}_{q_0,p_0}(\cO)$ and
$H^{\nu_0,q_0}(\cO)$ are replaced by their
Dirichlet counterparts.

\smallskip

In many situations (see e.g.\ \cite{agresti2024reaction,milesis2026global}), it is possible to obtain $L^r$-estimates for solutions to the system of SPDEs \eqref{eq:reaction_diffusion_system}. In case of Lebesgue spaces, the above result immediately yields the following.

\begin{corollary}[Blow-up criteria in Lebesgue spaces]
\label{cor: blow up}
Let the assumptions of Theorem \ref{thm:mainlocal} be satisfied with parameters $(s,q,p,\kappa,\zeta,\rho_1,\rho_2,\rho_3)$, and let $(u,\sigma)$ be the maximal $(s,q,p,\kappa)$-solution to \eqref{eq:reaction_diffusion_system} provided by Theorem \ref{thm:mainlocal}.
Let  
$(s_0,q_0,p_0,\kappa_0,\zeta,\rho_{0},\rho_{0}/2,\rho_{0,3})$ be another set of parameters satisfying the conditions  of Corollary \ref{cor: critical local rho2} or Corollary \ref{cor: critical localrho1} where  $\rho_{0,3} \leq \frac{\rho_{0}}{2}(1-\frac{d}{2}+\frac{d}{\zeta})$.
    Let $m_0:=\tfrac{d\rho_0}{2}$. Then the following assertions hold for all $r>0$:
    \begin{enumerate}[{\rm(1)}]
        \item\label{it:cor_blow_up1} If $m_0 = q_0$, then for any $m>q_0$,
    \begin{align}
        \P\Big( r<\sigma<\infty,\, \sup_{t\in [r,\sigma)}\|u(t)\|_{L^m(\cO)}<\infty \Big)=0.
    \end{align}
    \item\label{it:cor_blow_up2} If $q_0>m_0$, $\max\{q_0,\rho_0\}<p_0<\infty$ and $\tfrac{2}{p_0}+\tfrac{d}{q_0} = \tfrac{2}{\rho_0}$, then
    \begin{align}
        \P\Big( r<\sigma<\infty,\,\sup_{t\in [r,\sigma)}\|u(t)\|_{L^{m_0}(\cO)} &+ \|u\|_{L^{p_0}(r,\sigma;L^{q_0}(\cO))}<\infty\Big) = 0.
    \end{align}
    \end{enumerate}
\end{corollary}

We remark that $\tfrac{2}{p_0}+\tfrac{d}{q_0} = \tfrac{2}{\rho_0}$ implies that $p_0>\rho_0$.

The comments on items \eqref{it:eq: blow up 1} and \eqref{it:eq: blow up 2} of Theorem \ref{thm: blow up} readily extend to Corollary \ref{cor: blow up} \eqref{it:cor_blow_up1}-\eqref{it:cor_blow_up2}. In particular, \eqref{it:cor_blow_up1} and \eqref{it:cor_blow_up2} are sub-optimal and optimal from a scaling point of view
(see Subsection \ref{ss:scaling_intro}), respectively.

\smallskip

The proofs of Theorem \ref{thm: blow up} and Corollary \ref{cor: blow up} are postponed to Subsection \ref{ss:proof_blow_up_criteria}.

\subsection{Special case: space-time white noise in 1D}
\label{ss: additional wellposedness}
Next we present a version of the theory in Subsection \ref{ss:locreg} for the case $\zeta=\infty$ for $d=1$ which corresponds to space-time white noise. 

\begin{theorem}[$\zeta = \infty$ and $d=1$]
\label{thm:mainlocal zeta infty}
     Let Assumption \ref{ass:reaction_diffusion_global}$(s,q,p,\rho_1,\rho_2,\rho_3)$
    hold with $q>2$, let $\kappa\in [0,\tfrac{p}{2}-1)$
    and consider \eqref{eq:reaction_diffusion_system} for $u_0\in L^0_{\cF_0}(\Omega,\BD_{q,p}^{1-s-2\tfrac{1+\kappa}{p}}(\cO))$. Suppose that 
    \[\frac{1}{2}<s<1-\frac{1}{q}\]
    \begin{enumerate}[\rm (1)]
    \item\label{it1:mainlocal zeta infty} If $f\neq 0$ let \eqref{eq: mainlocal combined ineq f 1} and \eqref{eq: mainlocal combined ineq f 2} hold, and if $F\neq 0$ let \eqref{eq: mainlocal combined ineq F} hold. Let Assumption \ref{ass: zeta f} hold with $\zeta = \infty$
    and for $2<q<\infty$,
        \begin{align}
            \label{eq: thm local s ineq 3 zeta infty}
        s + 2\frac{1+\kappa}{p} + \frac{1}{q} &\leq 1+\frac{1}{2\rho_3} \  \ \text{and} \ \ s + \frac{2\rho_3}{1+\rho_3}\frac{1+\kappa}{p}\leq 1.
        \end{align}
         Then the assertions of Theorem \ref{thm:mainlocal} continue to hold true.

        \item\label{it2:mainlocal zeta infty}  Let $\rho_1 = \rho$, $\rho_2 = \rho/2$ and $\rho_3 \leq \rho/4$         for a $\rho>2$.
        If $(s,q,p,\kappa,\rho,\rho_3)$ are such that
        \begin{align}
            \frac{2}{\rho(\rho+1)}<\frac{1}{q},
            \quad\text{and}\quad
            \max\Big\{\frac{2}{\rho},\Big(\frac{2}{\rho} - \frac{1}{q}\Big)(\rho_3+1)\Big\}<2\frac{1+\kappa}{p}<\frac{1}{2}+\frac{2}{\rho}-\frac{1}{q},
        \end{align}
        then the assertions of Theorem \ref{thm:mainlocal} hold for the critical choice $s = \tfrac{\rho+2}{\rho}-\tfrac{1}{q}-2\tfrac{1+\kappa}{p}$, where the (critical) space of initial conditions is $\BD_{q,p}^{\tfrac{1}{q}-\tfrac{2}{\rho}}(\cO)$. Moreover, if $f=0$, then the lower restriction on $\frac{1}{q}$ can be omitted.  
        \item\label{it3:mainlocal zeta infty} If either \eqref{it1:mainlocal zeta infty} or \eqref{it2:mainlocal zeta infty} are satisfied, then for the maximal $(s,q,p,\kappa)$-solution $(u,\sigma)$ we have almost surely 
        \begin{align}
            u&\in H^{\theta,\bar p}_{\loc}\big((0,\sigma),\HD^{\frac{1}{2} -2\theta-\varepsilon,\bar q}(\cO)\big),\quad\text{for $\overline q\geq 2$, $\overline p>2$ and $\theta,\varepsilon\in (0,1/2)$},\\
            u&\in C^{\theta_1,\theta_2}_{\loc}((0,\sigma)\times \overline{\cO}),
            \quad\text{for every $\theta_1\in [0,1/4),\theta_2\in [0,1/2)$}.
        \end{align}
        \item\label{it4:mainlocal zeta infty} 
        Let $\rho_1 = \rho$, $\rho_2 = \rho/2$ and $\rho_3 \leq \rho/4$ for a $\rho>2$. Then the assertions of Corollary \ref{cor: blow up} hold with $m_0:=\tfrac{\rho}{2}$.
        \end{enumerate}
\end{theorem}
The proof of Theorem \ref{thm:mainlocal zeta infty} \eqref{it1:mainlocal zeta infty} is given in Section \ref{sec:proofs}. The proofs of Theorem \ref{thm:mainlocal zeta infty} \eqref{it2:mainlocal zeta infty}, \eqref{it3:mainlocal zeta infty} and \eqref{it4:mainlocal zeta infty} are a verbatim repetition of the proofs of the analogous results, Corollaries \ref{cor: critical local rho2}, \ref{cor: critical localrho1}, Theorem \ref{thm: regularization} and \ref{cor: blow up}, respectively, with the adjustment that $\zeta = \infty$.

\subsection{Special case: general covariance structures without spatial bounds}
The following result is a version of Theorem \ref{thm:mainlocal} for the case where we know only $\mu^j\in\ell^\zeta$, without imposing a relation of $\mu$ to $(\|\varphi_n^j\|_{L^\infty})_{n\geq 1}$ and where we do not even require that $\varphi_n^j\in L^\infty$, $n\geq 1$. 
However, without taking into account any additional structure, this forces $d=1$.

\begin{proposition}[$\mu\in \ell^\zeta$ for $\zeta\in [2,\infty)$ and $d=1$]
\label{prop:mainlocal unweighted}
     Let Assumption \ref{ass:reaction_diffusion_global}$(s,q,p,\rho_1,\rho_2,\rho_3)$
    hold, let $\kappa\in [0,\tfrac{p}{2}-1)$ 
    and consider \eqref{eq:reaction_diffusion_system} for an initial condition $u_0\in L^0_{\cF_0}(\Omega,\BD_{q,p}^{1-s-2\tfrac{1+\kappa}{p}}(\cO))$. Let 
    \[ \frac{1}{2}-\frac{1}{\zeta}<s< 1-\frac{1}{q},\]
    where $s=0$ is allowed if $\zeta=2$. 
    If $f\neq 0$ let \eqref{eq: mainlocal combined ineq f 1} and \eqref{eq: mainlocal combined ineq f 2} hold, and if $F\neq 0$ let \eqref{eq: mainlocal combined ineq F} hold. 
    Let $\mu^1, \ldots, \mu^\ell\in\ell^\zeta$ and
     suppose that for $\zeta\in [2,\infty)$ we have
        \begin{align}
            \label{eq: thm local s ineq 3 no f}
        s + 2\frac{1+\kappa}{p} + \frac{1}{q} &\leq \frac{\rho_3+1}{\rho_3}  -\frac{1}{2\rho_3}
        \qquad\qquad\qquad
        s>\frac{1}{2}-\frac{1}{q},\\
        \label{eq: thm local s ineq 3 no f 2}
        \text{and}\qquad s + \frac{2\rho_3}{1+\rho_3}\frac{1+\kappa}{p}&\leq 1.
        \end{align}
        Then the assertions of Theorem \ref{thm:mainlocal} continue to hold true.
\end{proposition}
However, as the reader can check the conditions on the parameters are more restrictive in the above proposition when compared to Theorem \ref{thm:mainlocal}. Moreover, because of the $\zeta$-independence of the condition in \eqref{eq: thm local s ineq 3 no f} it does not capture critical scaling  in SPDEs.

\subsection{Positivity}\label{s:pos}

Positivity of solutions to stochastic reaction-diffusion equations has been studied by many authors (see e.g.\ \cite{Ass99,CPT16,CES13,K92_positivity,Mar19,MS19} and the references therein). Our contribution establishes positivity even for rough initial data and noise. Leveraging our instantaneous regularization result (Theorem \ref{thm: regularization}), we provide a streamlined proof by reducing it to the maximum principle in \cite{Krylov13}. 

Let $\D_+(\Dom)$ be the subset of test functions such that $\varphi\geq 0$ on $\Dom$.
Below, we say that a distribution $v\in\D'(\cO)$ is non-negative in distribution (or briefly, $v\geq 0$ in $\D'(\cO)$) if 
$$
\langle v,\varphi\rangle \geq 0 \ \text{ for all } \ \varphi\in \D_+(\cO).
$$

\begin{theorem}[Positivity]
\label{thm:positity_nonlinear_eq}
In the setting of Theorems \ref{thm:mainlocal} or \ref{thm:mainlocal zeta infty}, suppose further that for all $i\in\{1,\dots,\ell\}$, and $y\in [0,\infty)^\ell$, a.e.\ on $\R_+\times\Omega$ it holds that 
\begin{equation}
\label{eq:assumption_positivity}
\begin{aligned}
&f_i(\cdot, y_1,\dots,y_{i-1},0,y_{i+1},\dots,y_\ell)\geq 0,  \qquad 
F_i(\cdot, y_1,\dots,y_{i-1},0,y_{i+1},\dots,y_\ell)=c_i,\\ 
&\qquad\qquad\qquad \qquad \qquad g_i^j(\cdot, y_1,\dots,y_{i-1},0,y_{i+1},\dots,y_\ell)=0,
\end{aligned}
\end{equation}
where $c_i:[0,\infty)\times \O\to \R^d$ are progressively measurable (hence independent of $x,y$), and  
$$
u_{0,i}\geq 0 \text{ in }\D'(\cO) \text{ a.s. } 
$$
Then $u_i(t)\geq 0$ a.s.\ for all $t\in (0,\sigma)$ and $i\in \{1, \ldots, \ell\}$. 
\end{theorem}

Note that, due to the instantaneous regularization Theorem \ref{thm: regularization}, $u(t)\in C(\overline{\cO})$ a.s.\ for all $t\in (0,\sigma)$, and the conclusion of Theorem \ref{thm:positity_nonlinear_eq} is well-defined pointwise in $\cO$.

\section{Applications}\label{sec:appl}
In this section, we apply our abstract well-posedness and regularization framework to several prototypical stochastic partial differential equations (SPDEs). While our main theorems are stated in terms of general growth parameters $\rho_i$, these bounds are deeply connected to the canonical nonlinearities that drive physical, biological, and fluid-dynamical models. 

In particular, we focus on polynomial nonlinearities of cubic ($\rho_1 = 2$) and quadratic ($\rho_1 = 1$ and $\rho_2=1$) type. These structures are not merely mathematical test cases, but rather represent the fundamental leading-order interactions in nonlinear systems:
\begin{itemize}
    \item {\em Cubic non-divergence nonlinearities ($\rho_1 = 2$):} Arising naturally in models of phase separation and quantum field theory (such as the Ginzburg-Landau or Allen-Cahn equations), cubic drifts model systems driven by symmetric double-well potentials, where the nonlinearity drives the separation of states. See Subsection \ref{ss: AllenCahn}.
    \item {\em Quadratic non-divergence nonlinearities ($\rho_1 = 1$):} In biological and chemical models (like the Fisher-KPP or coupled Lotka--Volterra systems), quadratic terms representing binary interactions dictate population dynamics and reaction rates.  See Subsection \ref{ss: FisherKPP}.
    \item {\em Quadratic divergence nonlinearities ($\rho_2=1$):} In fluid dynamics,  advective transport fundamentally generates quadratic divergence-form nonlinearities, driving the energy cascade in the Burgers and Navier-Stokes equations. See Subsection \ref{ss: Burgers}.
\item Finally, in Subsection \ref{ss:applwhite} we consider more general $\rho_1$, and focus on the case  $\zeta=\infty$ (white noise). We present applications to some rather academic superlinear examples in which we can establish global existence and uniqueness for rough initial data.
\end{itemize}

The interplay between polynomial growth and singular stochastic forcing typically restricts solutions to distribution spaces of very low regularity. The examples below illustrate how our abstract framework translates directly to these specific physical models. Based on  Theorems \ref{thm:mainlocal} and \ref{thm: regularization}, we identify the exact critical scaling regimes for the Allen-Cahn, Burgers, and Fisher-KPP equations, thereby guaranteeing local well-posedness for highly irregular initial states. 

In a forthcoming paper, we will improve our current local results to global ones. This will be done through the regularization and blow-up criteria of Theorem \ref{thm: regularization} and Corollary \ref{cor: blow up} and sophisticated energy bounds. 

\subsubsection*{Well-posedness for SPDEs with $L^r(\cO)$-valued initial data} While this section
collects applications of our main result, Theorem \ref{thm:mainlocal}, to a selection of well known SPDEs, the following table extracts the respective special cases of $L^r$-valued initial data. More precisely, it displays the lowest required integrability $r$ of the initial data $u_0$ in order to obtain local well-posedness  in $L^r(\cO)$. 

Explanations of the $L^r$-case for the models listed in the table are given in remarks or comments in their respective subsections, see, in particular, Remarks \ref{rem:L r Allen Cahn}, \ref{rem:L r Fisher} and \ref{rem:L r stwn}.

We also recall from Subsection \ref{ss:critical cases} on the critical setting that,  if $\zeta\in [2,\infty]$ and $\rho_1$ or $\rho_2$ are given, then 
the admissible values of $\rho_3$ satisfy
$\rho_3\leq \rho_\star \big( 1-\frac{d}{2}+\frac{d}{\zeta}  \big)$ where $\rho_\star$ is either $\rho_1/2$ or $\rho_2$ depending on the application; see \eqref{eq:scaling_parameter4}.

\begin{table}[H]
    \centering
    \renewcommand{\arraystretch}{1.3}
    \begin{tabular}{|l|c|c|c|}
        \hline
        \textbf{Equation} & $\boldsymbol{\rho_1}$ & $\boldsymbol{\rho_2}$ & \textbf{Initial value in $L^r(\cO)$} \\
        \hline
        Allen-Cahn  & 2 &  & $r \geq \max\{d, 1+\varepsilon\}$ \\
        Gray-Scott  & 2 &  & $r \geq \max\{d, 1+\varepsilon\}$ \\
        Fisher-KPP & 1 &  & $r \geq \max\{\frac{d}{2}, 1+\varepsilon, \1_{d=2}\frac{2(1+\rho_3)}{2+\rho_3}+\varepsilon\}$ \\
        Predator-Prey & 1 &  & $r \geq \max\{\frac{d}{2}, 1+\varepsilon, \1_{d=2}\frac{2(1+\rho_3)}{2+\rho_3}+\varepsilon\}$ \\
        Burgers &  & 1 & $r \geq \max\{d, 1+\varepsilon\}$ \\
        Space-time white noise in 1D & $\rho>2$ & & $r \geq \frac{\rho}{2}$ \\
        \hline
    \end{tabular}
    \caption{Summary of conditions for the growth parameters and spatial integrability for initial data in $L^r(\cO)$ with $\cO\subseteq \R^d$, with $d\geq 1$. Here $\varepsilon>0$ is arbitrary.
    }
    \label{tab:growth_parameters}
\end{table}

\subsection{Allen-Cahn type nonlinearities: cubic $f$}
\label{ss: AllenCahn}
To see the following results, it suffices to translate the specific physical nonlinearities into the growth bounds of Theorem \ref{thm:mainlocal} and Corollary \ref{cor: critical localrho1} in the critical case. These results immediately give local existence and uniqueness results for the equations below. All problems are considered on a bounded $C^2$-domain $\cO\subset \R^d$ and with Dirichlet boundary conditions.

\begin{example}[stochastic Allen-Cahn equation]\label{ex:AllenCahn}
Consider the stochastic Allen-Cahn equation, which arises in the modeling of phase transitions:
\begin{align}
\label{eq:ex_allen_cahn}
    \dd u - \Delta u \, \dd t = (u - u^3) \, \dd t + g(u) \, \dd W(t).
\end{align}
The drift is given by the cubic nonlinearity $f(u) = u - u^3$. This implies the local Lipschitz bound $|f(u)-f(v)| \lesssim (1+|u|^2+|v|^2)|u-v|$, which perfectly matches our framework with $\rho_1 = 2$ and $F = 0$. We apply Theorem \ref{thm:mainlocal} with the space of initial values  $\BD_{q,p}^{1-s-2\tfrac{1+\kappa}{p}}(\cO)$. For this we obtain the following restrictions on the regularity index $s$, the spatial integrability $q \in [2, \infty)$, and the temporal parameters $p \in (2, \infty)$ and $\kappa \in [0, p/2 - 1)$:
\begin{itemize}
    \item In dimension $d=1$: For $s \in (0, 1 - 1/q)$, the parameter space is partitioned by $q$:
    \begin{align*}
        1+\frac{1}{3q} < s  + 2\frac{1+\kappa}{p}+\frac{1}{q} &< 2 \quad &&\text{for } q \in [2, 3), \\
         1+\frac{1}{3q} < s+ 2\frac{1+\kappa}{p}+\frac1q &< 1+\frac{3}{q} \quad &&\text{for } q \in [3, \infty).
    \end{align*}
    \item In dimension $d=2$: For $s \in (0, 1)$,
    \begin{align*}
        1+\frac{2}{3q} &< s+2\frac{1+\kappa}{p} + \frac{2}{q} \leq 2, \quad s+\frac{2}{q}<\frac{5}{3} \quad &&\text{for } q \in [2, 6), \\
        1+\frac{2}{3q} &< s+2\frac{1+\kappa}{p} + \frac{2}{q} < 1+\frac{6}{q} \quad &&\text{for } q \geq 6.
    \end{align*}
    \item In dimension $d=3$: For $s \in (0, 1)$,
    \begin{align*}
        1+\frac{1}{q} &< s + \frac{3}{q} + 2\frac{1+\kappa}{p} \leq 2 \quad &&\text{for } q \in [2, 9), \\
        1 &< s + 2\frac{1+\kappa}{p} + \frac{2}{q} < 1 + \frac{8}{q} \quad &&\text{for } q \in [9, \infty).
    \end{align*}
\end{itemize}
For $g$ and the noise we require 
\begin{itemize}
\item $\zeta<\frac{2d}{d-2}$ if $d\geq 2$;
\item  $s>\tfrac{d}{2}-\tfrac{d}{\zeta}$ if $\zeta>2$ ;
\item the condition \eqref{eq: thm local s ineq 3} holds if $g$ is globally Lipschitz continuous (i.e. $\rho_3>0$ can be taken arbitrarily small). Beyond the globally Lipschitz setting for $g$, the result holds if additionally  
\[\rho_3 \leq \min\Big\{1-\frac{d}{2}+\frac{d}{\zeta},\,3\frac{1+\kappa}{p}-1\Big\},\]
so long as $\kappa,p$ can be chosen such that $\frac{1+\kappa}{p}>\frac{1}{3}$. If such $\kappa,p$ cannot be chosen, then \eqref{eq: thm local s ineq 3} may lead to further restrictions on $s,q,p,\kappa$, in addition to the ones presented above.
\end{itemize}
\end{example}

\begin{example}[critical stochastic Allen-Cahn equation]\label{ex:AllenCahn2}
For the Allen-Cahn equation ($\rho_1=2$) in dimensions $d \geq 2$, the critical scaling balances at
\begin{align*}
    s + \frac{d}{q} + 2\frac{1+\kappa}{p} = 2.
\end{align*}
Suppose that $g$ is globally Lipschitz and hence $\rho_3$ can be chosen arbitrarily small. Then, Corollary \ref{cor: critical localrho1} yields the critical initial data space $\BD_{q,p}^{\frac{d}{q}-1}(\cO)$ under the constraints
\begin{align*}
    \frac{1}{3}&< \frac{d}{q} < \min\Big\{1-\frac{d}{6}+\frac{d}{\zeta}, \, 2-\frac{d}{2}+\frac{d}{\zeta}\Big\},\\
    \max\Big\{1 - \frac{d}{q}, 1-\frac{d}{3}\Big\} &< 2\frac{1+\kappa}{p} < 2-\frac{d}{q}-\frac{d}{2}+\frac{d}{\zeta}.
\end{align*}
in the case $g$ is globally Lipschitz. Letting $q \uparrow 3d$, the initial data regularity approaches $-2/3$. Provided we select $q$ sufficiently close to $3d$ and optimize the temporal parameters $(p, \kappa)$, the only restriction on the noise coloring is the baseline geometric constraint $\zeta < \frac{2d}{d-2}$ (for $d \geq 2$) from Assumption \ref{ass: zeta f}.

Beyond the globally Lipschitz setting for $g$, the result holds if additionally 
\begin{align}
\rho_3 \leq 1-\frac{d}{2} + \frac{d}{\zeta}\qquad\text{and}\qquad
     \frac{d}{q}  + 2\frac{1+\kappa}{p}\frac{1}{\rho_3+1}>1.
\end{align}

If $\rho_3\leq \frac{1}{2}$, then we can still let $q \uparrow 3d$.
\end{example}

\begin{remark}[$L^r$-initial data for Allen-Cahn]
\label{rem:L r Allen Cahn}
    By inspecting the initial value spaces $ \BD_{q,p}^{\sigma}(\cO)$ for $\sigma = 1-s-2\frac{1+\kappa}{p}$, one can check that, under the constraint $\rho_3 \leq 1-\frac{d}{2} + \frac{d}{\zeta}$,
    for any $r\geq\max\{d,1+\varepsilon\}$ and $\varepsilon>0$ arbitrarily small,
    one can find $q,p,s,\kappa$  such that the Sobolev embedding $L^r(\cO)\hookrightarrow \BD_{q,p}^{\sigma}(\cO)$ holds for $d\geq 1$ and such that \eqref{eq:ex_allen_cahn} is well-posed in $L^r(\cO)$. 
    For the reader's convenience, we give an admissible choice of parameters to verify this: for $\delta>0$ sufficiently small choose
    \begin{align}
        d&=1,\quad q = \frac{2}{1-2\delta},\quad s = \frac{1+\delta}{2},\quad p = q,\quad \kappa = 0,\\
        d&=2,\quad q=3,\quad s=1-\frac{\rho_3}{2},\quad p\geq 6,\quad \kappa = \Big(\frac{1}{6}+\frac{\rho_3}{4}\Big)p-1,\\
        d&\geq 3,\quad q = d,\quad s = 1-\frac{\lambda}{2},\quad p\geq\max\Big\{d,\frac{4}{\lambda}\Big\},\quad \kappa = \frac{\lambda p}{4}-1,
    \end{align}
    where $\lambda = 1-\frac{d}{2}+\frac{d}{\zeta}$.
\end{remark}

Global well-posedness of the Allen--Cahn equation in the above setting will be considered in a future paper.

\begin{example}[The Stochastic Gray-Scott model]\label{ex:GrayScott}
To illustrate our framework's capacity to handle coupled cubic systems, we consider the stochastic Gray-Scott model. This system is famous for its rich array of pattern-forming behaviors (such as Turing spots, stripes, and spirals) driven by an autocatalytic reaction ($A + 2B \to 3B$):
\begin{equation}
\left\{
\begin{aligned}
    \dd u - \dv(a_1 \cdot \nabla u) \, \dd t &= (-u v^2 + r_f(1-u)) \, \dd t + g_1^1(u,v) \, \dd W^1(t) + g_1^2(u,v) \, \dd W^2(t), \\
    \dd v - \dv(a_2 \cdot \nabla v) \, \dd t &= (u v^2 - (r_f+r_k)v) \, \dd t + g_2^1(u,v) \, \dd W^1(t) + g_2^2(u,v) \, \dd W^2(t).
\end{aligned}
\right.
\end{equation}
Here, $r_f$ and $r_k$ are dimensionless constants representing feed and kill rates. The critical nonlinearities are the cubic autocatalytic cross-terms $\mp u v^2$. For the joint state vector $(u, v) \in \mathbb{R}^2$, this cubic cross-product structure perfectly matches our abstract local Lipschitz condition with  $\rho_1 = 2$ and $F = 0$. 

Because the growth parameter ($\rho_1 = 2$) is mathematically identical to that of the scalar Allen-Cahn equation of Examples \ref{ex:AllenCahn} and \ref{ex:AllenCahn2}, the exact same geometric scaling restrictions on the parameter space apply. 
Thus, we obtain the same local theory and critical spaces. Global existence for the stochastic Gray-Scott model is much harder to obtain than for fully coercive systems. For the case of transport noise and $\zeta=2$, global well-posedness was recently established in \cite{agresti2024reaction} for the more general Brusselator model for $d=2$ and $d=3$. Global well-posedness for certain triangular systems for the rougher situation $\zeta>2$ was recently considered in \cite{milesis2026global}. The blow-up criteria of Corollary \ref{cor: blow up} set the stage to obtain many further global well-posedness results, for rough classes of initial data. 

Finally, Remark \ref{rem:L r Allen Cahn} also applies to the presented Gray-Scott model.
\end{example}

\subsection{Fisher-KPP type nonlinearities: quadratic $f$}
\label{ss: FisherKPP}

\begin{example}[The Stochastic Fisher-KPP Equation]\label{ex:FisherKPP}
Consider the stochastic Fisher-KPP equation, a classical model for population dynamics and chemical kinetics:
\begin{align}
\label{eq:ex_fisher}
    \dd u - \Delta u \, \dd t = u(1-u) \, \dd t + g(u) \, \dd W(t).
\end{align}
Here, the drift is driven by a quadratic polynomial $f(u) = u - u^2$. The local Lipschitz geometry yields $|f(u)-f(v)| \lesssim (1+|u|+|v|)|u-v|$, which corresponds exactly to our abstract framework with $\rho_1 = 1$ and $F = 0$. 
We apply Theorem \ref{thm:mainlocal} with the space of initial values $\BD_{q,p}^{1-s-2\tfrac{1+\kappa}{p}}(\cO)$. The regularity index $s$, the spatial integrability $q \in [2, \infty)$, and the temporal parameters $p \in (2, \infty)$ and $\kappa \in [0, p/2 - 1)$ are subject to the following bounds:
\begin{itemize}
    \item In dimension $d=1$: For $s \in (0, 1 - 1/q)$, the combined scaling parameters must satisfy:
    \begin{align*}
        1 + \frac{1}{2q} < s + \frac{1}{q} + 2\frac{1+\kappa}{p} < 1 + \frac{2}{q}.
    \end{align*}
    
    \item In dimension $d=2$: For $s \in (0, 1)$ the parameter space is bounded by:
    \begin{align*}
        1 + \frac{1}{q} < s + \frac{2}{q} + 2\frac{1+\kappa}{p} < 1 + \frac{4}{q}.
    \end{align*}
    \item In dimension $d=3$: For $s \in (0, 1)$ the parameter space partitions at $q=3$:
    \begin{align*}
        1 + \frac{3}{2q} &< s + \frac{3}{q} + 2\frac{1+\kappa}{p} \leq 3 \quad &&\text{for } q \in [2, 3), \\
        1 + \frac{3}{2q} &< s + \frac{3}{q} + 2\frac{1+\kappa}{p} < 1 + \frac{6}{q} \quad &&\text{for } q \in [3, \infty).
    \end{align*}
\end{itemize}
For $g$ and the noise we require
\begin{itemize}
\item $\zeta<\frac{2d}{d-2}$ if $d\geq 2$;
\item  $s>\tfrac{d}{2}-\tfrac{d}{\zeta}$ if $\zeta>2$ ;
\item the condition \eqref{eq: thm local s ineq 3}, which is trivially satisfied for globally Lipschitz $g$, or $\rho_3$ arbitrarily small. 
For $g$ of the form considered in this paper, with associated growth parameter $\rho_3$, a sufficient condition for the first condition in \eqref{eq: thm local s ineq 3} is 
\begin{align}
\rho_3 \leq \frac{1}{2}\Big( 1-\frac{d}{2}+\frac{d}{\zeta}\Big).
\end{align}
The second condition in \eqref{eq: thm local s ineq 3} has to be checked separately, as it does not  allow for a uniform, stricter condition like the above.
\end{itemize}
\end{example}

\begin{example}[Critical Stochastic Fisher-KPP Equation]\label{ex:CriticalFisherKPP}
For the Fisher-KPP equation ($\rho_1=1$), the criticality condition balances at:
\begin{align*}
    s + \frac{d}{q} + 2\frac{1+\kappa}{p} = 3.
\end{align*}
Remarkably, as observed in Remark \ref{rem:rho1}, the structural requirement $s < 1$ forces the condition $\rho_1 > \frac{4}{2+d}$. For our quadratic nonlinearity ($\rho_1=1$), this dictates that $d \geq3$. Therefore, the critical space is mathematically inaccessible within this framework for $d \in \{1, 2\}$, and we restrict our attention to dimensions $d \geq 3$.

Suppose that $g$ is globally Lipschitz. Then, applying Corollary \ref{cor: critical localrho1}, we achieve local well-posedness in the highly singular critical initial data space $\BD_{q,p}^{\frac{d}{q}-2}(\cO)$, provided the parameters satisfy:
\begin{align*}
    1 &< \frac{d}{q} < \min\Big\{1+\frac{d}{\zeta}, \, 3-\frac{d}{2}+\frac{d}{\zeta}\Big\},\\
    2 - \frac{d}{q}&< 2\frac{1+\kappa}{p} < 3-\frac{d}{q}-\frac{d}{2}+\frac{d}{\zeta}.
\end{align*}
The condition $\zeta<\frac{2d}{d-2}$ if $d\geq 2$ is also needed.
Since the initial data regularity is $\frac{d}{q}-2$, taking $q \uparrow d$ (which satisfies the strict lower bound $d/q > 1$) pushes the permissible initial data regularity arbitrarily close to $-1$.

If $g$ is not globally Lipschitz, but of the form considered in this paper with associated growth parameter $\rho_3$, then we need additionally that
\begin{align}
\rho_3 \leq \frac{1}{2}\Big( 1-\frac{d}{2}+\frac{d}{\zeta}\Big)\qquad\text{and}\qquad
     \frac{d}{q}  + 2\frac{1+\kappa}{p}\frac{1}{\rho_3+1}>2.
\end{align}
In this case we see, that the admissible roughness of the initial data depends on $\rho_3$.
\end{example}

\begin{remark}[$L^r$-initial data for Fisher-KPP equation]
\label{rem:L r Fisher}
    As in Remark \ref{rem:L r Allen Cahn}, one can check that 
    for any $\rho_3 \leq \frac{1}{2}\big( 1-\frac{d}{2}+\frac{d}{\zeta}\big)$, $r\geq\max\{\tfrac{d}{2},\,1+\varepsilon,\,\1_{d = 2}\frac{2(1+\rho_3)}{2+\rho_3}+\varepsilon\}$ and $\varepsilon>0$ arbitrarily small, one can find $q,p,s,\kappa$  such that the Sobolev embedding $L^r(\cO)\hookrightarrow \BD_{q,p}^{1-s-2\frac{1+\kappa}{p}}(\cO)$ holds for $d\geq 1$ and such that \eqref{eq:ex_fisher} is well-posed in $L^r(\cO)$. 
\end{remark}

As before for the Allen-Cahn equation, global well-posedness of the above model will be considered in a future work. 

The local theory of the following more complicated model fits into our setting in exactly the same manner. 

\begin{example}[Coupled Chemical Kinetics and Predator-Prey Systems]\label{ex:SystemsKinetics}
Our framework naturally accommodates systems of SPDEs ($\ell \ge 2$), which are essential for modeling interactions between multiple distinct species. Consider a stochastic reaction-diffusion system modeling mass transfer in a bimolecular reaction or a Lotka--Volterra predator-prey dynamic:
\begin{equation}
\left\{
\begin{aligned}
    \dd u - \dv(a_1 \cdot \nabla u) \, \dd t &= (- uv + u) \, \dd t + g_1^1(u,v) \, \dd W^1(t) + g_1^2(u,v) \, \dd W^2(t), \\
    \dd v - \dv(a_2 \cdot \nabla v) \, \dd t &= (-v + uv) \, \dd t + g_2^1(u,v) \, \dd W^1(t) + g_2^2(u,v) \, \dd W^2(t).
\end{aligned}
\right.
\end{equation}
Here, the cross-terms have opposite signs, reflecting the conservation of mass or energy transfer between the species (e.g., the predator $v$ grows at the expense of the prey $u$). This still leads to $\rho_1 = 1$ and $F = 0$.

Because the growth parameter ($\rho_1 = 1$) is mathematically identical to the scalar Fisher-KPP equation (Examples \ref{ex:FisherKPP} and \ref{ex:CriticalFisherKPP}), the exact same parameter restrictions on $s$, $q$, $p$, and $\kappa$ apply. Thus, we obtain the same local theory and critical spaces. Global existence results are considerably more challenging to establish. For the case $\zeta=2$, global well-posedness was recently established in \cite{agresti2024reaction}. Global well-posedness in the case $\zeta>2$ is an interesting open problem. 

Remark \ref{rem:L r Fisher} concerning $L^r$ initial data also applies for the predator-prey system.
\end{example}

\subsection{Burgers type nonlinearities: quadratic $F$}
\label{ss: Burgers}

The setting of Examples \ref{ex:Burgers1} and  \ref{ex:Burgers2} below applies to a large class of nonlinearities in which the nonlinearity 
appears as the first spatial derivative of a quadratic expression in $u$. For example, for many fluid dynamics equations this is the case (e.g.\ Navier-Stokes equations). The  conditions on the parameters do not change when applied in such contexts.

\begin{example}[The Stochastic Burgers Equation]\label{ex:Burgers1}
Next, we consider the stochastic Burgers equation, a fundamental model for turbulent fluid flow:
\begin{align}
\label{eq:ex_burgers}
    \dd u - \Delta u \, \dd t = \nabla \cdot F(u) \, \dd t + g(u) \, \dd W(t).
\end{align}
Here, the nonlinearity $F$ is assumed to be quadratic. The local Lipschitz geometry yields $|F(u)-F(v)| \lesssim (1+|u|+|v|)|u-v|$, which corresponds exactly to our abstract framework with $\rho_2 = 1$ and $f = 0$.
We apply Theorem \ref{thm:mainlocal} with the space of initial values $\BD_{q,p}^{1-s-2\tfrac{1+\kappa}{p}}(\cO)$.

For any dimension $d \geq 1$, the restrictions on $s \in (0, 1)$ (with the added structural constraint $s < 1 - 1/q$ if $d=1$) become:
\begin{align*}
    1+\frac{d}{2q} < s + \frac{d}{q} + 2\frac{1+\kappa}{p} \leq 2.
\end{align*}
For $g$ and the noise we require 
\begin{itemize}
\item $\zeta<\frac{2d}{d-2}$ if $d\geq 2$;
\item  $s>\tfrac{d}{2}-\tfrac{d}{\zeta}$ if $\zeta>2$ ;
\item the condition \eqref{eq: thm local s ineq 3}, which is trivially satisfied for globally Lipschitz $g$, or $\rho_3$ arbitrarily small. 
For general $g$ with associated growth parameter $\rho_3$, a sufficient condition for the first condition in \eqref{eq: thm local s ineq 3} is 
\[\rho_3 \leq 1-\frac{d}{2}+\frac{d}{\zeta}.\]
Again, the second condition in \eqref{eq: thm local s ineq 3} has to be checked separately.
\end{itemize}
\end{example}

\begin{example}[Critical Stochastic Burgers Equation]\label{ex:Burgers2}
To ensure criticality in the Burgers setting ($\rho_2=1$) in dimensions $d \geq 2$, we impose the criticality condition:
\begin{align*}
    s + \frac{d}{q} + 2\frac{1+\kappa}{p} = 2.
\end{align*}
Suppose that $g$ is globally Lipschitz. Then, Corollary \ref{cor: critical local rho2} guarantees well-posedness for initial data in the critical space $\BD_{q,p}^{\frac{d}{q}-1}(\cO)$, provided the parameters satisfy:
\begin{align*}
    \frac{d}{q} < 2-\frac{d}{2}+\frac{d}{\zeta}, \quad  \text{and} \quad 1 - \frac{d}{q} < 2\frac{1+\kappa}{p} < 2-\frac{d}{q}-\frac{d}{2}+\frac{d}{\zeta}.
\end{align*}
The condition $\zeta<\frac{2d}{d-2}$ if $d\geq 2$ is also needed.
Remarkably, as we take the spatial integrability $q \to \infty$, the regularity index $\frac{d}{q} - 1$ approaches $-1$. By carefully selecting the temporal parameters $(p, \kappa)$ within the valid range, our framework accommodates initial data whose regularity is arbitrarily close to $-1$.

If $g$ is not globally Lipschitz, but of the form considered in this paper with associated growth parameter $\rho_3$, then we need additionally
\begin{align}
\rho_3 \leq 1-\frac{d}{2}+\frac{d}{\zeta}\qquad\text{and}\qquad
     \frac{d}{q}  + 2\frac{1+\kappa}{p}\frac{1}{\rho_3+1}>1.
\end{align}
In this case again, the lowest admissible roughness of the initial data depends on $\rho_3$.

Since equation \eqref{eq:ex_burgers} scales exactly like \eqref{eq:ex_allen_cahn}, we see that Remark \ref{rem:L r Allen Cahn} concerning $L^r$ initial data also applies to \eqref{eq:ex_burgers}.
\end{example}

Global well-posedness in the above flexible setup will be considered in a future paper.

\subsection{A class of equations with space-time white noise}\label{ss:applwhite}

In this subsection we apply our results to the following equation in $d=1$:
\begin{equation}\label{eq:heatwhite}
\left\{
\begin{aligned}
\dd u - \Delta u \,\dd t &= f(u) \,\dd t + g(u) \,\dd W(t),\\
u(0) &= u_{0}.
\end{aligned}\right.
\end{equation}

\begin{assumption}\label{ass:fsuper0spacetimewhite}
Let $\cO$ be a bounded interval. Let $\rho\in (2, \infty)$. Suppose that $f,g:\R\to \R$ satisfy
\begin{align*}
 |f(y)-f(y')| &\lesssim (1+|y|^{\rho}+|y'|^{\rho})|y-y'|,
\\ |g(y)-g(y')| &\lesssim (1+|y|^{\rho/4}+|y'|^{\rho/4})|y-y'|,
  \end{align*}
  and $F=0$. Suppose that $W$ is a space-time white noise, i.e.\ $R = I$.
\end{assumption}
If necessary, one can replace $\rho$ by a larger number without loss of generality. Thus, if $f,g$ admit the above bounds with a $\rho\in (0,2]$, then we just replace it by $2+\delta$ for $\delta>0$ small. The special relation between $f$ and $g$ will play a role in the argument below.

\begin{theorem}
\label{thm:maximal_initial_data_rho1_rho3}
Suppose that Assumption \ref{ass:fsuper0spacetimewhite} holds.
Then, for every $\varepsilon\in(0,\frac12 - \frac{2}{\rho(\rho + 1)})$, letting the integrability parameter $q_\varepsilon \in (2, \infty)$ be defined by
\begin{equation}\label{eq: q_eps_def}
    \frac{1}{q_\varepsilon} = \frac{2}{\rho(\rho + 1)} + \varepsilon,
\end{equation}
there exists an explicit choice of parameters $p_\varepsilon > 2$ and $\kappa_\varepsilon \in [0, p_\varepsilon/2 - 1)$ and $s_{\varepsilon}\in (\frac{1}{2},1-\frac{1}{q_\varepsilon})$ such that for initial value 
\begin{equation}
 u_0\in  L^0_{\cF_0}(\Omega,\BD^{-\frac{2}{\rho + 1} + \varepsilon}_{q_\varepsilon, p_\varepsilon}(\cO)), 
\end{equation}
there exists a maximal $(s_\varepsilon,q_\varepsilon,p_\varepsilon,\kappa_\varepsilon)$-solution $(u,\sigma)$. Moreover, one has that, a.s.,
\begin{align*}
u&\in C([0,\sigma);\BD^{-\frac{2}{\rho + 1} + \varepsilon}_{q_\varepsilon, p_\varepsilon}(\cO)),\\
u&\in C^{\theta_1,\theta_2}_{\loc}((0,\sigma)\times \overline{\cO}) \ \ 
        \text{for $\theta_1\in [0,1/4), \ \theta_2\in [0,1/2)$},
\end{align*}
$u(0)= u_0$, and $u = 0$ on $(0,\sigma)\times \partial \cO$.
\end{theorem}
We point out that the space of initial values is critical. Moreover, for $\varepsilon\downarrow 0$ we have $q_{\varepsilon}\uparrow \frac{\rho(\rho + 1)}{2}$.

\begin{proof}
We apply Theorem~\ref{thm:mainlocal zeta infty}\eqref{it2:mainlocal zeta infty}
with $\rho_1=\rho$ and $\rho_3=\frac{\rho}{4}$. 
Define
\begin{align}
    s_\varepsilon :=\frac{1}{2}\Big(\max\Big\{\frac{1}{2},\,\frac{2}{\rho}-\frac{1}{q_\varepsilon}\big\} + \min\Big\{1-\frac{1}{q_\varepsilon},\,\frac{1}{2}+\frac{\rho}{4q_\varepsilon}\Big\}\Big).
\end{align}
Then, in particular $\frac{1}{2}<s_\varepsilon<1-\frac{1}{q_\varepsilon}$.
Set
\[
    p_\varepsilon :=
    \max\Big\{
    \frac{2}{1+2/\rho-1/q_\varepsilon-s_\varepsilon},
    \frac{\rho}{2}
    \Big\},
    \qquad
    \kappa_\varepsilon :=
    \frac{p_\varepsilon}{2}
    \Big(1+\frac{2}{\rho}-\frac{1}{q_\varepsilon}-s_\varepsilon\Big)-1.
\]
which coincides with the critical relation in
Theorem~\ref{thm:mainlocal zeta infty}\eqref{it2:mainlocal zeta infty}. It is easy to check that then
$2<p_\varepsilon$ and
$0\leq\kappa_\varepsilon<p_\varepsilon/2-1$.

The corresponding trace exponent is
\[
    1-s_\varepsilon-2\frac{1+\kappa_\varepsilon}{p_\varepsilon}
    =
    \frac{1}{q_\varepsilon}-\frac{2}{\rho}
    =
    -\frac{2}{\rho+1}+\varepsilon.
\]
Thus, the initial-value space is as stated in the theorem. One can check that the chosen parameters satisfy the conditions of Theorem~\ref{thm:mainlocal zeta infty}, which then provides the result.
\end{proof}

\begin{remark}[$L^r(\cO)$-initial data for 1D space-time white noise]
\label{rem:L r stwn}
    Again we see that for any $r\geq \tfrac{\rho}{2}$ we may find $(s,q,p,\kappa)$ such that $L^r\hookrightarrow \BD^{-\frac{2}{\rho + 1} +\varepsilon }_{q_\varepsilon, p_\varepsilon}(\cO)$ and such that \eqref{eq:heatwhite} is well-posed in $L^r(\cO)$ (note that, for $r=\rho/2$, we choose $\varepsilon>0$
sufficiently small so that $q_\varepsilon>\rho/2$ and $p_\varepsilon\geq\rho/2$ ensure the needed embedding).
\end{remark}

In our follow-up paper we will study sufficient conditions for global well-posedness for the typical physical nonlinearities. The following global result is somewhat separate from our future work. Therefore, it is more natural to present it here already. For H\"older continuous initial data, the result below was already proved in \cite[Theorem 1.9]{DKZ19}. 
Our instantaneous regularisation result allows us to extend
this global-existence conclusion to the rough initial data of
Theorem~\ref{thm:maximal_initial_data_rho1_rho3}.
For $L^2(\cO)$ initial data, \cite{FKN25} proves well-posedness when
the drift is locally Lipschitz with the stated $L\log L$ growth and
the noise coefficient is globally Lipschitz and bounded; see also
Remark~\ref{rem:cmoparisonwithearlierwork}.

\begin{corollary}[Special cases with global existence and uniqueness]
Suppose that the conditions of Theorem \ref{thm:maximal_initial_data_rho1_rho3} are satisfied. Let $(u,\sigma)$ be the maximal $(s_\varepsilon,q_\varepsilon,p_\varepsilon,\kappa_\varepsilon)$-solution given there. 
If $|f(y)| = O(|y|\log(|y|))$ and $g(y)=o(|y|\log(|y|)^{1/4})$ for $|y|\to 
\infty$, then $u$ is the unique global $(s_\varepsilon,q_{\varepsilon},p_\varepsilon,\kappa_\varepsilon)$-solution, i.e.\ $\sigma=\infty$. 
\end{corollary}

\begin{proof}
Next we establish global well-posedness under the additional growth conditions on $f$ and $g$. Choose an arbitrary time $\delta > 0$. By the instantaneous regularization property established above, almost surely on the event $\{\sigma > \delta\}$, the solution is Hölder continuous and satisfies Dirichlet boundary conditions, meaning $u(\delta) \in \CD^\alpha(\overline{\cO})$ for some $\alpha \in (0, 1)$. 
Consider the problem
\begin{equation}\label{eq:heatwhite-delta}
\left\{
\begin{aligned}
\dd v - \Delta v \,\dd t &= f(v) \,\dd t + g(v) \,\dd W(t),\\
v(\delta) &= \one_{\sigma>\delta} u(\delta).
\end{aligned}\right.
\end{equation}
Clearly $\one_{\sigma>\delta} u(\delta)\in \CD^\alpha(\overline{\cO})$ and is strongly $\mathscr{F}_{\delta}$-measurable. 
Because $f$ and $g$ are locally Lipschitz continuous and satisfy the asymptotic growth bounds $|f(y)| = O(|y|\log|y|)$ and $|g(y)| = o(|y|(\log|y|)^{1/4})$, all hypotheses of \cite[Theorem 1.9]{DKZ19} are satisfied. By their result which extends to random initial data, the stochastic reaction-diffusion equation starting from $t=\delta$ admits a unique, global random field solution $v\in C([\delta,\infty)\times\overline{\cO})$ with $v = 0$ on $\partial \cO$. Since $\one_{\sigma>\delta} u$ is a random field solution to \eqref{eq:heatwhite-delta} on $[\delta,\sigma)$, the uniqueness in \cite[Theorem 1.9]{DKZ19} implies that $u = v$ on $[\delta,\sigma)$  on $\{\sigma>\delta\}$. In particular, $u$ does not blow up at $\sigma$, and thus 
\begin{align*}
\P(\delta<\sigma<\infty) &= \P\Big(\delta<\sigma<\infty, \,  \sup_{t\in [\delta,\sigma)} \|u(t)\|_{C(\cO)}<\infty \Big) = 0,
\end{align*}
where in the last step we used Theorem \ref{thm:mainlocal zeta infty} \eqref{it4:mainlocal zeta infty}. Letting $\delta\downarrow 0$ and using that $\sigma>0$ a.s.\ we obtain that $\sigma=\infty$ a.s.
\end{proof}

\begin{remark}\label{rem:cmoparisonwithearlierwork}
The polynomial local Lipschitz bounds in
Assumption~\ref{ass:fsuper0spacetimewhite} provide a local theory for
rough initial data. Under the additional growth assumptions of the
preceding corollary, instantaneous regularisation and restarting allow
us to apply \cite[Theorem~1.9]{DKZ19} and obtain global well-posedness.
Thus our method permits both rough initial data and unbounded noise
coefficients, at the cost of quantitative restrictions on the local
Lipschitz constants.

The admissible initial integrability depends on the exponent $\rho$:
with suitable auxiliary parameters, $L^r(\cO)$ initial data are included
for $r\geq\rho/2$. In particular, the present argument applies to
$L^2$ initial data when $2<\rho\leq4$, and to a fixed $L^r$ space with
$1<r<2$ when $2<\rho\leq2r$.
If the drift has a local Lipschitz constant of at most quadratic
growth and the noise coefficient is globally Lipschitz, we may take
$\rho>2$ arbitrarily close to $2$, thereby including every fixed
$r>1$. Subject to the growth assumptions of the corollary,
these give partial positive answers, at the level of well-posedness,
to Problems~1.4 and~1.5 of \cite{FKN25} in the indicated subclass.
The general locally Lipschitz case is not covered by this argument.
A bounded example outside our polynomial local Lipschitz assumptions
is $f(y)=\sin(e^y)$.
\end{remark}

\section{Proofs of the main results}\label{sec:proofs}

\subsection{Abstract stochastic evolution equations in critical spaces}\label{sec:loc-well-posed}

In this subsection, we recall existence and uniqueness results for the stochastic evolution equation from \cite{AV19_QSEE_1,AV19_QSEE_2, AVsurvey,AV-ob}. 
Let $X_0$ and $X_1$ be UMD Banach spaces with type $2$ such that $X_1$ embeds continuously and densely into $X_0$. On $X_0$ we consider the following equation
\begin{equation}
\label{eq:SEE}
\left\{
\begin{aligned}
&\dd u + A u \,\dd t = [\Phi_1(\cdot, u) + \Phi_2(\cdot, u)+\phi]\,\dd  t + [\Psi(\cdot, u) +\psi] \,\dd  W_\mathcal{U}(t),\\
&u(0)=u_{0}.
\end{aligned}
\right.
\end{equation}
Here $A$ is the leading order (differential) operator, $A:X_1\mapsto X_0$. 
The nonlinearities $\Phi_1, \Phi_2$ and $\Psi$ are assumed to be locally Lipschitz, but are allowed to have critical forms. 
The functions $\phi,\psi$ model inhomogeneities and $W_\mathcal{U}$ is an $\mathcal{U}$-cylindrical Brownian motion.

The natural regularity space for the solution to \eqref{eq:SEE} is
\begin{equation}
\label{eq:max_reg_space}
L^p(0,T,w_\kappa;X_1)\cap C([0,T];X_{1-\frac{1+\kappa}{p},p}),
\end{equation}
where we denote $w_{\kappa}^a(t) = (t-a)^{\kappa}$, for $a\geq 0$, and $\kappa\in [0,\frac{p}{2}-1)$ in case $p>2$. If $a=0$ we simply write $w_\kappa^a = w_\kappa$. 

For the leading operator $A$ we will need a form of stochastic maximal $L^p$-regularity, \cite{AVsurvey}. However, since the (non)linearities in the noise term are of lower order, transference results (see \cite{AV-ob} and \cite[Theorem 3.9]{VP18}) imply that we can reduce this to two deterministic conditions: 
\begin{itemize}
\item pathwise maximal $L^p$-regularity of $A$ with uniform constants in $\Omega$;
\item the existence of a  reference operator $A_0$ with a  bounded $H^\infty$-calculus of angle $<\pi/2$.
\end{itemize}
For details on the $H^\infty$-calculus the reader is referred to  \cite{Analysis2}. In particular, the calculus implies that $A_0$ has stochastic maximal $L^p$-regularity if $X_0$ is isomorphic to a closed subspace of $L^q$ (see \cite{MaximalLpregularity}). In applications to second order equations, one can typically take $A_0 = -\Delta$, in case of Dirichlet boundary conditions. If $A$ does not depend on $(t,\omega)$, one could just take $A_0=A$ in most situations. 

For completeness and later reference we include a definition of maximal $L^p$-regularity for time-dependent families.
\begin{definition}[Maximal $L^p$-regularity]\label{def:MaxReg}
Let $p\in (1, \infty)$ and $\kappa\in (-1,p-1)$. A strongly measurable operator family $L:[0,T]\to \calL(X_1, X_0)$ is said to have maximal $L^p(0,T,w_{\kappa})$-regularity if for every $0\leq a<b\leq T$ and any $f\in L^p(a,b,w_{\kappa}^a;X_0)$ there exists a unique $u\in L^p(a,b,w_{\kappa}^a;X_1)$ such that for all $t\in [a,b]$, 
\[u(t)  + \int_a^t L(s) u(s)  \,\dd  s= \int_a^t f(s) \,\dd  s\]
and there is a constant $C_T$ independent of $a,b$ and $f$ such that 
\[\|u\|_{L^p(a,b,w_{\kappa}^a;X_1)}\leq C_T\|f\|_{L^p(a,b,w_{\kappa}^a;X_0)}.\]
\end{definition}

\begin{assumption}\label{ass:FGcritical}
Suppose that $X_0$ is isomorphic to a closed subspace of an $L^q$-space with $q\geq 2$. Let $X_1$ be such that $X_1\hookrightarrow X_0$ densely.
  Let $p\in (2, \infty)$, $\kappa\in [0,p/2-1)$.
Suppose that 
\begin{itemize}
    \item $A:[0,\infty)\times\Omega\to \calL(X_1, X_0)$ is strongly $\Progress$-measurable and for every $T<\infty$, there exists a constant $C_T$ such that for all $\omega\in \Omega$, $\sup_{t\in [0,T]} \|A(t,\omega)\|_{\calL(X_1, X_0)}\leq C_T$ and for all $r\in (2, \infty)$ and $\delta\in [0,r/2-1)$, $A(\cdot,\omega)$ has maximal $L^r(0,T,w_{\delta})$-regularity with constant $\leq C_{T,r,\delta}$.
\item {\rm (Reference operator)}
there exists $A_0:\mathsf{D}(A_0)=X_1\subseteq X_0\to X_0$ which has a bounded $H^\infty$-calculus of angle $<\pi/2$.
\end{itemize}
Let $(p_f,\alpha_f)$ be such that $\tfrac{1}{p_f} = \tfrac{1}{p}+\tfrac{\alpha_f}{1+\kappa}$. Assume that 
$$
\Phi_1, \Phi_2:[0,\infty)\times \Omega\times X_1\to X_0\quad  \text{ and } \quad \Psi:[0,\infty)\times \Omega\times X_1\to \gamma(\mathcal{U},X_{\frac12})
$$
are $\mathcal{P}\otimes \mathcal{B}(X_1)$-measurable and that the following local Lipschitz conditions are satisfied: 

$\Phi_1(\cdot,0) = \Phi_2(\cdot, 0)=0$ and $\Psi(\cdot, 0) = 0$ and 
there exists $L$ such that for all $u,v\in X_1$,
\begin{align}
\|\Phi_1(\cdot, u) - \Phi_1(\cdot, v)\|_{X_{\alpha_1}}
&\leq L (1+\|u\|_{X_{\beta_1}}^{\rho_1} + \|v\|_{X_{\beta_{1}}}^{\rho_1}) \|u-v\|_{X_{\beta_{1}}},
\\
\label{eq: abstract coefficient estimates}
\|\Phi_2(\cdot, u) - \Phi_2(\cdot, v)\|_{X_{\alpha_2}}& \leq L (1+\|u\|_{X_{\beta_2}}^{\rho_2} + \|v\|_{X_{\beta_{2}}}^{\rho_2}) \|u-v\|_{X_{\beta_{2}}},
\\ \|\Psi(\cdot, u) - \Psi(\cdot, v)\|_{\gamma(\mathcal{U},X_\frac12)}& \leq 
L (1+\|u\|_{X_{\beta_3}}^{\rho_3} + \|v\|_{X_{\beta_{3}}}^{\rho_3}) \|u-v\|_{X_{\beta_{3}}},
\end{align}
where $\alpha_1,\alpha_2\in [0,1-\frac{1+\kappa}{p})$, $\beta_1,\beta_2, \beta_3 \in  (1-\frac{1+\kappa}{p},1)$, and $\rho_1,\rho_2, \rho_3>0$ satisfy
\begin{align}\label{eq:subcritical}
\frac{1+\kappa}{p} & \leq \frac{(1+\rho_1)(1-\beta_1)+\alpha_1 }{\rho_1},
\\ \label{eq:subcritical2}
\frac{1+\kappa}{p} & \leq \frac{(1+\rho_2)(1-\beta_2)+\alpha_2 }{\rho_2},
\\ \label{eq:subcritical3} \frac{1+\kappa}{p}&\leq \frac{(1+\rho_3)(1-\beta_3)}{\rho_3}. 
\end{align}
Finally, suppose that $\phi\in L^0(\Omega;L^p_{\rm loc}([0,\infty),w_{\kappa};X_0))$ and $\psi\in L^0(\Omega;L^p_{\rm loc}([0,\infty),w_{\kappa};\gamma(\mathcal{U},X_\frac12)))$ are progressively measurable inhomogeneities, and that the initial data $u_0:\Omega\to X_{1-\frac{1+\kappa}{p},p}$ is strongly $\cF_0$-measurable.
\end{assumption}
The conditions in Assumption \ref{ass:FGcritical} imply that $A$ has stochastic maximal $L^p$-regularity on bounded time-intervals with weight $t^{\kappa}$ (see \cite{MaximalLpregularity} and \cite[Theorem 3.9]{VP18}).

\smallskip

The couple $(p,\a)$ or the setting $(X_0,X_1,p,\a)$ will be called {\em critical} (resp., {\em subcritical}) for \eqref{eq:SEE} if at least one of the \eqref{eq:subcritical}, \eqref{eq:subcritical2}, \eqref{eq:subcritical3} holds with equality (resp., if strict inequality holds in each \eqref{eq:subcritical}, \eqref{eq:subcritical2}, \eqref{eq:subcritical3}). Similar terminology extends naturally to the trace space $X_{1-\frac{1+\kappa}{p},p}$.

\begin{definition}
Suppose Assumption \ref{ass:FGcritical} is satisfied for some $p\in (2, \infty)$ and $\kappa\in [0, p/2-1)$. A pair $(u,\sigma)$ is called an {\em $L^p_{\kappa}$-strong solution} of \eqref{eq:SEE} if $\sigma:\Omega\to [0,\infty)$ is a stopping time and $u:[0,\sigma]\to X_0$ is a strongly progressively measurable process such that
\[u\in L^p(0,\sigma,w_{\kappa};X_1)\cap C([0,\sigma];X_{1-\frac{1+\kappa}{p},p}),\]
and the following identity holds a.s.\ for all $t\in [0,\sigma]$:
\begin{equation}\label{eq:stronsol}
\begin{aligned}
u(t) - u_0 + \int_0^t A u(s) \,\dd s = & \int_0^t \Phi_1(s,u(s)) + \Phi_2(s,u(s)) + \phi(s) \,\dd s 
\\ & + \int_{0}^t \one_{[0,\sigma]}(s) \Psi (s,u(s)) +\psi(s) \,\dd  W(s).
\end{aligned}
\end{equation}
\end{definition}
Let $\frac{1}{p_j} = \frac{1}{p} + \frac{\alpha_j}{1+\kappa}$ for $j\in \{1, 2\}$. Using interpolation estimates (see \eqref{eq:shifted-interpolation} below) one can check that for \[u\in L^p(0,\sigma,w_{\kappa};X_1)\cap C([0,\sigma];X_{1-\frac{1+\kappa}{p},p}),\] one has
\begin{align*}\Phi_1(\cdot,u)&\in L^{p_1}(0,\sigma,w_{\kappa};X_{\alpha_1}),\\  
\Phi_2(\cdot,u)&\in L^{p_2}(0,\sigma,w_{\kappa};X_{\alpha_2}), 
\\ \Psi(\cdot, u) & \in L^p(0,\sigma,w_{\kappa};\gamma(\mathcal{U},X_{1/2}))
\end{align*}
Moreover, this implies that all the integrals are well-defined in $X_0$.

\begin{definition}\label{def:localsol}
Suppose Assumption \ref{ass:FGcritical} is satisfied for some $p\in (2, \infty)$ and $\kappa\in [0, p/2-1)$.
\begin{enumerate}[{\rm(1)}]
\item\label{it1:localsol} A pair $(u,\sigma)$ is called an {\em $L^p_{\kappa}$-local solution} to \eqref{eq:SEE} if $\sigma:\Omega\to [0,\infty]$ is a stopping time and $u:[0,\sigma)\to X_0$ is a strongly progressively measurable process, and there exists an increasing sequence of stopping times $(\sigma_{n})_{n\geq 1}$ such that $\lim_{n\to \infty} \sigma_n = \sigma$ a.s., and $(u|_{[0,\sigma_n]}, \sigma_n)$ is a $L^p_{\kappa}$-strong solution to \eqref{eq:SEE}. The sequence $(\sigma_{n})_{n\geq 1}$  is called a {\em localizing sequence} for $(u, \sigma)$.
\item\label{it2:localsol} An $L^p_{\kappa}$-local solution $(u,\sigma)$ to \eqref{eq:SEE} is called {\em unique} if for every $L^p_{\kappa}$-local solution $(v,\tau)$ one has that a.s.\ $u=v$ on $[0,\sigma\wedge \tau)$.
\item\label{it3:localsol} An $L^p_{\kappa}$-local solution $(u,\sigma)$ to \eqref{eq:SEE} is called {\em $L^p_{\kappa}$-maximal} if for any other unique $L^p_{\kappa}$-local solution $(v, \tau)$ to \eqref{eq:SEE} one has that a.s.\ $\tau\leq \sigma$ and $u = v$ on $[0,\tau)$.
\item An $L^p_{\kappa}$-local solution $(u,\sigma)$ of \eqref{eq:SEE} is called {\em global} if $\sigma=\infty$ a.s.
\end{enumerate}
\end{definition}

The following theorems are results from the forthcoming paper \cite{AV-ob}, which are improvements of analogous results in \cite{AV19_QSEE_1} and \cite{AV19_QSEE_2}. After the local existence theorem, we outline the modifications to its proof required by allowing the more flexible indices $\alpha_j$.

\begin{theorem}[Local existence and uniqueness for evolution equations]
\label{thm:local}
Suppose that Assumption \ref{ass:FGcritical} holds. Let $\theta^* = \min\{1-\alpha_1, 1-\alpha_2, \frac{1}{2}\}$. Then for every $u_0\in L^0_{\cF_0}(\Omega;X_{1-\frac{1+\kappa}{p},p})$, there exists an $L^p_{\kappa}$-maximal solution $(u,\sigma)$ to \eqref{eq:SEE} with $\sigma>0$ a.s. Moreover, the following properties hold:
\begin{enumerate}[{\rm(1)}]
\item\label{it1:localwellposedabs} {\em (Regularity)} For each localizing sequence $(\sigma_n)_{n\geq 1}$ for $(u,\sigma)$ one has
for all $n\geq 1$ and all $\theta\in [0,\theta^*)$
  \[u\in H^{\theta,p}(0,\sigma_n,w_{\kappa};X_{1-\theta})\cap C([0,\sigma_n];X_{1-\frac{1+\kappa}{p},p}) \ \ \text{a.s.},\]
  and one has the following instantaneous regularity
  \begin{align}\label{eq:instant}
  u\in H^{\theta,p}_{\rm loc}((0,\sigma);X_{1-\theta})\cap C((0,\sigma);X_{1-\frac{1}{p},p}) \ \ \ \text{a.s.}
  \end{align}
\item\label{it2:localwellposedabs} {\em (Localization)} Let $v_0\in L^0_{\cF_0}(\Omega;X_{1-\frac{1+\kappa}{p},p})$. If $(v,\tau)$ is the $L^p_{\kappa}$-maximal solution to \eqref{eq:SEE} with initial value $v_0$, then a.s.\ on the set $\{u_0=v_0\}$ one has $\tau=\sigma$ and $u = v$ on $[0,\sigma\wedge \tau)$.
\end{enumerate}
\end{theorem}

The proof of the above is via a fixed point argument, which follows almost verbatim from the one in \cite{AV19_QSEE_1} where the case $\alpha_j=0$ is treated in full detail. In the notation of the latter paper we additionally have $\varphi_j = \beta_j$, and $B = 0$. To cover the case of $F = \Phi_1+\Phi_2$ with $\alpha_j\neq 0$, the key interpolation estimate (see e.g.\ \cite[Lemma 4.3]{AVsurvey}) should be replaced by 
\begin{align}\label{eq:shifted-interpolation}
\|u\|_{L^{(\rho_j+1)p_j}(0,\sigma,w_{\kappa};X_{\beta_j})}\lesssim \|u\|_{L^\infty(0,\sigma;X_{1-\frac{1+\kappa}{p},p})}^{1-\theta_j} \|u\|_{L^p(0,\sigma;w_{\kappa};X_1)}^{\theta_j} \ \ \text{with $\theta_j\in (0,1)$}.
\end{align}
In light of \eqref{eq: abstract coefficient estimates}, \eqref{eq:subcritical} and \eqref{eq:subcritical2}, the above ensures $\Phi_j(u)\in L^{p_j}(0,\sigma,w_{\kappa};X_{\alpha_j})$ and the following locally Lipschitz estimate 
\begin{align*}
&\|\Phi_j(u) - \Phi_j(v)\|_{L^{p_j}(0,\sigma,w_{\kappa};X_{\alpha_j})} \\ &\leq C L (1+\|u\|_{L^{(\rho_j+1)p_j}(0,\sigma,w_{\kappa};X_{\beta_j})}^{\rho_j} + \|v\|_{L^{(\rho_j+1)p_j}(0,\sigma,w_{\kappa};X_{\beta_j})}^{\rho_j} )\|u-v\|_{L^{(\rho_j+1)p_j}(0,\sigma,w_{\kappa};X_{\beta_j})} .
\end{align*}
These are the key bounds which can be combined with parabolic smoothing for the deterministic parts, and it leads to Theorem \ref{thm:local} as in \cite{AV19_QSEE_1}.  
The corresponding linear input is a bounded operator from these shifted forcing spaces (i.e.\ defined on $X_{\alpha_j}$ rather than $X_0$) into
\[
 L^p(0,\sigma,w_\kappa;X_1)
 \cap C([0,\sigma];X_{1-\frac{1+\kappa}{p},p}),
\]
with constants uniform for $\sigma\leq T$.
For the reference operator $A_0$, this is a parabolic smoothing
estimate. To transfer it to $A$, let $z$ solve the reference problem
and solve the correction equation
\[
 w'+Aw=(A_0-A)z,\qquad w(0)=0.
\]
Since $A_0-A$ is bounded from
$L^p(0,\sigma,w_\kappa;X_1)$ into
$L^p(0,\sigma,w_\kappa;X_0)$, maximal $L^p$-regularity of $A$
controls this correction. The stochastic forcing is transferred in
the same way after applying stochastic maximal regularity to $A_0$.
This is the transference mechanism used in
\cite[Theorem~3.9]{VP18} and illustrated in
\cite[Proposition~4.5]{TV26}.
The fixed-point construction then uses the shifted nonlinear
estimates above in place of their unshifted counterparts.

\begin{theorem}[Blow-up criteria I: subcritical case]
\label{thm: blow up subcritical}
Suppose that the conditions of Theorem \ref{thm:local} hold and that \eqref{eq:subcritical}, \eqref{eq:subcritical2} and \eqref{eq:subcritical3} hold with strict inequalities. Let $(u,\sigma)$ be the maximal $L^p_{\kappa}$-solution to \eqref{eq:SEE}. For all $r\geq 0$, 
\[\P\Big(r<\sigma<\infty,\, \sup_{t\in [r,\sigma)} \|u(t)\|_{X_{1-\frac{1+\kappa}{p},p}}<\infty \Big) = 0.\]
\end{theorem}

\begin{theorem}[Blow-up criteria II: critical case]
\label{thm: blow up critical 2}
Suppose that the conditions of Theorem \ref{thm:local} hold. Let $(u,\sigma)$ be the maximal $L^p_{\kappa}$-solution to \eqref{eq:SEE}. For all $r\geq 0$, 
\[\P\Big(r<\sigma<\infty,\, \sup_{t\in [r,\sigma)} \|u(t)\|_{X_{1-\frac{1+\kappa}{p},p}} + \|u\|_{L^p(r,\sigma;X_{1-\frac{\kappa}{p}})}<\infty \Big) = 0.\]
\end{theorem}

The above blow-up criteria have been proven in the cases in \cite[Sections 4-6]{AV19_QSEE_2} (see also \cite[Section 5]{AVsurvey}) in the case $\alpha_j \equiv 0$. However, some modifications are needed in the case $\alpha_j \neq0$. For instance, in the case of Theorem \ref{thm: blow up subcritical}, the proof of \cite[Theorem 5.2]{AVsurvey} does not directly extend. However, arguing as in \eqref{eq:shifted-interpolation}, one can check that there exists $\delta>0$ such that, for all $t>0$,
$$
\|v\|_{L^{(\rho_j+1)p_j}(0,t,w_{\kappa};X_{\beta_j})}\lesssim t^\delta\big( \|v\|_{L^\infty(0,t;X_{1-\frac{1+\kappa}{p},p})}+  \|v\|_{L^p(0,t;w_{\kappa};X_1)} \big),
$$
for $v$ for which the right-hand side is finite. The explicit rate $\delta>0$, due to the subcriticality assumption, allows one to repeat the argument of \cite[Theorem 4.9(2)]{AV19_QSEE_2} in this setting. The case of Theorem \ref{thm: blow up critical 2} is analogous, but instead of using small times, one exploits that $\lim_{t\uparrow \sigma}\|u\|_{L^p(t,\sigma;X_{1-\a/p})}=0$ if $\|u\|_{L^p(0,\sigma;X_{1-\a/p})}<\infty$; see the proof of \cite[Theorem 4.9(3)]{AV19_QSEE_2}.

\smallskip

Even in this abstract framework it is possible to obtain the following powerful temporal regularization result. In Subsection \ref{ss:regbootstrap} we show that with this result one can bootstrap space regularity.

\begin{theorem}[Temporal regularization]
\label{thm: temporal regularization}
Suppose that the conditions of Theorem \ref{thm:local} hold with $p\in (2, \infty)$ and $\kappa\in [0,p/2-1)$ and let $(u,\sigma)$ be the $L^p_{\kappa}$-maximal solution to \eqref{eq:SEE}. Let $\theta^*= \min\{1/2, 1-\alpha_1, 1-\alpha_2\}$. Suppose that $\phi\in L^0(\Omega;L^\infty_{\loc}([0,\infty);X_0))$ and $\psi\in L^0(\Omega;L^\infty_{\rm loc}([0,\infty);\gamma(\mathcal{U},X_\frac12)))$. Then the following path regularity holds a.s.\
\begin{align*}
u\in H^{\theta,r}_{\rm loc}((0,\sigma);&X_{1-\theta})\cap C^{\theta-\varepsilon}_{\rm loc}((0,\sigma);X_{1-\theta}), \ \ \theta\in [0,\theta^*), r\in [2, \infty), \varepsilon\in (0,\theta).
  \end{align*}
\end{theorem}
This result can be proved in a similar way as in \cite[Theorems 5.6 and 5.7]{AVsurvey}. The restriction on $\theta$ comes from applying the linear deterministic maximal regularity estimates as was the case in \cite[Section 3.5]{TV26}.

\begin{proposition}[Local continuity]
\label{prop:local_continuity_abstract}
Let $u_0\in L^p_{\cF_0}(\O;X_{1-\frac{1+\kappa}{p},p})$, and let $(u,\sigma)$ be the maximal solution to \eqref{eq:SEE}. Then there exist constants $C_0,T_0,\varepsilon_0>0$, a stopping time $\sigma_0\in (0,\sigma]$ a.s., and a family of constants $(R(t))_{t\in [0,T_0]}$ such that
$\lim_{t\downarrow 0} R(t)=0$ and for which the following assertion holds:
For all 
$$
v_0\in L^p_{\cF_0}(\O;X_{1-\frac{1+\kappa}{p},p})\qquad \text{ with }\qquad 
\E\|u_0-v_0\|_{X_{1-\frac{1+\kappa}{p},p}}^p\leq \varepsilon_0,
$$
the maximal solution $(v,\tau)$ to \eqref{eq:SEE} with initial data $v_0$ has the property that there exists a stopping time $\tau_0\in (0,\tau]$ a.s.\ such that, for all $t\in [0,T_0]$ and $\g>0$, it holds that 
\begin{align*}
\P\Big(\sup_{s\in [0,t]}\|u(s)-v(s)\|_{X_{1-\frac{1+\kappa}{p},p}}\geq \g ,\, \sigma_0\wedge \tau_0> t \Big)
&\leq \frac{C_0}{\g^p} \E\|u_0-v_0\|_{X_{1-\frac{1+\kappa}{p},p}}^p,\\
\P\Big(\|u-v\|_{L^p(0,t,w_{\kappa};X_{1})}\geq \g ,\, \sigma_0\wedge \tau_0> t \Big)
&\leq \frac{C_0}{\g^p} \E\|u_0-v_0\|_{X_{1-\frac{1+\kappa}{p},p}}^p,\\
\P(\sigma_0\wedge \tau_0 \leq t) 
&\leq 
C_0 \E\|u_0-v_0\|_{X_{1-\frac{1+\kappa}{p},p}}^p+R(t).
\end{align*}
\end{proposition}
A similar local continuity result can be found in \cite[Proposition 4.8]{AVsurvey} and the proof follows verbatim.

\begin{lemma}[Compatibility of solutions to linear equations]
\label{lem:compatibility_solutions}
Suppose that Assumption \ref{ass:FGcritical} holds. 
Let $p_i > 2$ and $\kappa_i \in [0, p_i/2 - 1)$, $i=0,1$, and 
assume $\phi \in L^{p_i}(a,\sigma,w_{\kappa_i}^a,X_0)$ and $\psi\in L^{p_i}(a,\sigma,w_{\kappa_i}^a,\gamma(\mathcal{U},X_{1/2}))$, for $i=0,1$.
Let $(u^i, \sigma)$, $i=0,1$, be $L^{p_i}_{\kappa_i}$-strong solutions to 
\begin{equation}
\left\{
\begin{aligned}
&\dd u + A u \,\dd t = \phi\,\dd  t + \psi \,\dd  W,\\
&u(a)=u_{a}.
\end{aligned}
\right.
\end{equation}
Then $u^0=u^1$ almost surely on $[a, \sigma]$.
\end{lemma}

\begin{proof}
To prove that $u^0$ and $u^1$ coincide, we set for simplicity $a:=0$ and show that $L^{p_i}(0, \sigma, w_{\kappa_i}; X_1)$ embeds into an unweighted space $L^p(0, \sigma; X_1)$ for a sufficiently small choice of $p > 2$. 

Using H\"older's inequality one can check that this holds for exponents $p$ such that
$$2 < p < \min\left\{ \frac{p_0}{1+\kappa_0}, \frac{p_1}{1+\kappa_1} \right\}.$$
Indeed, a simple calculation gives
$$\int_0^\sigma \|u^i(t)\|_{X_1}^{p} \diff t = \int_0^\sigma \|u^i(t)\|_{X_1}^{p} t^{\frac{\kappa_i p}{p_i}} t^{-\frac{\kappa_i p}{p_i}} \diff t \leq \left( \int_0^\sigma \|u^i(t)\|_{X_1}^{p_i} t^{\kappa_i} \diff t \right)^{\frac{p}{p_i}} \left( \int_0^\sigma t^{-\frac{\kappa_i p}{p_i - p}} \diff t \right)^{1 - \frac{p}{p_i}}.$$
Hence, both $u^0$ and $u^1$ belong to $L^p(0, \sigma; X_1)$ almost surely. 
Since $A$ has stochastic maximal $L^p$-regularity on bounded time-intervals (see below Assumption \ref{ass:FGcritical}),  it follows from \cite[Prop. 3.11]{AVsurvey} that $u^0 = u^1$ a.s.\ on $[0, \sigma]$.
\end{proof}

\subsection{Maximal regularity for divergence form operators in extrapolation spaces}\label{subsec:MR}
In this subsection we give conditions which imply that a second order divergence form operator $L$ has maximal $L^p$-regularity as defined in Definition \ref{def:MaxReg}. 
This is used in the proof of Theorem \ref{thm:mainlocal} in Subsection \ref{ss:local_wp_proof_critical}.
For simplicity, we only consider scalar equations in this subsection, which is enough for our purposes. However, it is possible to extend it to (complex) systems using the appropriate extension of the ellipticity condition (see \cite[Section 5]{DK18}). As the deterministic setting allows more flexibility in the parameters, we will allow $p,q\in (1, \infty)$ and $\kappa\in (-1, p-1)$ in this subsection.

There is an extensive literature on maximal regularity for second order differential operators both in divergence and non-divergence form (see \cite{Analysis3,DHP,KuWe}). The main new aspect below is that we consider the differential operator on the space
    \begin{align}
        X_0 &=\HD^{-1-s,q}(\cO)
        \quad\text{with domain}\quad
        X_1 = \HD^{1-s,q}(\cO),
        \end{align}
with $s\in [-1,1]$ and $q\in (1, \infty)$. Here $\cO\subseteq \R^d$ is a bounded $C^2$-domain. 

\begin{assumption}\label{ass:DETMR}
Let $\mathcal O\subset \mathbb R^d$ be a bounded $C^2$-domain. Let
$1<p,q<\infty$ and $\kappa\in (-1,p-1)$. Let $a=(a^{j,k})_{j,k=1}^d:[0,T]\times \cO\to \R^{d\times d}$ be measurable and assume there exist $M>0$, $\ellip>0$ such that
\begin{align*}
\|a^{j,k}(t,\cdot)\|_{C^1(\overline{\mathcal{O}})} & \leq M, \ \ t\in [0,T],
\\ \sum_{j,k=1}^d a^{j,k}(t,x)
 \xi_j \xi_k
&\geq  \ellip |\xi|^2, \ \ \ t\in [0,T], x\in \overline{\cO}, \xi\in \R^d.
\end{align*}
\end{assumption}
Under Assumption \ref{ass:DETMR}, we define $L(t):\HD^{2,q}(\cO)\to L^q(\cO)$ 
by $L(t)u = -\operatorname{div}(a(t,\cdot)\nabla u)$. 
Integrating by parts twice, one finds
\[        \langle L(t)u,\psi\rangle = -\int_{\mathcal O}
        u(x)\,\operatorname{div}\big(a(t,x)^T\nabla\psi(x)\big)\,\dd x,
        \qquad
        \psi\in \HD^{2,q'}(\cO).\] 
This implies that we can also view $L$ as a $\calL(L^q(\cO),\HD^{-2,q}(\cO))$-valued function. By complex interpolation we see that $L:[0,T]\to \calL(\HD^{1-s,q}(\cO),\HD^{-1-s,q}(\cO))$ for any $s\in [-1,1]$. 

Using $s=1$, one can define very weak solutions as in \cite{LionsMagenesI, LionsMagenesII}. Moreover, using the latter setting we will obtain a result on maximal $L^p$-regularity for any $s\in [-1,1]$. 
\begin{theorem}[A Lions--Magenes type result]
\label{thm:LM-interpolation-fractional-dirichlet}
Suppose that Assumption \ref{ass:DETMR} holds.  
Let $s\in[-1,1]$ and let the family $L:[0,T]\to \calL(\HD^{1-s,q}(\mathcal O), \HD^{-1-s,q}(\mathcal O))$ be defined as  above. Then $L$ has maximal $L^p(0,T,w_{\kappa})$-regularity with constant $C$ which only depends on $T,M, \nu, p, q, \kappa,d,\cO$.  
\end{theorem}

\begin{proof}
In the proof we use general $A_p$-weights in time, since we need duality and time-reversal arguments in the proof. One can check that the definition of maximal $L^p$-regularity extends to general $A_p$-weights in time. For the theory of $A_p$-weights, the reader is referred to \cite[Chapter 7]{Grafakos1}. Recall that the dual weight $w' = w^{-1/(p-1)}$ is in $A_{p'}$. Moreover, $w(t) = t^{\kappa}$ is in $A_p$ if and only if $\kappa\in (-1, p-1)$. 

Maximal $L^p(0,T,w)$-regularity can be equivalently stated as the boundedness of 
\[\mathcal{M}:L^p(0,T,w;\HD^{-1-s,q}(\mathcal O))\to L^p(0,T,w;\HD^{1-s,q}(\mathcal O)),\]
where $\mathcal{M} f = u$, is defined as the unique $u\in L^p(0,T,w;\HD^{1-s,q}(\cO))$ such that for all $t\in [0,T]$, 
\[u(t) = \int_0^t - L(r) u(r) + f(r) \,\dd r. \]
Therefore, by complex interpolation and \cite[Theorem 2.2.6]{Analysis2}, it is enough to prove $\mathcal{M}$ is well-defined and bounded for $s=-1$ and $s=1$. 

\smallskip

{\em Step 1 -- Case $s=-1$:}  For $\cO=\R^d$ and $\cO = \R^d_+$, maximal $L^p(0,T,w)$-regularity follows from \cite[Theorems 5.2, 5.4 and Section 8]{DK18}. The case of bounded $C^2$-domains can be deduced from this by standard localization/partition of unity arguments (see \cite[Sections 6.3 and 7.1.3]{Evans}, \cite[Chapter 8]{Krylov2008_book}, \cite[Chapter IV]{LSU}). In these references, the passage from the whole-space and half-space estimates to bounded $C^2$-domains is obtained by the usual localization and boundary-flattening argument. The change of variables preserves uniform ellipticity and produces only lower-order terms with coefficients bounded in terms of $M$ and $\mathcal O$; these are absorbed by standard perturbation theorems for maximal regularity (cf.\ \cite[Theorem 3.2]{AVtorus} and \cite[Section 4.5]{GV17_pot}). The weights in time do not cause any difficulties since the finite partition of unity argument takes place in the space variable. One needs to shift the operator, but since we are working on finite time intervals one can shift back by using an exponential multiplication factor. The resulting constants depend only on $T,M,\nu,p,q,\kappa,d$ and $\mathcal O$.

\smallskip

{\em Step 2 -- Case $s=1$:} Let $\wt{L}:[0,T]\to \calL(\HD^{2,q'}(\cO), L^{q'}(\cO))$  be defined by
\[\wt{L}(t)\psi
        :=
        -\operatorname{div}\big(a(t,\cdot)^T\nabla\psi\big)
        =
        -\sum_{j,k=1}^d \big( a^{k,j}(t,\cdot)\partial_{j}\partial_k\psi
        -
        (\partial_j a^{k,j})(t,\cdot)\partial_k\psi\big).
\]
From Step 1 with weight $\wt{w}(t) := w'(T-t)$, and a time-reversal one can see that 
the backward problem
\begin{align}\label{eq:backward}
        -v'(t)+\wt{L}(t)v(t)=g(t),
        \qquad
        v(T)=0,
\end{align}
with $g\in L^{p'}(0,T,w';L^{q'}(\mathcal O))$
has a unique solution $v\in L^{p'}(0,T,w';\HD^{2,q'}(\cO))$, where a solution is defined in the strong sense as
\[        v(t)+\int_t^T \wt{L}(r)v(r) \,\dd r =\int_t^T g(r) \,\dd r.
\]
Moreover, by the maximal regularity of Step 1:
\[\|v\|_{L^{p'}(0,T,w';\HD^{2,q'}(\cO))}\le C\|g\|_{L^{p'}(0,T,w';L^{q'}(\cO))}.
\]
Furthermore, the above identity also implies that $v\in W^{1,p'}(0,T,w';L^{q'}(\cO))$.

Now, to prove maximal $L^{p}$-regularity, choose an arbitrary $f\in L^p(0,T,w;\HD^{-2,q}(\cO))$. For $g\in L^{p'}(0,T,w';L^{q'}(\cO))$ let $v_g\in L^{p'}(0,T,w';\HD^{2,q'}(\cO))$ denote the solution to \eqref{eq:backward}. Define 
\begin{align}\label{eq:Lambda1}
\Lambda_f(g) = \int_0^T \langle f(t),v_g(t)\rangle_{\HD^{-2,q}(\cO),\HD^{2,q'}(\cO)}
        \,\dd t.
\end{align}
Then $\Lambda_f$ is linear and by the regularity estimate for the backward problem 
\begin{align*}
|\Lambda_f(g)|&\leq \|f\|_{L^p(0,T,w;\HD^{-2,q}(\cO))}   \|v_g\|_{L^{p'}(0,T,w',\HD^{2,q'}(\cO))} 
\\ & \le C \|f\|_{L^p(0,T,w;\HD^{-2,q}(\cO))} \|g\|_{L^{p'}(0,T,w';L^{q'}(\cO))}.
\end{align*}
Therefore, since $(L^{p'}(0,T,w';L^{q'}(\cO)))^* = L^p(0,T,w;L^q(\cO))$ (with respect to the Lebesgue measure), there exists a unique $u_f\in L^p(0,T,w;L^q(\cO))$ such that 
\begin{align}\label{eq:Lambda2}
\Lambda_f(g) = \int_0^T \langle u_f(t),g(t)\rangle_{L^q(\cO),L^{q'}(\cO)} \,\dd t.
\end{align}
and $\|u_f\|_{L^p(0,T,w;L^q(\cO))} = \|\Lambda_f\| \leq C \|f\|_{L^p(0,T,w;\HD^{-2,q}(\cO))}$. 

It remains to check that $u_f$ is the solution to the original parabolic problem. For $u_f$ to be a solution we need to show that for all $t\in [0,T]$, 
\[u_f(t) = \int_0^t \big(- L(r) u_f(r) + f(r) \big)\,\dd r.\]
For this it is enough to check that for all  $\zeta \in C^\infty_c([0,T))$ (thus $\zeta(0)$ may be nonzero)
\[\int_0^T u_f(t) \zeta'(t) \,\dd t
        = \int_0^T \big(L(t) u_f(t) \zeta(t) - f(t) \zeta(t)\big) \,\dd t.
\]
This is equivalent to the statement that for all $\psi\in \HD^{2,q'}(\cO)$
\[\int_0^T \lb u_f(t), \psi\rb \zeta'(t) \,\dd t
        =  \int_0^T [\lb u_f(t),\wt{L}(t)\psi \rb  - \lb f(t),\psi\rb] \zeta(t) \,\dd t.
\]
To prove the latter identity, it suffices to set $g(t) = -\zeta'(t) \psi + \wt{L}(t) \psi \zeta(t)$, and note that $v_g = \zeta \psi$ and combine this with \eqref{eq:Lambda1} and \eqref{eq:Lambda2}.  

By linearity it is enough to prove uniqueness for $f=0$. Let $u\in L^p(0,T,w;L^q(\mathcal O))$ be such that for all $t\in [0,T]$
\[ u(t)+\int_0^t L(r)u(r) \,\dd r=0. \] 
Then an approximation argument shows that for every 
\begin{align}\label{eq:generalclass}
 \varphi\in W^{1,p'}(0,T,w';L^{q'}(\mathcal O)) \cap L^{p'}(0,T,w';\HD^{2,q'}(\mathcal O)) 
 \end{align} 
 with \(\varphi(T)=0\) 
one has 
\[ \int_0^T \langle u(t), \varphi'(t)-\widetilde L(t)\varphi(t) \rangle_{L^q,L^{q'}} \,\dd t =0. 
\] 
Indeed, this can first be checked for $\varphi\in C^1([0,T])\otimes \HD^{2,q'}(\cO)$ and then extended by density, first to $C^1([0,T];L^{q'}(\mathcal{O}))\cap C([0,T];\HD^{2,q'}(\cO))$ and then to the general case \eqref{eq:generalclass}. 

Now, let $g\in L^{p'}(0,T,w';L^{q'}(\mathcal O))$ be arbitrary, and let \(v_g\) be the solution of the backward adjoint problem \[ -v_g'(t)+\widetilde L(t)v_g(t)=g(t), \qquad v_g(T)=0. \] Taking \(\varphi=v_g\) in the preceding identity gives (in the beginning of Step 2 we have seen that $v_g$ has the required regularity) 
\[ \int_0^T \langle u(t),g(t)\rangle_{L^q,L^{q'}} \,\dd t = -\int_0^T \langle u(t), v_g'(t)-\widetilde L(t)v_g(t) \rangle_{L^q,L^{q'}} \,\dd t = 0. \] Since $g$ was arbitrary, it follows that $u=0$. 

The same argument applies on every subinterval $(a,b)\subset(0,T)$ after the translation $t\mapsto t-a$, with the weight $w_\kappa^a(t)=(t-a)^\kappa$. Since the $A_p$-characteristic of $w_\kappa^a$ is independent of $a$ and $b$, the constants are uniform. This gives maximal regularity in the sense of Definition \ref{def:MaxReg}.
\end{proof}

\begin{remark}
For $s=0$ many sufficient conditions are known under which an elliptic second order operator $A$ with Dirichlet or Neumann boundary conditions has maximal $L^p$-regularity in the case the coefficients are $(t,x)$-dependent. See for instance \cite[Theorem 8.1]{DongKimCalc11} and \cite[Theorem 3.3]{CK19} for $p,q\in (1, \infty)$, $C^1$-domains $\cO$, and with coefficients  measurable in time and VMO in space (at least unweighted).

One would expect that if the coefficients $a^{j,k}\in C^{|s|+}(\cO)$, then the maximal $L^p$-regularity results remain valid on $\HD^{-1-s,q}(\cO)$. Such results for space and time-dependent $a^{j,k}$ seem unavailable in the case of Dirichlet boundary conditions. In the periodic case and on $\R^d$ such results can be found in \cite[Theorem 5.2]{AVtorus}, where even the stochastic setting is covered. It might be possible to apply a similar technique as in the latter paper (but without the stochastic part) and reduce to the case $s=0$ by applying $(-\Delta)^{-s/2}$ to the equation and use perturbation arguments. 
\end{remark}

\begin{remark}\label{rem:NeumannMR}
In case the coefficients $a^{j,k}$ are time-independent, we can also consider $L(t) u =- \dv(a \nabla u)$ with the Neumann-type boundary condition $(a\nabla u)\cdot n=0$ on $\partial \cO$. Indeed, this follows from the same method as before again applying a deformed version of \cite[Theorem 6.4]{DK18}. 

Time-dependent coefficients would lead to time-dependent boundary conditions. In principle, this could be something one could study for some values of $s$ such as $s=0$. Then the domain of the operator remains time-independent which is something we need in Subsection \ref{sec:loc-well-posed}.
\end{remark}

\subsection{Proofs of the local well-posedness and identification of critical spaces}
\label{ss:local_wp_proof_critical}

Before we turn to the proof of Theorem \ref{thm:mainlocal}, we introduce some notation. 
    We define the following spaces on the domain $\cO$ 
    \begin{align}
        X_0 &=\HD^{-1-s,q}(\cO)
        \quad\text{and}\quad
        X_1 = \HD^{1-s,q}(\cO)
        \end{align}
        so that by interpolation
        \begin{align}
        X_{1-\tfrac{1+\kappa}{p},p} &= \BD^{1-s-2\tfrac{1+\kappa}{p}}_{q,p}(\cO)
        \quad\text{and}\quad
        \quad X_\theta = \HD^{-1-s+2\theta,q}(\cO),
        \quad\text{for $\theta\in (0,1)$.}
    \end{align}
Often we will use the same notation for the $\R^\ell$-valued spaces. 

The idea will be to apply Theorem \ref{thm:local}, for which we need to verify that
    Assumption \ref{ass:FGcritical} holds, with the mappings
    \begin{align}
        (Au)_i:=-\dv(a_i\cdot\nabla u_i),\quad (\Phi_2(\cdot,u))_i:=\dv(F_i(\cdot,u) - F_i(\cdot, 0)), \quad       (\Phi_1(\cdot,u))_i:&=f_i(\cdot,u) - f_i(\cdot, 0)
    \end{align}
    and $\phi_i = \dv( F_i(\cdot, 0)) + f_i(\cdot, 0)$  for $i=1,\dots,\ell$. 

    For the noise term we use $\cU = L^2(\cO)^\ell$ and define $\Psi(t,\omega,u)\in \gamma(\mathcal{U},X_{1/2})$ as
\begin{equation}\label{eq:defPsi}
    (\Psi(\cdot,u)h)_i = \sum_{j=1}^\ell
    M_{g_i^j(\cdot,u) - g_i^j(\cdot,0)}R^jh^j,
    \qquad h:=(h^1, \ldots, h^\ell)\in\mathcal U,
\end{equation}
and $(\psi h)_i =\sum_{j=1}^\ell M_{g_i^j(\cdot,0)}R^jh^j$ where we used Assumption \ref{ass: zeta f}. Recall that $R^j (R^j)^*$ is the covariance operator of $W^j(1)$.

Note that an $(s,q,p,\kappa)$-solution to  \eqref{eq:reaction_diffusion_system} is the same as an
$L^p_\kappa$-solution to \eqref{eq:SEE} with the above choices of the spaces.

\begin{proof}[Proof of Theorem \ref{thm:mainlocal}]
Since $q\in [2, \infty)$, the spaces $X_0$ and $X_1$ are UMD and have type $2$ (see \cite{Analysis1,Analysis2}).
For the reference operator $A_0$ appearing in Assumption \ref{ass:FGcritical}, we can just take the Dirichlet Laplacian $-\Delta$ (see \cite[Example A.4]{AVsurvey}). The maximal $L^p$-regularity condition on $A$ follows from Theorem \ref{thm:LM-interpolation-fractional-dirichlet} applied pointwise in $\Omega$, and using that $A$ has a diagonal structure.

    For the $\Phi_1$-estimate, we compute, for $\alpha_1\in [0,1-\tfrac{1+\kappa}{p})$ and $i=1,\dots,\ell$,
    \begin{align}
    \|f_i(u)-f_i(v)\|_{X_{\alpha_1}} 
    &\stackrel{(i)}{\lesssim} \|f_i(u)-f_i(v)\|_{L^{\eta_1}(\cO)}\\
    &\lesssim \|(1+|u|^{\rho_1}+|v|^{\rho_1})(u-v)\|_{L^{\eta_1}(\cO)}\\
    &\stackrel{(ii)}{\lesssim} 
    (1+\|u\|_{L^{(\rho_1+1)\eta_1}(\cO)}^{\rho_1} + \|v\|_{L^{(\rho_1+1)\eta_1}(\cO)}^{\rho_1})\|u-v\|_{L^{(\rho_1+1)\eta_1}(\cO)}\\
    \label{eq: estimate f main local}
    &\stackrel{(iii)}{\lesssim} 
    (1+\|u\|_{X_{\beta_1}}^{\rho_1} + \|v\|_{X_{\beta_1}}^{\rho_1})\|u-v\|_{X_{\beta_1}},
\end{align}
where in (ii) we applied H\"older's inequality and in (i) and (iii) we used the Sobolev embeddings $L^{\eta_1}(\cO)\hookrightarrow \HD^{-1-s+2\alpha_1,q}(\cO)$  and $\HD^{-1-s+2\beta_1,q}(\cO)\hookrightarrow L^{(\rho_1+1)\eta_1}(\cO)$ applied to $X_{\alpha_1} = \HD^{-1-s+2\alpha_1,q}(\cO)$ and $X_{\beta_1} = \HD^{-1-s+2\beta_1,q}(\cO)$, respectively, requiring
\begin{align}
\label{eq: wellposedness thm sobolev f 1}
    -1-s+2\alpha_1&\leq 0, \qquad\, -\frac{d}{\eta_1} = -1-s+2\alpha_1-\frac{d}{q}\\
    \label{eq: wellposedness thm sobolev f 2}
    -1-s+2\beta_1&\geq  0, \qquad\,
    -1-s+2\beta_1-\frac{d}{q}= -\frac{d}{(\rho_1+1)\eta_1}.
\end{align}
The special cases where one does not have equality could be considered separately. However, we do not consider them here since these cases are not relevant for the critical case, which is a main point of interest in this paper.
The criticality condition \eqref{eq:subcritical} in Assumption \ref{ass:FGcritical} requires further that
\begin{align}
\label{eq: wellposedness thm crit ineq f}
    \rho_1\big(\beta_1 -1 +\frac{1+\kappa}{p}\big)+\beta_1\leq 1 + \alpha_1,\quad \alpha_1\in \big[0,1-\frac{1+\kappa}{p}\big),\,\beta_1\in \big(1-\frac{1+\kappa}{p},1\big).
\end{align}
Solving for $2\beta_1$ in the (in)equalities \eqref{eq: wellposedness thm sobolev f 2} and \eqref{eq: wellposedness thm crit ineq f} and combining them yields  
\begin{align}
\label{eq: range beta_1}
   1+  s+\frac{d}{q}-\frac{d}{(\rho_1+1)\eta_1} = 2\beta_1\leq  2-2\frac{\rho_1}{\rho_1+1}\cdot \frac{1+\kappa}{p} + \frac{2\alpha_1}{\rho_1+1}.
\end{align}
Expressing $\eta_1$ from \eqref{eq: wellposedness thm sobolev f 1} yields
\begin{align}
\label{eq: def eta_1}
    \eta_1 = \frac{dq}{q+qs-2\alpha_1 q+d}.
\end{align}
We note that by \eqref{eq: wellposedness thm sobolev f 2} we need that $\eta_1\geq\tfrac{q}{\rho_1 + 1}$.
Plugging the expression for $\frac{d}{\eta_1}$ from \eqref{eq: wellposedness thm sobolev f 1} in \eqref{eq: range beta_1}, and requiring $2\beta_1\in (2-2\tfrac{1+\kappa}{p},2)$, as well as $\eta_1>1$ and $\alpha_1 <1-\tfrac{1+\kappa}{p}$ yields the bound
\begin{align}
\label{eq: range alpha case I}
    \max\Big\{(\rho_1+1)\Big(1-\frac{1 + \kappa}{p}\Big) - \frac{\rho_1}{2}\Big(1+s+\frac{d}{q}\Big),
    \,&\frac{1}{2}+\frac{s}{2}-\frac{(q-1)d}{2q},\frac{1}{2}+\frac{s}{2}-\frac{\rho_1 d}{2q}\Big\}\\
    <\alpha_1&<\min\Big\{    
    \frac{\rho_1 + 2}{2}-\frac{\rho_1 s}{2}-\frac{\rho_1 d}{2q},\, \frac{1}{2}+\frac{s}{2},\,\, 1-\frac{1+\kappa}{p}\Big\}.
\end{align}
Requiring that the range for $\alpha_1$ above is non-empty gives the bounds
\begin{align}
\label{eq: s extra bound alpha}
   1-2\frac{1+\kappa}{p}-\frac{\rho_1}{\rho_1+1}\frac{d}{q}<
    s<\min\Big\{1+\frac{d}{\rho_1+1}-\frac{d}{q},\,
    d+1-\frac{d}{q}-2\frac{1+\kappa}{p}, 1+\frac{d\rho_1}{q} -2\frac{1+\kappa}{p}\Big\},
\end{align}
which gives \eqref{eq: mainlocal combined ineq f 1}.
Plugging the expression from \eqref{eq: wellposedness thm sobolev f 1} for $\tfrac{d}{\eta_1}$ in \eqref{eq: range beta_1} gives
\begin{align}
    s\frac{\rho_1}{\rho_1+1}+\frac{d}{q}\frac{\rho_1}{\rho_1+1}-\frac{1}{\rho_1+1}
    \leq 1-2\frac{\rho_1}{\rho_1+1}\cdot \frac{1+\kappa}{p},
\end{align}
so that we obtain 
\begin{align}
\label{eq: s bound f term}
    s\leq \frac{\rho_1+2}{\rho_1}-\frac{d}{q} - 2\frac{1+\kappa}{p},
\end{align}
which is  \eqref{eq: mainlocal combined ineq f 2}. 
For $\Phi_2$ first consider $s>0$. We calculate
\begin{align}
    \|&\dv(F_i(\cdot,u) - F_i(\cdot,v))\|_{\HD^{-1-s+2\alpha_2,q}(\cO)}\\
    &\stackrel{\text{(i)}}{\lesssim} \|\dv(F_i(\cdot,u) - F_i(\cdot,v))\|_{\HD^{-1,\eta_2}(\cO)}\\
    &\lesssim \|F_i(\cdot,u) - F_i(\cdot,v)\|_{L^{\eta_2}(\cO)}\\
    &\stackrel{\text{(ii)}}{\lesssim} (1+\|u\|_{L^{(\rho_2+1)\eta_2}(\cO)}^{\rho_2} + \|v\|_{L^{(\rho_2+1)\eta_2}(\cO)}^{\rho_2})\|u-v\|_{L^{(\rho_2+1)\eta_2}(\cO)}\\
    &\stackrel{\text{(iii)}}{\lesssim} 
    (1+\|u\|_{\HD^{-1-s+2\beta_2,q}(\cO)}^{\rho_2} + \|v\|_{\HD^{-1-s+2\beta_2,q}(\cO)}^{\rho_2})\|u-v\|_{\HD^{-1-s+2\beta_2,q}(\cO)},
\end{align}
where in (ii) we again used H\"older's inequality and for (i) and (iii) we used the Sobolev embeddings
$\HD^{-1,\eta_2}(\cO)\hookrightarrow \HD^{-1-s+2\alpha_2,q}(\cO)$  and $\HD^{-1-s+2\beta_2,q}(\cO)\hookrightarrow L^{(\rho_2+1)\eta_2}(\cO)$, for which we require $\eta_2<q<(\rho_2+1)\eta_2$ and
\begin{align}
\label{eq: wellposedness thm sobolev F 1}
    -1-s+2\alpha_2&\leq -1,\qquad
    -1-\frac{d}{\eta_2} = -1-s+2\alpha_2-\frac{d}{q},\\
    \label{eq: wellposedness thm sobolev F 2}
    -1-s+2\beta_2&\geq 0,\qquad
    -1-s+2\beta_2 - \frac{d}{q} =  -\frac{d}{(\rho_2+1)\eta_2},
\end{align}
together with the criticality condition from Assumption \ref{ass:FGcritical}, which reads 
\begin{align}
\label{eq: wellposedness thm crit F}
    \rho_2\big(\beta_2-1+\frac{1+\kappa}{p}\big)+\beta_2\leq 1+\alpha_2.
\end{align}
Solving for $2\beta_2$ in \eqref{eq: wellposedness thm sobolev F 2} and \eqref{eq: wellposedness thm crit F} gives
\begin{align}
\label{eq: range beta 2}
   1+  s+\frac{d}{q}-\frac{d}{(\rho_2+1)\eta_2} = 2\beta_2\leq  2-2\frac{\rho_2}{\rho_2+1}\cdot \frac{1+\kappa}{p} + \frac{2\alpha_2}{\rho_2 +1}.
\end{align}
To obtain an admissible range of $\alpha_2$, we first use \eqref{eq: wellposedness thm sobolev F 1} to express 
\begin{align}
\label{eq: def eta_2}
    \eta_2 = \frac{dq}{d+qs-2q\alpha_2}.
\end{align}
We note again that from \eqref{eq: wellposedness thm sobolev F 2} we require $\eta_2> \tfrac{q}{\rho_2+1}$.
Plugging the expression for $\tfrac{d}{\eta_2}$ into \eqref{eq: range beta 2} and requiring that $2\beta_2\in (2-2\tfrac{1+\kappa}{p},2)$, as well as
 $\max\{1,\tfrac{q}{\rho_2+1} \}<\eta_2<q$, $\alpha_2<1-\tfrac{1+\kappa}{p}$ and that the inequality in \eqref{eq: wellposedness thm sobolev F 1} holds, yields 
\begin{align}
\label{eq: range alpha 2}
    \max\Big\{\frac{\rho_2+1}{2}\Big(1-2\frac{1+\kappa}{p}\Big)-\frac{\rho_2}{2}\Big(s+\frac{d}{q}\Big),\,
    &\frac{s}{2}-\frac{(q-1)d}{2q},\, 
    \frac{s}{2}-\frac{d\rho_2}{2q}
    \Big\}  \\
    < \alpha_2&<
    \min\Big\{ 
    \frac{\rho_2+1}{2}-\frac{\rho_2s}{2}-\frac{\rho_2d}{2q},\,
    \frac{s}{2},\,
    1-\frac{1+\kappa}{p}
    \Big\}.
\end{align}
We note that $\alpha_2<\frac{s}{2}$ 
stems from the embedding in \eqref{eq: wellposedness thm sobolev F 1} and that it ensures
that the denominator in the expression for $\eta_2$ is positive. We note further that, since $s<d-\tfrac{d}{q}$ by assumption, we have $\tfrac{s}{2}-\tfrac{(q-1)d}{2q}<0$, so that we may omit this term.
Requiring that the range for $\alpha_2$ is non-empty gives the first strict inequality in \eqref{eq: mainlocal combined ineq F}, as well as the second condition in \eqref{eq: mainlocal combined ineq F}.
Plugging in the expression for $\tfrac{d}{\eta_2}$ from \eqref{eq: wellposedness thm sobolev F 1} in \eqref{eq: range beta 2} then yields
\begin{align}
\label{eq: s bound F term}
    s\leq \frac{\rho_2+1}{\rho_2}-\frac{d}{q}-2\frac{1+\kappa}{p},
\end{align}
which is the remaining inequality in \eqref{eq: mainlocal combined ineq F}.

If $s=0$ (and thus $\zeta=2$) one can check that the argument goes through with the following choices: $\alpha_2=0$, $\eta_2=q$, and $\beta_2=\frac12+\frac{\rho_2}{2(\rho_2+1)}\frac{d}{q}$.

To treat the $g$ term, we recall that Assumption \ref{ass: zeta f} holds and apply Theorem \ref{thm:MgTmu delta} for $\zeta\in [2,\infty)$. In this case we obtain
\begin{align}
    \|&M_{g_i^j(\cdot,u)}R^j - M_{g_i^j(\cdot,v)}R^j\|_{\gamma(L^2(\cO),\HD^{-s,q}(\cO))} \\
    \label{eq: loc Lipschitz calculation G AC1D}
    &\stackrel{\text{(i)}}{\lesssim} \|g_i^j(\cdot,u)-g_i^j(\cdot,v)\|_{L^{\eta_3}(\cO)}\|\mu^j\|_{\ell^\zeta(\Sf^j)}\\
    &\lesssim\|(1+|u|^{\rho_3}+|v|^{\rho_3})|u-v|\|_{L^{\eta_3}(\cO)}\|\mu^j\|_{\ell^\zeta(\Sf^j)}
    \\
    &\lesssim (1+\|u\|_{L^{\eta_3(\rho_3+1)}(\cO)}^{\rho_3}+\|v\|_{L^{\eta_3 (\rho_3+1)}(\cO)}^{\rho_3})\|u-v\|_{L^{\eta_3 (\rho_3+1)}(\cO)}\|\mu^j\|_{\ell^\zeta(\Sf^j)}\\
    &\stackrel{\text{(ii)}}{\lesssim} (1+\|u\|_{\HD^{-1-s+2\beta_3,q}(\cO)}^{\rho_3}+\|v\|_{\HD^{-1-s+2\beta_3,q}(\cO)}^{\rho_3})\|u-v\|_{\HD^{-1-s+2\beta_3,q}(\cO)}\|\mu^j\|_{\ell^\zeta(\Sf^j)},
\end{align}
where in (i) we apply Theorem \ref{thm:MgTmu delta} with the condition
\begin{align}
\label{eq: wellposedness thm g new condition}
    \frac{s}{d} + \frac{1}{q} = \frac{1}{\eta_3} + \frac{1}{2} - \frac{1}{\zeta},\quad\text{and}\quad \frac{1}{\eta_3} - \frac{1}{\zeta}<\frac{1}{2},
\end{align}
and in (ii) we  require $X_{\beta_3} = \HD^{-1-s+2\beta_3,q}(\cO)\hookrightarrow L^{\eta_3 (\rho_3+1)}(\cO)$, or 
\begin{align}
\label{eq: wellposdness thm sobolev g}
-1-s+2\beta_3\geq 0,\qquad
-1-s+2\beta_3-\frac{d}{q}\geq -\frac{d}{\eta_3 (\rho_3+1)}    
\end{align}
 To apply Theorem \ref{thm:local} we require additionally the criticality condition \eqref{eq:subcritical3} from Assumption \ref{ass:FGcritical}, i.e.
 \begin{align}
 \label{eq: wellposedness thm crit g}
     \rho_3(\beta_3-1+\frac{1+\kappa}{p})+\beta_3\leq 1,\quad\beta_3\in \big(1-\frac{1+\kappa}{p},1\big).
 \end{align}
 Solving for $2\beta_3$ in \eqref{eq: wellposdness thm sobolev g} and \eqref{eq: wellposedness thm crit g}
 yields
 \begin{align}
 \label{eq: wellposedness thm g combined condition with beta}
     \max\Big\{1+s+\frac{d}{q}-\frac{d}{\eta_3(\rho_3+1)},1+s\Big\}\leq 2\beta_3\leq 2-2\frac{\rho_3}{\rho_3 +1}\frac{1+\kappa}{p}
 \end{align}
From \eqref{eq: wellposedness thm g new condition}, we can express
\begin{align}
\label{eq: wellposedness thm eta expression}
    \frac{d}{\eta_3} = s + \frac{d}{q} - \frac{d}{2}+\frac{d}{\zeta}. 
\end{align}
Substituting $\tfrac{d}{\eta_3}$ into 
 \eqref{eq: wellposedness thm g combined condition with beta} we see that we ensure existence of $\beta_3\in (1-\tfrac{1+\kappa}{p},1)$ if it holds that
\begin{align}
\label{eq: s rho_3 condition proof wellposedness thm}
    s\leq \min\Big\{\frac{\rho_3+1}{\rho_3} - 2\frac{1+\kappa}{p} - \frac{d}{q}-\frac{d}{2\rho_3} + \frac{d}{\zeta\rho_3},  1-\frac{2\rho_3}{\rho_3+1} \frac{1+\kappa}{p}\Big\}.
\end{align}
From the second condition in \eqref{eq: wellposedness thm g new condition} we additionally get
\begin{align}
    s<d-\frac{d}{q} 
\end{align}
which ensures, using \eqref{eq: wellposedness thm eta expression}, that $\tfrac{1}{\eta_3}<1$. 
We note moreover that the last displayed inequality is only more restrictive than $s<1$ if $d=1$. Since by assumption $s>\tfrac{d}{2}-\tfrac{d}{\zeta}$ we know that also $\tfrac{1}{\eta_3}>\tfrac{1}{q}$ is satisfied. 

Next, we consider the case $\zeta = 2$, i.e.\ we apply Theorem \ref{thm:MgTmu delta}\eqref{it2:MgTmu delta} in (i) in \eqref{eq: loc Lipschitz calculation G AC1D}. Then, we may suppose that the first inequality in \eqref{eq: wellposedness thm g new condition} holds with $\zeta = 2$, we may disregard the second inequality in \eqref{eq: wellposedness thm g new condition} and we require that $s\geq 0$ and $\eta_3\in (1,q]$. Then the only changes to the calculations above are that instead of $s>\tfrac{d}{2}-\tfrac{d}{\zeta}$ we get $s\geq 0$ and we get \eqref{eq: s rho_3 condition proof wellposedness thm} with $\zeta =2$.

The conditions on $\phi$ and $\psi$ are easy to check due to the conditions on $f(\cdot, 0)$ and $F(\cdot, 0)$ in Assumption \ref{ass:reaction_diffusion_global}. This finishes the proof.
\end{proof}

\begin{proof}[Proof of Corollary \ref{cor: critical local rho2}]
Recall that $f=0$ in this case and $\rho_2 = \rho/2$. From now on we set $s = \frac{\rho+2}{\rho} - \frac{d}{q} - 2\gamma$, where for simplicity we denote $\gamma = \frac{1+\kappa}{p}$. We need to check that, for the given choices one has 
\begin{align}
\label{eq: cor rho_1=0 s condition}
    \frac{d}{2}-\frac{d}{\zeta} <s< 1
    ,\quad\text{for $\zeta\in (2,\infty)$, or}\quad
    0\leq s<1,\quad\text{for $\zeta = 2$,}
\end{align}
and that the inequalities \eqref{eq: mainlocal combined ineq F} and \eqref{eq: thm local s ineq 3} are satisfied.
To check $\tfrac{d}{2}-\tfrac{d}{\zeta}<s<1$, note that the upper bound on $s$ is equivalent to 
$\frac{2}{\rho} - \frac{d}{q} < 2\gamma$, which is one of the terms in  \eqref{eq:rho2condpkappa}. Requiring $s>\tfrac{d}{2}-\tfrac{d}{\zeta}$ for $\zeta>2$ immediately gives the upper bound on $\gamma$ in \eqref{eq:rho2condpkappa}, whereas requiring $s\geq 0$ for $\zeta=2$ gives the analogous bound as non-strict inequality. From the requirement $s<d-\tfrac{d}{q}$  we get
$\frac{\rho+2}{\rho}-d<2\gamma$,
which is one of the lower bounds in \eqref{eq:rho2condpkappa}. 
Next we check inequalities
 \eqref{eq: mainlocal combined ineq F} and \eqref{eq: thm local s ineq 3}.
Indeed, \eqref{eq: mainlocal combined ineq F} simplifies to
\begin{align}
    1+\frac{2}{\rho+2}\frac{d}{q}<\frac{\rho+2}{\rho},\quad\text{and}\quad
    \frac{2}{\rho}-2\frac{1+\kappa}{p}<\frac{2d}{\rho+2},
\end{align}
 which yields one of the terms in \eqref{eq:rho2condpkappa} and 
 \begin{align*}
 \frac{d}{q}<1+\frac{2}{\rho}.
 \end{align*}
The latter is implied by \eqref{eq:rho2condq}. The first condition in \eqref{eq: thm local s ineq 3} is satisfied by our choice of $\rho_3$. 
The second condition in \eqref{eq: thm local s ineq 3} together with the choice of $s$ yields
\begin{align}
    1+\frac{2}{\rho}-\frac{d}{q}-2\gamma + 2\gamma\frac{\rho_3}{\rho_3+1}\leq 1,
\end{align}
which simplifies to the second lower bound in \eqref{eq:rho2condpkappa}.
It remains to show that due to 
\eqref{eq:rho2condq}, one can indeed find $\gamma\in (0,1/2)$ such that
\eqref{eq:rho2condpkappa}
holds. 
Requiring that $\tfrac{\rho+2}{\rho}-d$ is smaller than the right-hand side in \eqref{eq:rho2condpkappa} does not lead to any restrictions on $q$, as
\begin{align}
    \frac{\rho+2}{\rho}-d <\frac{\rho+2}{\rho}-\frac{d}{q}-\frac{d}{2}+\frac{d}{\zeta}
    \quad\text{is equivalent to}\quad
    \frac{1}{q}<\frac{1}{2}+\frac{1}{\zeta}
\end{align}
which is always satisfied for $q\geq 2$. 
Moreover, $\frac{\rho+2}{\rho}-d<1$ is satisfied precisely when $\rho>\frac{2}{d}$, which is assumed in the statement.
We also require, in $d\geq 2$, 
\begin{align}
    \Big(\frac{2}{\rho}-\frac{d}{q}\Big)(1+\rho_3)<\frac{\rho+2}{\rho}-\frac{d}{q}-\frac{d}{2}+\frac{d}{\zeta}
\end{align}
which follows from $1-\frac{d}{2}+\frac{d}{\zeta}>0$ and the condition on $\rho_3$ owing to Assumption \ref{ass: zeta f}. 
Requiring additionally that $\big(\frac{2}{\rho}-\frac{d}{q}\big)(1+\rho_3)<1$ yields the lower bound in \eqref{eq:rho2condq}.
The condition
\begin{align}
    \frac{2}{\rho}- \frac{2d}{\rho+2}<\frac{\rho+2}{\rho}-\frac{d}{q}-\frac{d}{2}+\frac{d}{\zeta}
\end{align}
is equivalent to the first minimum in the upper bound on $\tfrac{d}{q}$ in \eqref{eq:rho2condq}. 
Moreover, we see that $\frac{2}{\rho}- \frac{2d}{\rho+2}<1$ is satisfied due to our assumption on $\rho$.
Finally we see that the second minimum in \eqref{eq:rho2condq} implies that the right-hand side in \eqref{eq:rho2condpkappa} is positive.
This finishes the proof.
\end{proof}

\begin{proof}[Proof of Corollary \ref{cor: critical localrho1}]
From now on we set $s = \frac{\rho+2}{\rho} - \frac{d}{q} - 2\gamma$, where we denote again $\gamma = \frac{1+\kappa}{p}$. We need to check that for the given choices one has 
\eqref{eq: cor rho_1=0 s condition}, that $s<d-\tfrac{d}{q}$,
and that \eqref{eq: mainlocal combined ineq f 1} and \eqref{eq: thm local s ineq 3} hold. 

The condition $\tfrac{d}{2}-\tfrac{d}{\zeta}<s<1$ if $\zeta\in (2,\infty)$ gives the upper bound in \eqref{eq:rho1condpkappa} and 
$\frac{2}{\rho} - \frac{d}{q}<2\gamma$,
which is implied by \eqref{eq:rho1condpkappa}.
The condition $s<d-\tfrac{d}{q}$ gives the first term in the maximum in \eqref{eq:rho1condpkappa}.
If $\zeta= 2$ and we impose $s\geq 0$, we instead get the constraint $2\gamma\leq \tfrac{\rho+2}{\rho}-\tfrac{d}{q}$, as mentioned in the Corollary statement. 
For the present choice of $s$ condition \eqref{eq: mainlocal combined ineq f 1}
is equivalent to
\begin{align}
    \frac{1}{\rho+1}\frac{d}{q}<\frac{2}{\rho}<\min\Big\{d,\,\frac{d(\rho+1)}{q}\Big\}
    \quad\text{and}\quad
    \frac{2}{\rho}-2\frac{1+\kappa}{p}<\frac{d}{\rho+1},
\end{align}
which
yields the constraint $\rho>\tfrac{2}{d}$, the lower bound on $\tfrac{d}{q}$ in \eqref{eq:rho1condq}, the third term in the maximum in \eqref{eq:rho1condpkappa} and the
bound 
\[\frac{d}{q}<2+\frac{2}{\rho}.\] 
The latter is implied by \eqref{eq:rho1condq}. 

Again, the first inequality in \eqref{eq: thm local s ineq 3} holds by the condition on $\rho_3$. 
It remains to show that due to  \eqref{eq:rho1condq}, one can find $\gamma\in (0,1/2)$ such that
\eqref{eq:rho1condpkappa} holds, which introduces the first term in the minimum in \eqref{eq:rho1condq}, as well as the second term in the maximum therein. These checks are analogous to the ones in the proof of Corollary \ref{cor: critical local rho2}, which we leave to the reader. 
\end{proof}

\begin{proof}[Proof of Theorem \ref{thm:mainlocal zeta infty}]
     We may repeat the proof of Theorem \ref{thm:mainlocal}, where
     we apply Theorem \ref{thm:MgTmu delta} \eqref{it3:MgTmu delta}  in the estimate (i) of \eqref{eq: loc Lipschitz calculation G AC1D}, with $\eta:=\eta_3$ and $d:=1$. From Theorem \ref{thm:MgTmu delta}\eqref{it3:MgTmu delta} we get the condition
     \begin{align}
         \frac{1}{\eta_3}=s+\frac{1}{q}-\frac{1}{2}.
     \end{align}
      together with $2<\eta_3<q<\infty$. Using this in \eqref{eq: wellposedness thm crit g} we immediately obtain \eqref{eq: thm local s ineq 3 zeta infty}. We see also
that the condition $2<\eta_3<q<\infty$ is satisfied precisely when $\tfrac{1}{2}<s<1-\tfrac{1}{q}$.
\end{proof}

In order to prove Proposition \ref{prop:mainlocal unweighted}, we need the following version of Theorem \ref{thm:MgTmu delta}, which does not make any assumptions on $(\|\varphi_n\|_{L^\infty})_{n\geq 1}$. We refer to \cite{AGV} for a comparison of Theorem \ref{thm:generalONB} to Theorem \ref{thm:MgTmu delta}.

\begin{theorem}[$\mu\in\ell^\zeta$]
\label{thm:generalONB}
Let $s\in [ 0,2]$ and  $\eta\in (2, \infty)$. Suppose that either
\begin{enumerate}[(1)]
\item  $\zeta\in (2, \infty)$ and $q\in (1, \infty)$ satisfy
\begin{align}
\label{eq:parametersnew}
\frac{s}{d} + \frac{1}{q} = \frac{1}{\eta} +\frac{1}{2}, 
\quad  \text{and}  \quad  
\frac{1}{\eta} +\frac{1}{\zeta}>\frac{1}{q},
\end{align}
\item $\zeta = 2$ and $q\in [1,\infty)$ satisfy
\begin{align}
    \frac{s}{d} + \frac{1}{q} = \frac{1}{\eta}+\frac{1}{2}.
\end{align}
\end{enumerate}
Then for each $g\in L^{\eta}(\mathcal{O})$ 
the operator $M_g R:L^2(\mathcal{O})\to \HD^{-s,q}(\mathcal{O})$ is $\gamma$-radonifying and 
\begin{align}
\|M_g R\|_{\gamma(L^2(\mathcal{O}),\HD^{-s,q}(\mathcal{O}))}\lesssim \|\mu\|_{\ell^\zeta} \|g\|_{L^\eta(\mathcal{O})},
\end{align}
where the constant depends on the parameters $(d,s,q,\eta,\zeta)$. 
\end{theorem}
A similar result holds for the Neumann setting. 
\begin{proof}[Proof of Proposition \ref{prop:mainlocal unweighted}]
    To prove this, we may repeat the same calculations as in the proof of Theorem \ref{thm:mainlocal}, where instead of applying Theorem \ref{thm:MgTmu delta} in (i) in \eqref{eq: loc Lipschitz calculation G AC1D}, we apply Theorem \ref{thm:generalONB}. If $\zeta\in (2,\infty)$  we then get instead of \eqref{eq: wellposedness thm g new condition} the condition
\begin{align}
\label{eq: wellposedness thm general ONB condition}
    \frac{s}{d}+\frac{1}{q} = \frac{1}{\eta_3} + \frac{1}{2}
    \quad\text{and}\quad
    \frac{1}{\eta_3}+\frac{1}{\zeta}>\frac{1}{q}.
\end{align}
Solving for $\tfrac{d}{\eta_3}$ in \eqref{eq: wellposedness thm general ONB condition} and plugging it into \eqref{eq: wellposedness thm g combined condition with beta}
 we get instead of \eqref{eq: thm local s ineq 3} the first inequality in \eqref{eq: thm local s ineq 3 no f}
 and inequality \eqref{eq: thm local s ineq 3 no f 2}. Requiring that $\eta_3\in (2,\infty)$, or $\frac{1}{\eta_3}\in (0,\tfrac{1}{2})$ then yields the second condition in \eqref{eq: thm local s ineq 3 no f}. 
\end{proof}

\subsection{Proof of the regularization Theorem \ref{thm: regularization}}\label{ss:regbootstrap}

The proof of the regularization is divided in three steps:
\begin{enumerate}[\rm (1)]
\item temporal regularization through Theorem \ref{thm: temporal regularization}.
\item bootstrapping spatial integrability and removing the $\theta^*$ restriction.
\item bootstrapping spatial regularity. 
\end{enumerate}

\begin{proof}[Proof of Theorem \ref{thm: regularization}]
  \textit{Step 1: Temporal regularization.} By an application of Theorem \ref{thm: temporal regularization} we immediately get
   \begin{align}
   \label{eq: u temporal regularization}
   u\in H^{\theta,r}_{\loc}((0,\sigma);\HD^{1-s-2\theta,q}(\cO))\quad\text{for $\theta\in [0,\theta^*)$ and $r\in [2,\infty)$,}
\end{align}    
which for fixed $\theta\in (0,\theta^*)$ and $r$ sufficiently large implies
\begin{align}
    u\in
        C_{\loc}((0,\sigma),\HD^{1-s-2\theta,q}(\cO)).
\end{align}
This means that for each $\varepsilon>0$, on the set $\Gamma_\varepsilon:=[\varepsilon<\sigma]$ we have 
$\1_{\Gamma_\varepsilon}u(\varepsilon)\in \HD^{1-s-2\theta,q}(\cO)$ almost surely, for $\theta\in (0,\theta^*)$.\newline
\textit{Step 2: Bootstrapping spatial integrability.}
The strategy of this step (and the next) is as follows. Let $(\sigma_n)_{n\geq 1}$ be a localizing sequence for $(u,\sigma)$ satisfying $\sigma_n<\sigma$ almost surely (this is possible due to the predictability of $\sigma$, see, for instance, \cite[Prop. 4.12]{AV19_QSEE_2}). Instead of \eqref{eq:reaction_diffusion_system}, we consider for $\varepsilon>0$ and $n\geq 1$,
    \begin{equation}
\label{eq: reg lemma v eps spde}
\left\{
\begin{aligned}
\dd v_{i} -\dv(a_i\cdot\nabla v_{i}) \,\dd t& = \1_{\Gamma_\varepsilon}\1_{[\varepsilon,\sigma_n]}\Big[\dv(F_i(\cdot, u)) +f_i(\cdot, u)\Big]\,\dd t +  \sum_{j=1}^\ell \1_{\Gamma_\varepsilon}\1_{[\varepsilon,\sigma_n]}g_{i}^j(\cdot,u) \,\dd W^j(t),\\
v_{i}(\varepsilon)&=\1_{\Gamma_\varepsilon}u_i(\varepsilon).
\end{aligned}\right.
\end{equation}
In this sense, we treat the functions $f(u),F(u),g(u)$ as inhomogeneities and prove the result for a solution $v$ of the linear equation \eqref{eq: reg lemma v eps spde}, which we then identify $v=u$ on $\Gamma_\varepsilon\times [\varepsilon,\sigma_n]$ for all $n\geq 1$, using that $u$ and $v$ satisfy the same equation on this set. 

We now prove that
\begin{align}
\label{eq: assertion Step 2}
u\in H^{\theta,r}_{\loc}((0,\sigma),\HD^{1-s-2\theta,\overline{q}}(\cO))\quad \text{for $r>p$, $\overline{q}\geq q$ and $\theta\in[0,1/2)$.}
\end{align}
To do so, we apply stochastic maximal regularity to \eqref{eq: reg lemma v eps spde} with the inhomogeneities $f(u), \dv F(u)$ and $g(u)$ iteratively. 
We fix $m\geq q$ and suppose that we have shown almost surely
$$\1_{\Gamma_\varepsilon}u\in~H^{\theta,r}(\varepsilon,\sigma_n;\HD^{1-s-2\theta,m}(\cO)), \text{ for $\theta\in [0,\theta^*)$, for a $\theta^*>0$ and all $r\in (p,\infty)$.}$$ We now show the existence of a $\delta>0$, independent of $m$, such that $v\in H^{\theta,r}_{\rm loc}((\varepsilon,\infty);\HD^{1-s-2\theta,m+\delta}(\cO))$ a.s.\ for $\theta\in [0,1/2)$. 
Identifying $v=u$ on $\Gamma_\varepsilon\times [\varepsilon,\sigma_n]$ we have then shown the validity of the implication  
\begin{align}
\label{eq: reg corollary implication}
\hspace{-1cm} \1_{\Gamma_\varepsilon}u\in 
\bigcap_{\theta\in [0,\theta^*), r>p} &H^{\theta,r}(\varepsilon,\sigma_n;\HD^{1-s-2\theta,m}(\cO))\\
    \quad&\Rightarrow\quad
    \1_{\Gamma_\varepsilon}u\in\bigcap_{\theta\in [0,1/2), r>p} H^{\theta,r}(2\varepsilon,\sigma_n;\HD^{1-s-2\theta,m+\delta}(\cO)),
\end{align}
a.s. Since $\delta$  can be chosen independent of $m\geq q$ and $\varepsilon>0$ and $n\geq 1$ were arbitrary, this then implies \eqref{eq: assertion Step 2}.

We begin by estimating $f(u)$ in detail, $\dv F(u)$ and $g(u)$ can be treated analogously. 
Let $\delta>0$. 
Then we calculate
\begin{align}
\label{eq: estimate f regularization}
\|f(u)\|_{\HD^{-1-s,m+\delta}(\cO)}\stackrel{(i)}{\lesssim}
\|f(u)\|_{L^{\eta_1}(\cO)}\lesssim 1+\|u\|_{L^{\eta_1(\rho_1 + 1)}(\cO)}^{\rho_1+1}
\stackrel{(ii)}{\lesssim}1+\|u\|_{\HD^{1-s,m}(\cO)}^{\rho_1+1}, 
\end{align}
where for the Sobolev embeddings (i) and (ii) we require $L^{\eta_1}(\cO)\hookrightarrow \HD^{-1-s,m+\delta}(\cO)$ and $\HD^{1-s,m}(\cO)\hookrightarrow L^{\eta_1(\rho_1+1)}(\cO)$, respectively, or 
\begin{equation}
\begin{aligned}\label{eq: regularization f embeddings}
    \text{(i):}&\quad
    -\frac{d}{\eta_1}\geq -1-s-\frac{d}{m+\delta},\qquad 
    \eta_1\leq m+\delta, 
    \\ 
    \quad\text{(ii):}&\quad
    1-s-\frac{d}{m}\geq -\frac{d}{\eta_1(\rho_1+1)}, 
    \qquad m\leq \eta_1(\rho_1+1).
\end{aligned}
\end{equation}
For the range of $\eta_1$ to be non-empty, we hence require
\begin{align}
    -1-s-\frac{d}{m+\delta}\leq (\rho_1+1)\Big( 1-s-\frac{d}{m}\Big),
    \quad\text{or equivalently,}\quad
    \frac{d\rho_1(m+\delta)+d\delta}{m(m+\delta)}\leq 2+\rho_1-\rho_1 s.
\end{align}
Using the bound for $s$ in \eqref{eq: mainlocal combined ineq f 2}, together with the fact that $2\leq q\leq m$ it is easy to verify the above inequality is satisfied if we have that 
\begin{align}
\label{eq: delta_1 condition 1}
    \delta\leq \frac{8\rho_1}{d}\cdot\frac{1+\kappa}{p}.
\end{align}
The reader may further check that for any $\delta$ the condition $\eta_1\in (\max\{1,\tfrac{m}{\rho_1+1}\},m+\delta)$, or $-\tfrac{d}{\eta_1}\in (\max\{-d,-\tfrac{d(\rho_1+1)}{m}\},-\tfrac{d}{m+\delta})$ yields no further restriction due to the second condition in \eqref{eq: mainlocal combined ineq f 1} and hence, it is always possible to find an admissible $\eta_1\in (1,\infty)$ such that \eqref{eq: regularization f embeddings} holds, if  \eqref{eq: delta_1 condition 1} is satisfied. This yields that
\begin{align}
    f(u)\in L^r((\varepsilon,\sigma_n),\HD^{-1-s,m+\delta}(\cO)),
    \quad\text{a.s., for $\delta$ satisfying \eqref{eq: delta_1 condition 1}.}
\end{align}
Similarly, we also obtain that almost surely
\begin{align}
    \dv F(u)\in L^r((\varepsilon,\sigma_n),\HD^{-1-s,m+\delta}(\cO)),\quad
    M_{g(u)}R \in L^{r}((\varepsilon,\sigma_n),\gamma(L^2(\cO),\HD^{-s,m+\delta}(\cO))),
\end{align}
for $\delta$ satisfying
\begin{align}
\label{eq: delta_1 condition 2}
    \delta< \frac{8\min\{\rho_2,\rho_3,1\}}{d}\cdot\frac{1+\kappa}{p}.
\end{align}
Moreover, if $1-s-\tfrac{d}{m}> 1-s-2\tfrac{1+\kappa}{p}-\tfrac{d}{m+\delta}$, which holds, in particular, for $\delta$ satisfying \eqref{eq: delta_1 condition 2}, and if for $r\geq p$ the weight $\kappa_r\in [0,\tfrac{r}{2}-1)$ is chosen such that $\tfrac{1+\kappa}{p} = \tfrac{1+\kappa_r}{r}$,
then we have, for $\theta\in (0,1/2)$ sufficiently small,
\begin{align}
v(\varepsilon) = \1_{{\sigma>\varepsilon}} u(\varepsilon)\in \HD^{1-s-2\theta,m}(\cO)\hookrightarrow \BD_{m+\delta,r}^{1-s-2\tfrac{1+\kappa_r}{r}}(\cO),\quad\text{a.s.}
\end{align}
Stochastic maximal $L^r_{\kappa_r}$-regularity (in the form of \cite[Prop. 3.11]{AVsurvey}) applied with the spaces $Y_j = \HD^{-1-s+2j,m+\delta}(\cO)$, $j=0,1$, then implies that 
\begin{align}
v\in H^{\theta,r}_{\rm loc}([\varepsilon,\infty),w_{\kappa_r}^\varepsilon;\HD^{1-s-2\theta,m + \delta}(\cO)),    
\end{align}
where $w_{\kappa_r}^\varepsilon(t) = (t-\varepsilon)^{\kappa_r}$,
for  $\theta\in [0,1/2)$.
Lemma \ref{lem:compatibility_solutions} gives $u=v$ on $\Gamma_\varepsilon\times [\varepsilon,\sigma_n]$ almost surely. 
This implies \eqref{eq: reg corollary implication}, since the weight $w_{\kappa_r}^\varepsilon$ acts only on $\varepsilon$. Recalling that $\varepsilon>0$, $n\geq 1$ and $r\geq p$ were arbitrary and that $\Gamma_\varepsilon\uparrow\Omega$ as $\varepsilon\rightarrow 0$ finishes the proof of Step 2.\newline
\textit{Step 3: Bootstrapping spatial smoothness.}
In this step we show that for all $\theta\in [0,1/2)$
\begin{align}
\label{eq: assertion Step 3}
    u\in H^{\theta,r}_{\loc}((0,\sigma),\HD^{1-\frac
{d}{2}+\frac{d}{\zeta}-2\theta-\, ,\overline{q}}(\cO)),\quad \text{for $r>2$, $\overline{q}\geq q$.}
\end{align}
First we note that by choosing $\overline{q}>q$ and $r>2$  sufficiently large and $\theta>0$ sufficiently small in Step 2, \eqref{eq: assertion Step 2}, we have for any $n\geq 1$ and $\varepsilon>0$ a.s. 
\begin{align}\label{eq:Sobolembeddinginproof}
    u\in C([\varepsilon,\sigma_n],\HD^{1-s-2\theta,\overline{q}}(\cO))\hookrightarrow C([\varepsilon,\sigma_n],C^a(\cO)), \ \text{for some $a>0$.}
\end{align}
The reader is referred to \cite[Corollary 14.4.27 and Proposition 14.6.13]{Analysis3} for the Sobolev embedding we applied here. 
The latter continuity immediately gives, for any 
$\overline{q}> q$,
\begin{align}
\label{eq: f F estimate reg lemma}
    f(u)\in L^\infty(\varepsilon,\sigma_n,L^\infty(\cO)),\quad \dv F(u)\in L^\infty(\varepsilon,\sigma_n,\HD^{-1,\overline{q}}(\cO))
\end{align}
almost surely. Moreover, if $\zeta\in (2,\infty)$, then by Theorem \ref{thm:MgTmu delta}\eqref{it1:MgTmu delta},
\begin{align}
\label{eq: g estimate reg lemma}
    \sup_{t\in [\varepsilon,\sigma_n]}\|M_{g_i^j(u(t))}R^j\|_{\gamma(L^2(\cO),\HD^{-\overline{s},\overline{q}}(\cO))}&\lesssim \|\mu^j\|_{\ell^{\zeta}(\Sf^j)} \sup_{t\in [\varepsilon,\sigma_n]}\|g_i^j(u(t))\|_{L^{\eta_3}(\cO)}
    \\ & \lesssim 1+\|u\|_{C([\varepsilon,\sigma_n],C(\cO))}^{\rho_3+1},
\end{align}
almost surely, for any $\tfrac{d}{2}-\tfrac{d}{\zeta}<\overline{s}<1$, $\overline{q}> q$ and $\eta_3\in (2,\overline{q})$, such that 
\begin{align}
\label{eq: equality MgR}
    \frac{\overline{s}}{d} + \frac{1}{\overline{q}} = \frac{1}{\eta_3} + \frac{1}{2} - \frac{1}{\zeta}.
\end{align}
In particular, $\overline{s}>\tfrac{d}{2}-\tfrac{d}{\zeta}$ can be chosen arbitrarily close to $\tfrac{d}{2}-\tfrac{d}{\zeta}$, since we can always find $\overline{q}> q$ and $\eta_3\in (2,\overline{q})$ sufficiently close to $\overline{q}$ such that \eqref{eq: equality MgR} is satisfied.
Next we note that by choosing $\kappa_r\in [0,\tfrac{r}{2}-1)$ sufficiently large such that
    $1-\overline{s}-2\frac{1+\kappa_r}{r}<0$, 
we have 
\begin{align}
     C(\cO)\hookrightarrow \BD^{1-\overline{s}-2\tfrac{1+\kappa_r}{r}}_{\overline{q},r}(\cO)
     \quad\text{and hence a.s.}\quad
     \1_{\Gamma_\varepsilon}u(\varepsilon)\in \BD^{1-\overline{s}-2\tfrac{1+\kappa_r}{r}}_{\overline{q},r}(\cO).
\end{align}
Using stochastic maximal $L^r$-regularity with \eqref{eq: f F estimate reg lemma} and \eqref{eq: g estimate reg lemma} for the linear equation \eqref{eq: reg lemma v eps spde} with the spaces $Y_j = \HD^{-1-\overline{s}+2j,\overline{q}}(\cO)$, $j=0,1$, gives a solution $v\in H^{\theta,r}_{\rm loc}((\varepsilon,\sigma_n];\HD^{1-\overline{s}-2\theta,\overline{q}}(\cO))$ for any $\theta\in [0,1/2)$. By Lemma \ref{lem:compatibility_solutions} we then know that $u=v$ on $\Gamma_\varepsilon\times [\varepsilon,\sigma_n]$ almost surely.
Since $\overline{s}>\tfrac{d}{2}-\tfrac{d}{\zeta}$, $\varepsilon>0$  and $n\geq 1$ were arbitrary, this proves \eqref{eq: assertion Step 3}.

Finally, by Sobolev embedding, we have almost surely
 \begin{align}
     u\in C^{\theta_1}_{\rm loc}((0,\sigma),C(\cO))
     \quad\text{and}\quad
     u\in C_{\rm loc}((0,\sigma),C^{\theta_2}(\cO)),
     \quad\text{for $\theta_1\in [0,\lambda/2)$ and $\theta_2\in [0,\lambda)$,}
 \end{align}
 where $\lambda = 1-\tfrac{d}{2}+\tfrac{d}{\zeta}$.
 This implies \eqref{eq: reg u result 2}. 
 Finally, we see that if $\zeta = 2$, then we may apply Theorem \ref{thm:MgTmu delta}\eqref{it2:MgTmu delta} to obtain \eqref{eq: g estimate reg lemma} with $\eta_3 =\overline{q}$ and hence $\overline{s} =  0$, which gives \eqref{eq: reg u result 1} also with $\theta = 0$. In order to treat the initial data at time $\varepsilon$, one additionally uses that $\one_{\Gamma_{\varepsilon}} u(\varepsilon)\in C^a(\cO)$ as in \eqref{eq:Sobolembeddinginproof}. This finishes the proof.
\end{proof}

\subsection{Proof of the blow up criteria}
\label{ss:proof_blow_up_criteria}
We begin this subsection with a preliminary result which is of independent interest, and ensures that solutions provided by Theorem \ref{thm:mainlocal} with different choices of the parameters yield the same solution.

\begin{proposition}[Compatibility of different settings]
\label{prop: compatibility}
    Let the conditions of Theorem \ref{thm:mainlocal} hold for two sets of values $(s_i,q_i,p_i,\kappa_i, \rho_{i,1}, \rho_{i,2}, \rho_{i,3})$, $i=1,2$, and fixed $(d,\zeta)$, 
    and let $(u_i,\sigma_i)$, $i=1,2$, be the associated maximal $L^{p_i}_{\kappa_i}$-solutions with the same initial data $u_0\in \cap_{i\in \{1, 2\}}L^0_{\cF_0}(\Omega,\BD_{q_i,p_i}^{1-s_i-2\tfrac{1+\kappa_i}{p_i}}(\cO))$. Then the solutions coincide, i.e.\ $\sigma_1 = \sigma_2$ a.s.\ and $u_1=u_2$ a.e.\ on $ [0,\sigma_1)\times \Omega $.
\end{proposition}
\begin{proof}
    The result is \cite[Proposition 3.5]{agresti2023reaction} and the proof is a verbatim repetition of the proof therein.
    Note that to extend the arguments of \cite[Proposition 3.5]{agresti2023reaction} to the current situation, it suffices to use the blow-up criteria for the solutions in Theorem \ref{thm:mainlocal} given by Theorems \ref{thm: blow up subcritical} and \ref{thm: blow up critical 2} with the choice of the spaces done at the beginning of Subsection \ref{ss:local_wp_proof_critical}.
\end{proof}

\begin{proof}[Proof of Theorem \ref{thm: blow up}]
Let $(u,\sigma)$ be the maximal $(s,q,p,\kappa)$-solution to \eqref{eq:reaction_diffusion_system} provided by Theorem \ref{thm:mainlocal} and let 
the parameters 
$(s_0,q_0,p_0,\kappa_0,\zeta,\rho_{0},\rho_0/2,\rho_{0,3})$ be as in the theorem statement.
We claim that it suffices to prove Theorem \ref{thm: blow up}\eqref{it:eq: blow up 2} with the spaces $H$ and $B$ replaced by the ones with Dirichlet boundary conditions $\HD$ and $\BD$ (see Subsection \ref{ss:function_spaces_Dir}), respectively. Indeed, by Theorems \ref{thm:mainlocal} and \ref{thm: regularization}, we know that, a.s.,
$$
u\in C((0,\sigma)\times \overline{\cO})\quad \text{ and } \quad u|_{\partial\cO}=0 \ \text{ for all }\ t\in (0,\sigma). 
$$
In particular, $\|v\|_{\BD^{\beta}_{q,p}(\Dom)}\eqsim \|v\|_{B^{\beta}_{q,p}(\Dom)}$ for any $\beta\in [0,1)\setminus \{1/q\}$ and $q,p\in (1,\infty)$. Moreover, for $\beta\in (0,1)$ and $q,p$ as before, \eqref{eq:duality_HD_function_spaces} and \eqref{eq:def_BD_function_spaces} imply that $\|v\|_{\BD^{-\beta}_{q,p}(\cO)}\lesssim \|v\|_{B^{-\beta}_{q,p}(\cO)}$ whenever the right-hand side is finite. For the case \eqref{it:eq: blow up 1}, it suffices to use the same observation in the application of Theorem \ref{thm: regularization} below. 
Hence, it is sufficient to prove the claims in Theorem \ref{thm: blow up} with the Dirichlet spaces.

From Theorems \ref{thm: blow up subcritical} and \ref{thm: blow up critical 2}, it suffices to consider the case where
$(s_0,q_0,p_0,\kappa_0,\zeta,\rho_{0},\rho_0/2,\rho_{0,3})$ are different from the one used to construct the maximal $(s,q,p,\kappa)$-solution $(u,\sigma)$.
By the assumptions in Corollaries \ref{cor: critical local rho2} or  \ref{cor: critical localrho1}, it holds that
\begin{align}
\label{eq:definition_kappa_0_blow_up_criteria}
    \kappa_0=\frac{p_0}{2}\Big(\frac{\rho_0+2}{\rho_0}-\frac{d}{q_0}-s_0\Big)-1.
\end{align}
Correspondingly, we define
\begin{align}
        \gamma_0:= \frac{1+\kappa_0}{p_0}, 
        \quad\text{so that}\quad
        \beta_0:=1-s_0-2\gamma_0 = \frac{d}{q_0}-\frac{2}{\rho_0}.
    \end{align}
    It is easy to see that if the  parameters 
$(s_0,q_0,p_0,\kappa_0,\zeta,\rho_{0},\rho_0/2,\rho_{0,3})$ satisfy the conditions of Theorem \ref{thm:mainlocal}, there exist $q_1>q_0$ and $\kappa_1>\kappa_0$ sufficiently close to $q_0$ and $\kappa_0$, such that the parameters $(s_0,q_1,p_0,\kappa_1,\zeta,\rho_{0},\rho_0/2,\rho_{0,3})$ still satisfy the conditions of Theorem \ref{thm:mainlocal}
and such that 
\begin{align}
\label{eq: def kappa 1}
    \kappa_1<\frac{p_0}{2}\Big(\frac{\rho_0+2}{\rho_0}-\frac{d}{q_1}-s_0\Big)-1.
\end{align}
We fix such values $q_1$ and $\kappa_1$ and see that thanks to \eqref{eq: def kappa 1}, $\BD_{q_1,p_0}^{\beta_1}(\cO)$ is not critical for \eqref{eq:reaction_diffusion_system}.
Setting
\begin{align}
    \beta_1:= 1-s_0-2\frac{1+\kappa_1}{p_0},
\end{align}
we see that  $\beta_1<\beta_0$. 
Let $r>0$ be such that $\Gamma_r:=\{r<\sigma\}$ has positive probability
and consider the equation, for $i=1,\dots,\ell$,
    \begin{equation}
\label{eq: blow up v r spde}
\left\{
\begin{aligned}
\dd v_{i} -\dv(a_i\cdot\nabla v_{i}) \,\dd t& = \1_{\Gamma_r}\Big[\dv(F_i(\cdot, v)) +f_i(\cdot, v)\Big]\,\dd t +  \sum_{j=1}^\ell \1_{\Gamma_r}g_{i}^j(\cdot,v) \,\dd W^j(t),\\
v_{i}(r)&=\1_{\Gamma_r}u_i(r).
\end{aligned}\right.
\end{equation}
By Theorem \ref{thm: regularization} we know that almost surely
\begin{align}
    u\in C_{\loc}([r,\sigma),\CD^{1-\tfrac{d}{2}+\tfrac{d}{\zeta}-}(\overline{\cO})),
    \quad\text{and hence a.s.}\quad
     \one_{\Gamma_r} u(r)\in\BD^{\beta_1}_{q_1,p_0}(\cO),
\end{align}
since  
$\beta_1<\beta_0<1-s_0<1-\tfrac{d}{2}+\tfrac{d}{\zeta}$ as $s_0>\frac{d}{2}-\frac{d}{\zeta}$, 
we have 
$$
\CD^{1-\tfrac{d}{2}+\tfrac{d}{\zeta}-}(\overline{\cO})\hookrightarrow \BD^{\beta_0}_{q_1,p_0}(\cO)\hookrightarrow \BD^{\beta_1}_{q_1,p_0}(\cO).
$$
By Theorem \ref{thm:mainlocal} there hence exists a maximal $(s_0,q_1,p_0,\kappa_1)$-solution $(v,\tau)$ to \eqref{eq: blow up v r spde}, which by Theorem \ref{thm: blow up subcritical} satisfies
\begin{align}
    \P\Big(r<\tau<\infty,\, \sup_{t\in [r,\tau)}\|v(t)\|_{\BD^{\beta_1}_{q_1,p_0}(\cO)}<\infty \Big) = 0,
\end{align}
since the space $\BD^{\beta_1}_{q_1,p_0}(\cO)$ is not critical for \eqref{eq:reaction_diffusion_system}.
Finally, recalling that $\BD^{\beta_0}_{q_1,\infty}(\cO)\hookrightarrow \BD^{\beta_1}_{q_1,p_0}(\cO)$ and $v\in \HD^{1-s_0,q_1}(\cO)$ a.e.\ on $(0,\tau)\times \O$, the previous equation implies
\begin{align}
    \P\Big(r<\tau<\infty,\, \sup_{t\in [r,\tau)}\|v(t)\|_{\BD^{\beta_0}_{q_1,\infty}(\cO)}<\infty \Big) = 0.
\end{align}
Finally it suffices to recall that by Proposition \ref{prop: compatibility} we have
$(v,\tau) = (u,\sigma)$ on $\Gamma_r$ almost surely on $\Gamma_r\times [r,\sigma)$, and that $r>0$ was arbitrary. 
Since $\cO$ is bounded, we could also replace $q_1$ by a larger number.
This proves \eqref{it:eq: blow up 1}.

\smallskip

To see that \eqref{it:eq: blow up 2} holds, one can argue similarly. As above, by Theorem \ref{thm: regularization}, the SPDE \eqref{eq: blow up v r spde} has a $(s_0,q_0,p_0,\a_0)$-solution (see \eqref{eq:definition_kappa_0_blow_up_criteria}). By Theorem \ref{thm: blow up critical 2}, it holds that 
\begin{align}
    \P\Big(r<\tau<\infty,\, \sup_{t\in [r,\tau)}\|v(t)\|_{\BD^{\beta_0}_{q_0,p_0}(\cO)} + \|v\|_{L^{p_0}(r,\tau;H^{\nu_0,q_0}(\cO))}<\infty \Big) = 0.
\end{align}
where we recall that
 $X_{1-\tfrac{\kappa_0}{p_0}} = \HD^{1-s_0-2\tfrac{\kappa_0}{p_0},q_0}(\cO)$ and
\begin{align}
    1-s_0-2\frac{\kappa_0}{p_0} = \frac{d}{q_0}-\frac{2}{\rho_0} + \frac{2}{p_0}.
\end{align}
Since $(u,\sigma) = (v,\tau)$ on $\Gamma_r$ and $v\in \HD^{1-s_0,q_0}(\cO)$ a.e.\ on $(0,\tau)\times \O$, we then obtain \eqref{it:eq: blow up 2}.
\end{proof}

\begin{proof}[Proof of Corollary \ref{cor: blow up}]
We begin by proving \eqref{it:cor_blow_up1}.
    By Theorem \ref{thm: blow up} we have that there exists $q^*>q_0$ such that for any $q_1\in (q_0,q^*)$ we have
    \begin{align}
        \P\Big(r<\sigma<\infty,\,\sup_{t\in [r,\sigma)}\|u(t)\|_{B_{q_1,\infty}^0(\cO)}<\infty \Big)=0,
    \end{align}
    using that $\tfrac{d}{q_0}-\tfrac{2}{\rho_0} = 0$. Then it suffices to recall that for $m>q_1>q_0$ we have $L^m(\cO)\hookrightarrow L^{q_1}(\cO)\hookrightarrow \BD_{q_1,\infty}^0(\cO)$ and that
    \begin{align}
        \Big\{ r<\sigma<\infty,\,\sup_{t\in [r,\sigma)}\|u(t)\|_{L^m(\cO)}<\infty \Big\}
        \subseteq
        \Big\{ r<\sigma<\infty,\,\sup_{t\in [r,\sigma)}\|u(t)\|_{B_{q_1,\infty}^0(\cO)}<\infty \Big\}.
    \end{align}
    Next we show \eqref{it:cor_blow_up2}. We recall that by Theorem \ref{thm: blow up} we have 
    \begin{align}
    \label{eq:cor_blow_up_2_proof}
        \P\Big(r<\sigma<\infty,\, \sup_{t\in [r,\sigma)}\|u(t)\|_{\BD^{\beta_0}_{q_0,p_0}(\cO)} + \|u\|_{L^{p_0}(r,\sigma;L^{q_0}(\cO))}<\infty \Big) = 0,
    \end{align}
    where we used that by assumption $\nu_0 = \tfrac{2}{p_0}+\tfrac{d}{q_0}-\tfrac{2}{\rho_0} = 0$. Next, due to $m_0 = \tfrac{d\rho_0}{2}$ and $q_0>m_0$ we have $\beta_0 = \tfrac{d}{q_0}-\tfrac{2}{\rho_0}<0$ and hence,
$L^{m_0}(\cO)\hookrightarrow \BD^{\beta_0}_{q_0,q_0}(\cO)\hookrightarrow \BD^{\beta_0}_{q_0,p_0}(\cO)$, where for the second embedding we also used $p_0\geq q_0$. Then it suffices to observe that
    \begin{align}
        \Big\{ r<\sigma<\infty,\,\sup_{t\in [r,\sigma)}\|u(t)\|_{L^{m_0}(\cO)} &+ \|u\|_{L^{p_0}(r,\sigma;L^{q_0}(\cO))}<\infty \Big\}\\
        &\subseteq
        \Big\{ r<\sigma<\infty,\,\sup_{t\in [r,\sigma)}\|u(t)\|_{\BD_{q_0,p_0}^{\beta_0}(\cO)} + \|u\|_{L^{p_0}(r,\sigma;L^{q_0}(\cO))}<\infty \Big\} .
    \end{align}
    The above and \eqref{eq:cor_blow_up_2_proof} yield \eqref{it:cor_blow_up2}.
\end{proof}

\subsection{Additional local well-posedness results for H\"older continuous initial data}\label{ss:AddlocalHolder}

For H\"older continuous initial data it is much easier to obtain local well-posedness. At the same time this  immediately gives H\"older continuous solutions. One actually only needs the nonlinearities to be locally Lipschitz. These results are much simpler than Theorem \ref{thm:mainlocal}, and its corollaries on critical cases. However, in the proof of the positivity result of Theorem \ref{thm:positity_nonlinear_eq} we need exactly this. 

\begin{proposition}[Local existence and uniqueness for H\"older continuous  solutions]
\label{prop:mainlocal_bounded}
Let $s \in [0, 1)$, $q \in [2, \infty)$, $p \in (2, \infty)$, and $\kappa \in [0, p/2 - 1)$. 
Suppose that Assumption \ref{ass:reaction_diffusion_global} and \ref{ass: zeta f} hold for $\zeta \in [2, \infty]$.
Suppose that the parameters satisfy
\begin{equation}
  \frac{d}{2} - \frac{d}{\zeta}<s \ \ \ \text{and} \ \ \  s + 2\frac{1+\kappa}{p} + \frac{d}{q} < 1,
\end{equation}
where we allow $s=0$ if $\zeta=2$. 
Then, for every initial condition $u_0 \in L^0_{\cF_0}(\Omega; \BD_{q,p}^{1-s-2\frac{1+\kappa}{p}}(\cO))$, there exists a  maximal $(s,q,p,\kappa)$-solution $(u, \sigma)$ to \eqref{eq:reaction_diffusion_system} such that $\sigma > 0$ a.s. 

Moreover, the following assertions hold: 
\begin{enumerate}[\rm (1)]
\item for each localizing sequence $(\sigma_n)_{n \ge 1}$ for $(u, \sigma)$, one has for all $n \ge 1$ and all $\theta \in [0, 1/2)$
    \[
        u \in H^{\theta,p}(0, \sigma_n, w_\kappa; \HD^{1-s-2\theta,q}(\cO)) \cap C([0, \sigma_n]; \BD^{1-s-2\frac{1+\kappa}{p}}_{q,p}(\cO)) \quad \text{a.s.}
    \]
\item The following blow-up criterion holds:
    \[\P\Big(\sigma<\infty,\, \sup_{t\in [0,\sigma)}\|u(t)\|_{\BD^{1-s-2\frac{1+\kappa}{p}}_{q,p}(\cO)}<\infty \Big) = 0.\]
\item The solution regularizes in the same way as in Theorem \ref{thm: regularization}.
\item\label{it:prop:mainlocal_bounded4} The compatibility of Proposition \ref{prop: compatibility} extends to the above values of $s$. 
\end{enumerate}
\end{proposition}

\begin{remark}\label{rem:locallyLips}
In Proposition \ref{prop:mainlocal_bounded} it suffices to assume that the nonlinearities are locally Lipschitz. Indeed, this follows from \cite[Theorem 4.7]{AVsurvey} and the fact that $\BD^{1-s-2\frac{1+\kappa}{p}}_{q,p}(\cO)\hookrightarrow C(\overline{\cO})$ by Sobolev embedding. 
\end{remark}

\begin{proof}[Proof of Proposition \ref{prop:mainlocal_bounded}]
To prove the result we argue as in Theorem \ref{thm:mainlocal}. However, this time we can apply \cite[Theorem 4.7]{AVsurvey} with any $\beta_j>1-\frac{1+\kappa}{p}$ and growth choices arbitrarily close to zero, and $X_0 = \HD^{-1-s,q}(\cO)$. Note that $\BD_{q,p}^{1-s-2\frac{1+\kappa}{p}}(\cO)$ embeds into $C(\overline{\cO})$ (and even into $C^{\theta}(\overline{\cO})$) by Sobolev embedding. The reader is referred to  \cite[Lemma 4.3]{AV20_NS} for a similar calculation. As before one can show that the solution regularizes instantaneously. Arguing as in Proposition \ref{prop: compatibility} one sees that the solution coincides with the one obtained in Theorem \ref{thm:mainlocal}. 
\end{proof}

\subsection{Proof of the positivity of Theorem \ref{thm:positity_nonlinear_eq}}
In light of Theorems \ref{thm:mainlocal} and  \ref{thm: regularization}, the proof of Theorem \ref{thm:positity_nonlinear_eq}
follows from a well known linearization procedure in the context of reaction-diffusion equations, see \cite[Theorem 2.13]{agresti2023reaction}.
To this end, we begin by discussing positivity results for linear SPDEs on $\Dom$: 
\begin{equation}
\label{eq:SPDE_linear}
\left\{
\begin{aligned}
&\dd v =\big[\dv (a\cdot \nabla v + b v)+ c v +\phi\big]\,\dd t  +\sum_{j=1}^\ell \chi_j v\,\dd W^j,\\
& v(0)=v_0,\qquad v|_{\partial\Dom}=0.
\end{aligned}
\right.
\end{equation}
where $\ell\geq 1$ is an integer.
As in Remark \ref{r:neumann}, the results of the current section also extend to the case of Neumann boundary conditions. 
In the previous, $W^j=R_{\mu^j} B^j$, where $R_{\mu^j}$ and $B^j$ are as in Subsection \ref{sss:noise}. Thus, $(B^j)_{j=1}^\ell$ are $L^2(\cO)$-cylindrical Brownian motions and 
\begin{equation}
\label{eq:Rmu_positivity}
R_{\mu^j}= \sum_{n=1}^\infty \mu_n^j\, e_n^j \otimes \varphi_n^j,
\end{equation}
where $(e_n^j)_{n\geq 1}$, $(\varphi_n^j)_{n\geq 1}$ and $\mu^j$ are as in Assumption \ref{ass: zeta f}. In particular 
\begin{equation}
\label{eq:assumption_mu_positivity}
\mu^j\in \ell^\zeta(\Sf^j) \ \text{ where \ $\zeta\in [2,\infty]$ \  is such that } \
1-\frac{d}{2}+\frac{d}{\zeta}\in (0,1],
\end{equation}
cf.\ \eqref{def:ellzeta}. The coefficients $(a,b,c,\phi,\chi_j)$ satisfy the following assumption. Below, we use the function spaces with Dirichlet boundary conditions as introduced in Subsection \ref{ss:function_spaces_Dir}.

\begin{assumption}
\label{ass:positivity_linear}
There exist positive constants $M,\nu,\alpha$ for which the following assertions hold for all $j\in \{1,\dots,\ell\}$ and $i,k\in \{1,\dots,d\}$.
\begin{itemize}
\item {\rm (Regularity of the coefficients)} $a: \R_+\times \O\times \cO\to \R^{d\times d}$, $b: \R_+\times \O\times \cO\to \R^{d}$, $c,\chi_j: \R_+\times \O\times \cO\to \R$ are $\Progress\otimes \Borel(\cO)$-measurable and a.s.\ for all $t\in [0,\infty)$, 
\[\|a^{i,k}(t,\cdot)\|_{C^1(\overline{\mathcal{O}})}  \leq M, \ \ |b(t,\cdot)|\leq M, \ \ |c(t,\cdot)|\leq M, \ \ |\chi_j(t,\cdot)|\leq M. \]
   \item {\rm (Parabolicity)} a.e.\ on $\R_+\times \O$ and all $\xi\in \R^d$,
   \begin{align*}   
    \sum_{i,j=1}^d a^{i,j}\xi_i \xi_j \geq \nu |\xi|^2.
   \end{align*}
   \item $\phi$ and $v_0$ are $\Progress\otimes \Borel(\cO)$-measurable and $\cF_0$-measurable, respectively.
   \item It holds that $v_0\in L^0_{\cF_0}(\Omega,\CD^\alpha(\overline{\cO}))$, and 
   $\phi \in L^\infty_{\loc}([0,\infty)\times \cO)$ is $\Progress\times \Borel(\cO)$-measurable.
\end{itemize}
\end{assumption}

Clearly, $(s,q,p,\a)$-solutions to the SPDE \eqref{eq:SPDE_linear} can be defined as in Definitions \ref{def:strong_sqpkappa_solution} and \ref{def:maximal_sqpkappa_solution}. 
Below, we need the following consequence of Theorems \ref{thm:MgTmu delta} and \ref{thm:LM-interpolation-fractional-dirichlet}.

\begin{lemma}
\label{l:well_posedness_linear}
Let $s \in [0, 1)$, $q \in [2, \infty)$, $p \in (2, \infty)$, and $\kappa \in [0, p/2 - 1)$.
Let Assumption \ref{ass:positivity_linear} be satisfied. Suppose that \eqref{eq:assumption_mu_positivity} holds and 
\begin{equation}
\label{eq:choice_sq_linear}
s>\frac{d}{2}-\frac{d}{\zeta},\qquad 
1-s-2\frac{1+\a}{p}<\alpha,\quad  \text{ and }\quad 
1-s-2\frac{1+\a}{p}>\frac{d}{q}.
\end{equation}
Then there exists a maximal $(s,q,p,\a)$-solution $(v,\tau)$ to \eqref{eq:SPDE_linear} that is global in time (i.e.\ $\tau=\infty$) and satisfies
\begin{align}
\label{eq:pathwise_regularity_linear_SPDE_positivity}
v\in  L^p_{\loc}([0,\infty),w_{\kappa};\HD^{1-s,q}(\cO))\cap C([0,\infty)\times \overline{\cO}) \ \text{ a.s. }
\end{align}
Moreover, for all $T<\infty$,
\begin{align}
\label{eq:pathwise_regularity_linear_SPDE_positivity_estimate}
\E \Big[ \sup_{t\in [0,T]}\|v(t) \|_{C(\overline{\cO})}^p\Big]
+ \E\int_0^T \|v\|_{\HD^{1-s,q}(\cO)}^p\,t^\a \,\dd t 
&\lesssim \E\|v_0\|_{C^\alpha(\cO)}^p \\
\nonumber
&\,+\E\|\phi\|_{L^p(0,T,w_\a;L^q(\cO))}^p,
\end{align}
whenever the right-hand side of the above is finite, and the implicit constant depends on $\mu^j$ only through its $\ell^\zeta(\Sf^j)$-norm.
\end{lemma}

As $\alpha>0$ and $s<1$, it is always possible to choose $(q,p,\a)$ for which the second and third conditions in \eqref{eq:choice_sq_linear} are satisfied.
Finally, let us note that \eqref{eq:pathwise_regularity_linear_SPDE_positivity_estimate} is not optimal (see \cite[Subsection 1.1]{AGV}), but it is enough for the purposes of Theorem \ref{thm:positity_nonlinear_eq} due to the instantaneous regularization of maximal $(s,q,p,\a)$-solutions provided by Theorem \ref{thm: regularization}.

\begin{proof}
For simplicity, we consider the case $\ell=1$, the other case follows analogously. Thus, we write $\chi=\chi_1$ and $\mu=\mu^1$.
We begin by observing that, from Theorem  \ref{thm:LM-interpolation-fractional-dirichlet}, the statement of Lemma \ref{l:well_posedness_linear} immediately follows in the case $b=c=\chi=0$. Here to ensure stochastic maximal $L^p$-regularity, we used the text below Assumption \ref{ass:FGcritical}, where the reference operator $A_0$ is as in the beginning of the proof of Theorem \ref{thm:mainlocal}.

To allow for non-trivial $(b,c,\chi)$ we apply the perturbation result of \cite[Theorem 3.2]{AVtorus} with $X_1=\HD^{1-s,q}(\cO)$ and $X_0=\HD^{-1-s,q}(\cO)$. Indeed, as it follows from Subsection \ref{ss:function_spaces_Dir}, one has 
$$
X_{\theta}=[X_0,X_1]_{\theta}=\HD^{-1-s+2\theta,q}(\cO) \ \ \text{ for }\theta\in (0,1).
$$ 
Hence, $X_{1/2}=\HD^{-s,q}(\cO)$. Moreover, for all $v\in X_1$, it follows from the boundedness of the coefficients in Assumption \ref{ass:positivity_linear} that
$$
\|\dv (b v )+ c\, v \|_{X_{0}}
\lesssim \|v\|_{L^q(\cO)},
$$
and 
$$
\|\chi v\, R_\mu\|_{\g(L^2(\cO),X_{1/2})}
\lesssim\|\mu\|_{\ell^\zeta(\Sf)} \|v\|_{L^\eta(\cO)}\lesssim \|\mu\|_{\ell^\zeta(\Sf)}\|v\|_{L^q(\cO)},
$$
where in the above we applied Theorem \ref{thm:MgTmu delta} with $\eta$ close to $q$ (this choice is allowed due to the first in \eqref{eq:choice_sq_linear}). As $X_{(1+s)/2}=L^q(\cO)$ and $s<1$, it follows that the operators $v\mapsto \dv (b v )+ cv$ and $v\mapsto \chi v \, R_\mu$ are lower-order. The claimed estimates \eqref{eq:pathwise_regularity_linear_SPDE_positivity_estimate} in the case $v_0= 0$ are a consequence of \cite[Theorem 3.2]{AVtorus} and 
$$
(X_0,X_1)_{1-\frac{1+\a}{p},p}= \BD^{1-s-2\frac{1+\a}{p}}_{q,p}(\cO)\hookrightarrow C(\overline{\cO}),
$$
where in the last step we used the third inequality in \eqref{eq:choice_sq_linear}.

In case $v_0\neq 0$, we can apply \cite[Proposition 3.11]{AVsurvey} and use $\CD^\alpha(\cO)\hookrightarrow \BD^{1-s-2\frac{1+\a}{p}}_{q,p}(\cO) $ by the second condition in \eqref{eq:choice_sq_linear}.
\end{proof}

The main ingredient in the proof of Theorem \ref{thm:positity_nonlinear_eq} is the following result.

\begin{lemma}[Positivity for linear SPDEs with rough noise]
\label{l:positivity_linear_SPDEs}
Suppose that Assumption \ref{ass:positivity_linear} holds and that 
$$  v_0\geq 0 \ \text{ a.e.\ on } \  \O\times \Dom, \quad \text{ and }\quad
\phi\geq 0 \ \text{ a.e.\ on } \ \R_+\times \O\times \Dom.
$$
Let $s \in [0, 1)$, $q \in [2, \infty)$, $p \in (2, \infty)$, and $\kappa \in [0, p/2 - 1)$ be such that \eqref{eq:choice_sq_linear} holds.
Then the maximal $(s,q,p,\a)$-solution $v$ to \eqref{eq:SPDE_linear} satisfies
\begin{equation}
\label{eq:positivity_tested_linear}
 v \geq 0  \ \text{ a.e.\ on } \ \R_+\times \O\times \Dom.
\end{equation}
\end{lemma}

For convenience of the reader, we include a direct proof of this standard result.

\begin{proof}
For simplicity, we consider the case $\ell=1$, the other follows analogously. Thus, we write $\chi=\chi_1$ and $\mu=\mu^1$.
Arguing as in \cite[Lemma A.1]{agresti2023reaction}, the claim follows by combining the maximum principle in \cite[Theorem 4.3]{Krylov13} and an approximation argument.
By replacing $\phi$ by $\one_{[0,\tau_n]} \phi$ where
$$
\tau_n =\inf\{t\in [0,\infty)\,:\, \|\phi\|_{L^p(0,t,w_\a;L^q(\cO))}\geq n\} \quad \text{  where }\quad \inf\emptyset:=\infty,
$$
and $v_0$ by $\one_{\{\|v_0\|_{C^\alpha(\cO)}\leq n\}} v_0$ for some $n<\infty$, it suffices to consider $\phi$ and $v_0$ with finite moments of any order. Assumption \ref{ass:positivity_linear} ensures that $\tau_n$ is a stopping time for all $n\geq 1$. From \eqref{eq:pathwise_regularity_linear_SPDE_positivity_estimate} in Lemma \ref{l:well_posedness_linear} there exists a constant $C_0>0$ such that 
\begin{equation}
\label{eq:a_priori_bounds_positivity_proof}
\E\Big[\sup_{t\in [0,T]}\|v(t)\|_{C(\overline{\cO})}^p\Big] \leq  C_0,
\end{equation}
where $(s,q,p,\a)$ are as in \eqref{eq:choice_sq_linear}.
Fix $N\geq 1$ and $T\in (0,\infty)$. Let 
$\mu^{(N)}_k:=\mu_k$ for $k\leq N$ and $\mu^{(N)}_k=0$ otherwise. 
Clearly, 
$
\mu^{(N)}\in \ell^2(\Sf)
$.
Hence, from standard regularity theory of SPDEs (see e.g., \cite[Chapter 4]{LR15}), there exists a unique progressively measurable process 
$$
v^{(N)}\in L^2((0,T)\times \O;H^1_0(\Dom))\cap L^2(\O;C([0,T];L^2(\Dom)))
$$ 
satisfying the SPDE:
\begin{equation}
\label{eq:SPDE_linear_N}
\left\{
\begin{aligned}
&\dd v^{(N)} =\big[\dv (a\cdot \nabla v^{(N)} + b v^{(N)})+ c v^{(N)} +\phi \big]\,\dd t  +  \chi v^{(N)} R_{\mu^{(N)}}\,\dd B,\\
&v^{(N)}(0)=v_0,\qquad v^{(N)}|_{\partial\Dom}=0.
\end{aligned}
\right.
\end{equation}
Since $\mu^{(N)}\in \ell^2(\Sf)$, well known positivity results for linear SPDE with trace class noise (see e.g., \cite[Theorem 4.3]{Krylov13}) imply that \eqref{eq:positivity_tested_linear} holds with $v$ replaced by $v^{(N)}$.
Note that 
$$
w^{(N)}:= v-v^{(N)}
$$ 
solves the linear SPDE:
\begin{equation}
\label{eq:SPDE_linear_w_difference}
\left\{
\begin{aligned}
&\dd w^{(N)} =\big[ \dv (a\cdot \nabla w^{(N)} + b w^{(N)})+ c w^{(N)}  \big]\,\dd t+ \big({\chi w^{(N)}} R_{\mu^{(N)}} + \chi v R_{\mu-\mu^{(N)}}\big) \,\dd B,\\
& w^{(N)}(0)=0,\qquad w^{(N)}|_{\partial\Dom}=0,
\end{aligned}
\right.
\end{equation}
Arguing as in the proof of Lemma
\ref{l:well_posedness_linear}, using
\cite[Theorem 3.2]{AVtorus} with inhomogeneous
stochastic forcing, we obtain a constant
$C>0$, independent of $N$, such that 
\begin{align}
\label{eq:positivity_main_estimate_proof_difference}
\E \big[\sup_{t\in [0,T]}\|w^{(N)}(t)\|_{C(\overline{\cO})}^p\big]  &\leq C \,\E\int_0^T \|\chi v R_{\mu-\mu^{(N)}}\|_{\g(L^2(\cO),\HD^{-s,q}(\cO))}^p \, t^\a \,\dd t.
\end{align}
If $\zeta<\infty$, then from Theorem \ref{thm:MgTmu delta}\eqref{it1:MgTmu delta}-\eqref{it2:MgTmu delta}, it follows from \eqref{eq:a_priori_bounds_positivity_proof} and the choice of $(\mu^{(N)})_{N\geq 1}$ that 
$$
\E \big[\sup_{t\in [0,T]}\|w^{(N)}(t)\|_{C(\overline{\cO})}^p\big] 
\leq C \,\|\mu-\mu^{(N)}\|_{\ell^\zeta(\Sf)}^p \, \E \int_0^T \|\chi v\|_{L^\infty(\cO)}^p \,\dd t \stackrel{N\to \infty}{\to} 0,
$$
In the case $\zeta=\infty$, we need to argue differently. Here, we employ the convergence result for $\g$-radonifying operators in \cite[Theorem 9.1.14]{Analysis2}. 
Let $P_N $ be the orthogonal projection onto $\mathrm{span} (e_1,\dots,e_N)$, where $(e_n)_{n\geq 1}$ is the system appearing in \eqref{eq:Rmu_positivity} with $\ell=1$. Note that 
$$
M_{\chi v } R_{\mu-\mu^{(N)}}
= M_{\chi v } R_{\mu}(\mathrm{Id}-P_N). 
$$
where $M_{\chi v }$ denotes the multiplication by $\chi v$.
Hence, from Theorem \ref{thm:MgTmu delta}\eqref{it3:MgTmu delta} and \cite[Theorem 9.1.14]{Analysis2}, it follows that 
$$
M_{\chi v } R_{\mu-\mu^{(N)}}\stackrel{N\to \infty}{\to} 0 \ \text{ in }\ \g(L^2(\cO),\HD^{-s,q}(\cO)) \ \text{ a.e.\ on }(0,T)\times \O.
$$
Moreover, from the ideal property of $\g$-radonifying operators \cite[Theorem 9.1.10]{Analysis2}, we have
$$
\|
M_{\chi v } R_{\mu-\mu^{(N)}}\|_{\g(L^2(\cO),\HD^{-s,q}(\cO))}
\leq \|
M_{\chi v } R_{\mu}\|_{\g(L^2(\cO),\HD^{-s,q}(\cO))},
$$
as $\|I-P_N\|_{\calL(L^2(\cO))}\leq 1$, and by combining the Lebesgue domination theorem, Theorem \ref{thm:MgTmu delta}\eqref{it3:MgTmu delta} and \eqref{eq:a_priori_bounds_positivity_proof}, we infer
$$
\E \big[\sup_{t\in [0,T]}\|w^{(N)}(t)\|_{C(\overline{\cO})}^p\big] \leq 
C \,\E\int_0^T \|\chi v R_{\mu-\mu^{(N)}}\|_{\g(L^2(\cO),\HD^{-s,q}(\cO))}^p \, t^\a \,\dd t\stackrel{N\to \infty}{\to} 0.
$$
Hence, \eqref{eq:positivity_tested_linear} follows from the above and $v^{(N)}\geq 0$ a.e.\ on $\R_+\times \O\times \cO$ as proven below \eqref{eq:SPDE_linear_N}.
\end{proof}

With the above result at our disposal, we can prove Theorem \ref{thm:positity_nonlinear_eq}.

\begin{proof}[Proof of Theorem \ref{thm:positity_nonlinear_eq}]
One can follow the same method as in  \cite[Theorem 2.13]{agresti2023reaction}, which is based on a linearization procedure. We point out the main differences only. 

\smallskip

\emph{Step 1: (Reduction) It suffices to consider $u_0\in L_{\cF_0}^\infty(\O; \CD^{\alpha}(\cO))$ for some $\alpha>0$ in Theorem \ref{thm:positity_nonlinear_eq}.}
In light of the instantaneous regularization result \eqref{eq: reg u result 2} in Theorem \ref{thm: regularization} and Proposition \ref{prop:local_continuity_abstract} with the choices made at the beginning of Subsection \ref{ss:local_wp_proof_critical}, 
the claim of Step 1 follows verbatim from Steps 1 and 2 of  \cite[Theorem 2.13]{agresti2023reaction}. Here we used the extension of the compatibility result stated in  Proposition \ref{prop:mainlocal_bounded}\eqref{it:prop:mainlocal_bounded4}. Moreover, to approximate $u_0$, we can use the heat semigroup to regularize while preserving positivity. 

\smallskip

\emph{Step 2: Smooth initial data.}
From Step 1, we assume that $u_0\in L_{\cF_0}^\infty(\O; \CD^{\alpha}(\overline{\cO}))$ for some $\alpha>0$. 
As in \cite[Theorem 2.13]{agresti2023reaction}, we consider the following variant of \eqref{eq:reaction_diffusion_system}:
\begin{equation}
\label{eq:reaction_diffusion_system_positivity}
\left\{
\begin{aligned}
\dd u_i^+ -\dv(a_i\cdot\nabla u_i^+) \,\dd t& = \Big[\dv(F_{i}^+(\cdot,u^+)) +f_{i}^+(\cdot, u^+)\Big]\,\dd t + \sum_{j=1}^{\ell} g_{i}^{j,+} (\cdot,u^+) \,\dd W^j(t),\\
u_i^+(0)&=u_{0,i},
\end{aligned}\right.
\end{equation}
where, for $y\in \R^\ell$,
\begin{equation}
\label{eq:nonlinearity_positivity_enforced}
f_{i}^+(\cdot, y)
=
f_{i}(\cdot, y\vee 0),
\quad 
F_{i}^+(\cdot, y)
=
F_{i}(\cdot, y\vee 0),  \ \ \text{ and }\ \ 
g_{i}^{j,+}(\cdot, y)
=
g_{i}^j(\cdot, y\vee 0),
\end{equation}
and $y\vee 0:=(y_i\vee 0)_{i=1}^\ell$.
The idea behind the use of the system of SPDEs in \eqref{eq:reaction_diffusion_system_positivity} is that, for the latter, we can use \eqref{eq:assumption_positivity} as the vector $y\vee 0$ has all non-negative entries. Note that the nonlinearities $(f_i^+,F_i^+,g_i^{j,+})$ defined in \eqref{eq:nonlinearity_positivity_enforced} satisfy Assumption \ref{ass:reaction_diffusion_global} with the same parameters. In particular, the results in Corollary \ref{cor: blow up} and Proposition \ref{prop:mainlocal_bounded} 
are also valid for the system of SPDEs \eqref{eq:reaction_diffusion_system_positivity}.
Let $(u^+,\sigma^+)$ be the maximal $(s_0,q_0,p_0,\a_0)$-solution to \eqref{eq:reaction_diffusion_system_positivity} given by Proposition \ref{prop:mainlocal_bounded} with $(s_0,q_0,p_0,\a_0)$ satisfying 
\begin{equation}
\label{eq:condition_zeros_proof_positivity}
\frac{d}{2} - \frac{d}{\zeta} < s_0,\ \ \ \text{and} \ \ \ \frac{d}{q_0}<1-s_0-2\frac{1+\a_0}{p_0}<\alpha.
\end{equation}
For instance, the above holds for $s_0$ close to $1$, $\kappa_0 =0$ and $p_0,q_0$ large. Then $u_0\in \CD^{\alpha}(\cO) \hookrightarrow \BD^{1-s_0-2\frac{1+\kappa_0}{p_0}}_{q_0,p_0}(\cO)$. By Sobolev embedding and the second condition in \eqref{eq:condition_zeros_proof_positivity}, it follows that $\BD^{1-s_0-2\frac{1+\a_0}{p_0}}_{q_0,p_0}(\cO)\hookrightarrow \CD(\overline{\cO})$ and therefore
\begin{equation*}
u^+\in C([0,\sigma^+);\CD(\overline{\cO})) \ \text{ a.s. }
\end{equation*}
Due to the pathwise regularity of $u^+$ in the previously displayed formula, we can follow Step 4 from the proof of \cite[Theorem 2.13]{agresti2023reaction} verbatim,  and we get
\begin{equation}
\label{eq:positivity_uplus}
u^+\geq 0\  \text{ a.e.\ on }\ [0,\sigma^+)\times \O\times \cO. 
\end{equation}
Here we use Lemmas \ref{l:well_posedness_linear} and \ref{l:positivity_linear_SPDEs} to get the well-posedness and positivity of the associated linearized problem appearing in Step 4 of \cite[Theorem 2.13]{agresti2023reaction}. 

Arguing as in the argument below \cite[eq.\ (3.49)]{agresti2023reaction}, the positivity of $u^+$ in \eqref{eq:positivity_uplus}, the blow-up criterion in Proposition \ref{prop:mainlocal_bounded}, and the regularity of $u$ in Theorem \ref{thm: regularization} yield
\begin{equation}
\label{eq:uniqueness_u_uplus}
\sigma^+ = \sigma \ \text{ a.s., }\quad \text{ and }\quad 
u=u^+ \ \text{ a.e.\ on }\ [0,\sigma)\times \O.
\end{equation}
The positivity of $u$ clearly follows by combining \eqref{eq:positivity_uplus} and 
\eqref{eq:uniqueness_u_uplus}.
\end{proof}

\textbf{AI disclosure.} 
All mathematical ideas, arguments, and proofs were developed entirely by the authors between 9 September 2025 and 5 June 2026, before the current generation of powerful AI models became available. After receiving the first referee report, the authors used ChatGPT 5.6 and ChatGPT 6 solely to check the manuscript for mistakes and typographical errors. No AI-generated content has been incorporated into the manuscript. The authors remain fully responsible for the content and correctness of the manuscript.

\bibliographystyle{alpha}
\bibliography{literature}
\end{document}